\documentclass[11pt, letterpaper]{amsart}
\usepackage[margin=1.5in]{geometry}

\usepackage[english]{babel}
\usepackage[utf8]{inputenc}
\usepackage[T1]{fontenc}
\usepackage{lmodern} 

\usepackage{amsfonts} 
\usepackage{amssymb}
\usepackage{amsmath}
\usepackage{amsthm}
\usepackage{verbatim}
\usepackage{enumerate}
\usepackage{mathtools}

\usepackage{tikz-cd}

\usepackage[colorlinks=true, allcolors=blue]{hyperref}

\newtheorem{theorem}{Theorem}[section]
\newtheorem{proposition}[theorem]{Proposition}
\newtheorem{corollary}[theorem]{Corollary}
\newtheorem{lemma}[theorem]{Lemma}

\newtheorem*{theorem-non}{Theorem}

\newcounter{theoremintro}
\newtheorem{thmintro}[theoremintro]{Theorem}
\newtheorem{corintro}[theoremintro]{Corollary}

\theoremstyle{definition}
\newtheorem{definition}[theorem]{Definition}
\newtheorem{example}[theorem]{Example}
\newtheorem{remark}[theorem]{Remark}
\newtheorem*{nota-non}{Notation}

\newcommand{\Z}{\mathbb{Z}}
\newcommand{\N}{\mathbb{N}}

\newcommand{\C}{\mathbb{C}}

\newcommand{\cG}{\mathcal{G}}

\newcommand{\norm}[1]{ \left\| #1 \right\| }

\newcommand{\setbuilder}[2] { \left\{ #1 \enskip \middle| \enskip #2 \right\} }

\newcommand{\abs}[1]{ \left| #1 \right| }

\newcommand{\closure}[1]{\overline{#1}}

\newcommand{\set}[1]{\left\{#1\right\}}

\newcommand{\directsum}{\oplus}
\newcommand{\isoto}{\cong}
\newcommand{\boundary}{\partial}
\newcommand{\defeq}{\coloneqq}

\newcommand{\what}[1]{\widehat{#1}}
\newcommand{\wtilde}[1]{\widetilde{#1}}

\newcommand{\actson}{\curvearrowright}
\newcommand{\injectsinto}{\hookrightarrow}
\newcommand{\surjectsonto}{\twoheadrightarrow}
\newcommand{\tensor}{\otimes}

\renewcommand{\Re}{\operatorname{Re}}

\DeclareMathOperator{\conv}{conv}

\DeclareMathOperator{\supp}{supp}

\DeclareMathOperator{\id}{id}

\DeclareMathOperator{\Ad}{Ad}
\DeclareMathOperator{\Aut}{Aut}

\DeclareMathOperator{\FC}{FC}
\DeclareMathOperator{\FCH}{FCH}
\DeclareMathOperator{\RP}{RP}
\DeclareMathOperator{\LP}{LP}

\title[Simplicity of crossed products by FCH groups]{Simplicity of crossed products by FC-hypercentral groups}
\author{Shirly Geffen}
\address{Mathematisches Institut, Fachbereich Mathematik und Informatik, Universit{\"a}t M{\"u}nster \\ Einsteinstrasse 62 \\ 48149 M{\"u}nster \\ Germany}
\email{sgeffen@uni-muenster.de}

\author{Dan Ursu}
\address{Mathematisches Institut, Fachbereich Mathematik und Informatik, Universit{\"a}t M{\"u}nster \\ Einsteinstrasse 62 \\ 48149 M{\"u}nster \\ Germany}
\email{dursu@uni-muenster.de}

\makeatletter
\@namedef{subjclassname@2020}{\textup{2020} Mathematics Subject Classification}
\makeatother

\thanks{The project was funded by the Deutsche Forschungsgemeinschaft (DFG, German Research Foundation) -- Project-ID 427320536 -- SFB 1442, as well as under Germany's Excellence Strategy EXC 2044 390685587, Mathematics M{\"u}nster: Dynamics--Geometry--Structure, and by the ERC Advanced Grant 834267 -- AMAREC. The first named author was also supported by the generosity of Eric and Wendy Schmidt by recommendation of the Schmidt Futures program.}

\begin{document}

    \begin{abstract}
		In this paper, we give a complete, two-way characterization, of when a noncommutative crossed product $A \rtimes_\lambda G$ is simple, in the case of $G$ being an FC-hypercentral group. This is a large class of amenable groups that contains all virtually nilpotent groups, and in the finitely-generated setting, coincides with the set of groups which have polynomial growth. We further completely characterize the ideal intersection property under the assumption that the group is FC, meaning that every element has a finite conjugacy class. Finally, for minimal actions of arbitrary discrete groups on unital C*-algebras, we are able to characterize when the crossed product $A \rtimes_\lambda G$ is prime.
	\end{abstract}
	
	\maketitle
	
	\tableofcontents

	\renewcommand*{\thetheoremintro}{\Alph{theoremintro}}
	\section{Introduction}
    \label{sec:introduction}
	The last several years have seen tremendous progress in problems related to C*-algebras generated by groups and their dynamical systems. In particular, the problem of determining when the reduced group C*-algebra $C^*_\lambda(G)$ is simple, which has its origin in Powers' work from 1975 \cite{Powers}, was completely solved by Breuillard, Kalantar, Kennedy, and Ozawa in \cite{kalantar_kennedy_boundaries}, \cite{breuillard_kalantar_kennedy_ozawa_c_simplicity}, and \cite{kennedy_intrinsic} (see also \cite{HaagerupCstarSimplicity}), using new refreshing ideas. This theory was shortly after generalized to the case of reduced crossed products $C(X) \rtimes_\lambda G$ involving commutative C*-algebras, by Kawabe in \cite{kawabe_crossed_products}. Somewhat more recently, the corresponding result for {\'e}tale groupoid C*-algebras was obtained in collaborated work of the second author \cite{kklru_groupoids}. This last result is still in essence a result about dynamics on commutative C*-algebras, with the C*-algebra $C_0(\cG^{(0)})$ (where $\cG^{(0)} \subseteq \cG$ is the unit space) admitting a partial action of the inverse semigroup of open bisections, in the appropriate sense.

    Noncommutative C*-dynamical systems $G \actson A$ and their corresponding crossed products $A \rtimes_\lambda G$ had always been more mysterious, and before going into more detail, we would like to briefly describe the context and scope of our paper here. Our flagship result, Theorem~\ref{thmIntro:B}, fully characterizes simplicity of the noncommutative crossed product $A \rtimes_\lambda G$ in the case of FC-hypercentral groups $G$, a very large class of amenable groups that includes, for example, all virtually nilpotent groups. Related results are Theorem~\ref{thmIntro:A}, dealing with a related property called the \emph{ideal intersection property} in the non-minimal setting, and Theorem~\ref{thmIntro:C} characterizing primality of the crossed product with no assumptions on the group whatsoever - only minimality of the action. Our results are the first true \emph{if and only if} characterizations in such a general setting, and present substantial progress in our understanding of the ideal structure of noncommutative crossed products. The main existing results before our work include the work of Elliott \cite{elliott_properly_outer}, Kishimoto \cite{Kishimoto_freely_acting}, Olesen and Pedersen \cite{olesen_pedersen_I, olesen_pedersen_II, olesen_pedersen_III}, and Archbold and Spielberg \cite{archbold_spielberg_amenable_crossed_products} that give the sufficient (but far from necessary) conditions of the action being \emph{properly outer} or \emph{topologically free}. Olesen and Pedersen also give conditions for simplicity/primality of the crossed product in the case of abelian groups, using a somewhat complicated tool called the \emph{Connes spectrum} of the action. Our characterizations are far easier and can be completely described directly in terms of the dynamics of $G$ on either $A$ itself or the injective envelope $I(A)$. Likewise, Rieffel \cite{rieffel_finite_groups} gives further characterizations in the case of finite groups, but these are also indirect and written in terms of certain subalgebras $C \subseteq A \rtimes_\lambda G$. Somewhat more distant (but extremely important) are the results of Echterhoff \cite{echterhoff_jot} which give conditions on prime ideals in $A \rtimes_\lambda G$ being maximal, particularly useful (especially for us!) in the case of $G$ being FC-hypercentral. Finally, there was the recent work of Kennedy and Schafhauser in \cite{kennedy_schafhauser_noncommutative_crossed_products}, which obtains intrinsic sufficient conditions that are weaker than proper outerness, but are still far from necessary in general, with the converse holding under a strong untwisting condition that they call \emph{vanishing obstruction}. This extra assumption is known to not hold even in certain actions of finite $G$ on finite-dimensional $A$.
	
	Now we go back into somewhat more detail. As mentioned previously, the problem of determining when a reduced crossed product $A \rtimes_\lambda G$, where $A$ is now noncommutative, is simple (or more generally, has the intersection property) was addressed already in the classical work of Archbold and Spielberg \cite{archbold_spielberg_amenable_crossed_products}. They show that $A \rtimes_\lambda G$ is simple under minimality and a (noncommutative) topological freeness assumption on the action. While minimality is necessary for simplicity, topological freeness is known to be only a sufficient condition. Similarly, if $G$ is a countable amenable group acting on a separable C*-algebra, such that the induced action of $G$ on the primitive ideal space of $A$ is free, it is known that $A\rtimes_\lambda G$ is simple if and only if the induced action on $\mathrm{Prim}(A)$ is minimal (see, for example, \cite[Theorem~8.20 and Theorem~8.21]{Williams}). However, such forms of freeness are rarely satisfied (see for instance Example~\ref{eg}).
	This leaves the question of a necessary and sufficient condition for simplicity in the noncommutative setting open and mysterious. Note that if $A$ is commutative and the group $G$ is amenable (or acts amenably on $A$), then it is shown in \cite[Theorem~4.1]{kawamura_tomiyama_crossed_products} (see also \cite[Theorem~2]{archbold_spielberg_amenable_crossed_products}) that $A \rtimes_\lambda G$ is simple if and only if the action is minimal and topologically free. 
	
	A seemingly weaker condition of noncommutative topological freeness, which nowadays is considered to be ``the right condition'', is called \emph{proper outerness} and it was introduced by Kishimoto in \cite{Kishimoto_freely_acting} (under the original name of \emph{freely acting}). A different version of this property was introduced a few years prior by Elliott in \cite{elliott_properly_outer}, and the two notions are known to coincide when the underlying C*-algebra is separable. The first general result on proper outerness (together with minimality of the action and separability of the base algebra) implying simplicity of a reduced crossed product can be found in a paper by Olesen and Pedersen \cite[Theorem~7.2]{olesen_pedersen_III} (see also \cite[Remark~2.23]{Sierakowski}), generalizing older results of Elliott and Kishimoto in special cases. 
	
	Using the powerful tools developed for the theory of C*-simplicity, Kennedy and Schafhauser \cite{kennedy_schafhauser_noncommutative_crossed_products} obtain sufficient conditions for the action that are weaker than proper outerness in general, but still imply simplicity in the minimal setting. Moreover, under an ``untwisting'' assumption which they call \emph{vanishing obstruction}, they also manage to obtain converse results. For amenable groups, everything again simplifies: if $G$ is amenable and the action $G \actson A$ has vanishing obstruction, then simplicity of $A \rtimes_\lambda G$ is equivalent to proper outerness and minimality of the action \cite[Corollary~9.7]{kennedy_schafhauser_noncommutative_crossed_products}. However, \emph{vanishing obstruction} is a strong condition, which can fail even in the setting of finite abelian groups acting on finite-dimensional C*-algebras. See \cite[Example~5.6]{kennedy_schafhauser_noncommutative_crossed_products} for a counterexample of the form $M_2 \rtimes_\lambda (\Z / 2 \Z \times \Z / 2 \Z)$ (also analyzed in our paper in Section~\ref{sec:examples:finite_dimensional} using our new tools).

	Also worth mentioning are the works of Olesen and Pedersen \cite{olesen_pedersen_I} and \cite{olesen_pedersen_II}, where they fully characterize when a crossed product is prime, or has the intersection property, in the case of locally compact abelian groups. Their characterizations are in terms of the Connes spectrum $\Gamma(\alpha)$ of the homomorphism $\alpha\colon G \to \Aut(A)$, which is a certain subgroup of the dual group $\what{G}$. The precise definition of the Connes spectrum, found in \cite[Section~1]{olesen_inner_automorphisms}, is a fairly complicated construction that involves considering all possible $G$-invariant hereditary C*-subalgebras $B \subseteq A$, and the $L^1(G)$-action on each of these algebras.
	
	In addition, there is also the work of Rieffel in \cite{rieffel_finite_groups}, where he considers characterizing simplicity and primality of crossed products in the case of finite groups, using quite similar notions of \emph{partly inner} and \emph{purely outer} automorphisms. However, these results are written in terms of a suitable subalgebra $C \subseteq A \rtimes_\lambda G$ that is in principle easier to understand, but more intrinsic characterizations in terms of the dynamical system $(A,G)$ are not given.

	We now begin describing the results in this article, and the main ideas surrounding them. First of all, we would like to mention that, as in all recent results concerning C*-simplicity and ideal structure for crossed products, we will heavily use the theory of injective envelopes developed by Hamana. More precisely, for a C*-dynamical system $(A,G)$, we will frequently consider the two additional induced dynamical systems: $(I(A),G)$ and $(I_G(A),G)$, where $I(A)$, the \emph{injective envelope} of $A$, denotes the minimal injective C*-algebra containing $A$, and $I_G(A)$ denotes its equivariant version. The existence of these objects was proven by Hamana in \cite{hamana79_injective_envelopes_cstaralg, hamana85-injective_envelopes_equivariant}. We refer the reader to Section~\ref{sec:preliminaries:injective_envelopes} for more details.

    We also mention that we obtain different results based on different end goals (obtaining the intersection property, obtaining simplicity, ...) and different assumptions (FC groups, FC-hypercentral groups, ...). At the end of the day, all of these results use the exact same underlying characterizations, which are cumbersome to write out each time. As such, we split off the characterizations here. We also warn the reader that these are characterizations for the \emph{negations} of the properties that we are interested, simply due to them being much easier to write out and work with this way.

    \begin{nota-non}
    \label{notation:characterizations}
        Let $G$ be a discrete group acting on a C*-algebra $A$ by *-automorphisms. Denote the set of elements of $G$ whose conjugacy class is finite by $\FC(G)$. We say that the action has characterization (C1) if there exist if there exists some $r \in \FC(G) \setminus \set{e}$, a non-zero $r$-invariant central projection $p \in I(A)$, and a unitary $u \in U(I(A)p)$ such that
		\begin{enumerate}
			\item $r$ acts by $\Ad u$ on $I(A)p$;
			
			\item $s \cdot p = p$ and $s \cdot u = u$ for all $s \in C_G(r) \defeq \setbuilder{g \in G}{gr = rg}$.
		\end{enumerate}
  
		We say that the action has characterization (C2) if there exists some $r \in \FC(G) \setminus \set{e}$, a non-zero $r$-invariant ideal $J \subseteq A$, and a unitary $u$ in the multiplier algebra $M(J)$ such that:
		\begin{enumerate}
			\item $J \cap s \cdot J$ is essential in both $J$ and $s \cdot J$, for all $s \in C_G(r)$. In particular, $M(J)$ and $M(s \cdot J)$ both canonically embed into $M(J \cap s \cdot J)$;
			
			\item Letting $\varepsilon_1 = \norm{\alpha_r|_J - (\Ad u)|_J}$ and $\varepsilon_2 = \sup_{s \in C_G(r)} \norm{s \cdot u - u}$, we have
			\[ 2 \sqrt{2 - \sqrt{4 - \varepsilon_1^2}} + \varepsilon_2 < \sqrt{2}. \]
		\end{enumerate}
    \end{nota-non}
	
	In Section~\ref{sec:intersection_property_fc}, we consider the case of FC-groups, which are groups where every conjugacy class is finite. Recall that a C*-dynamical system $(A,G)$ is said to have the \emph{intersection property} if any non-zero ideal of $A \rtimes_\lambda G$ has non-zero intersection with $A$. It is clear that for minimal actions, this is equivalent to simplicity of $A \rtimes_\lambda G$. Again, We write down all of our main results in terms of the negations of these properties, as this makes the statements of the characterizations less cumbersome.
	
	\begin{thmintro}[Theorem~\ref{thm:mainSec4} and Theorem~\ref{thm:equivariant_proper_outer_equivalence}, together with Proposition~\ref{prop:nonunital_intersection_property_transfer} and Proposition~\ref{prop:nonunital_characterizations_equivalent} for the non-unital case]
		\label{thmIntro:A}
		Let $G$ be an FC-group acting on a C*-algebra $A$, not necessarily unital. Then the crossed product $A \rtimes_\lambda G$ does not have the intersection property if and only if the action has characterization (C1) involving $I(A)$. If the C*-algebra $A$ is separable, then this is furthermore equivalent to the action having characterization (C2) involving $A$ itself.
	\end{thmintro}

	In certain special cases, we obtain significantly nicer characterizations. For example, in the setting of simple C*-algebras, Corollary~\ref{cor:ip_fc} can reduce condition (C1) from $I(A)$ to $M(A)$ directly. Likewise, in the case of $G = \Z$ and square-free $G = \Z/n\Z$, Corollary~\ref{cor:cyclic} shows that the intersection property turns out to be completely equivalent to the usual notion of proper outerness, even in the setting of non-separable $A$.
	
	Similar methods to what are used to prove Theorem~\ref{thmIntro:A} allow us to also characterize when a crossed product $A \rtimes_\lambda G$ is prime - see Theorem~\ref{thm:mainSec4Prime}. This generalizes a characterization by Hamana, which was obtained in the setting of finite groups, see \cite[Theorem~10.1]{hamana85-injective_envelopes_equivariant}.
	
	The first step in our argument is that if the crossed product $A \rtimes_\lambda G$ does not have the intersection property, then through the machinery of Section~\ref{sec:pseudoexpectations}, we are fairly easily able to obtain an element $u\in I_G(A)$ with the above requirements. However, in Theorem~\ref{thmIntro:A} we want $u$ to rather belong to $I(A)$. A general theme in all of the modern C*-simplicity results is to take any results on $I_G(A)$, which is an extremely mysterious and poorly-understood space, and try to push them down onto a more tractable space, such as $I(A)$ or $A$ itself. In the commutative case, this ends up being relatively straightforward, given that we have dual maps between the spectra of all of these spaces. In the noncommutative case, this is not nearly as straightforward, and we end up developing a far more roundabout argument.
	
	Finally, after we have our unitary $u \in U(I(A)p)$ as in the statement of Theorem~\ref{thmIntro:A}, we may instead convert this to a more tractable property involving $A$ instead. Observe that if we removed the statements ``$s \cdot p = p$ and $s \cdot u = u$ for all $s \in C_G(r)$'' from the requirements in the theorem, then the characterization would essentially be asking whether or not the action of $G$ on $I(A)$ is properly outer, which is well-known to coincide with an appropriate analogue on $A$ (with a little bit of work, this follows from \cite[Theorem~7.4]{hamana85-injective_envelopes_equivariant}). In the separable setting at least, the simplest characterization of proper outerness is that of Elliott \cite{elliott_properly_outer}. A brief overview of these classical notions is given in Section~\ref{sec:preliminaries:properly_outer}. The equivariant version of this characterization is developed in Section~\ref{sec:proper_outerness_on_A}. This turns out to be substantially more difficult than one would expect, and is the most involved part of our paper. Another glance at the second half of Theorem~\ref{thmIntro:A} will lead to the observation that there is no $C_G(r)$-invariance on any of the constituent components, but rather an ``almost invariance'' in all of the appropriate senses. The proof of the equivalence is also far from being a straightforward generalization of the classical case, and requires the application of new ideas and techniques along every step of the way.
	
	We upgrade these results to the setting of FC-hypercentral groups, at least in the context of minimal actions. A brief review of FC-hypercentral groups can be found in Section~\ref{sec:preliminaries:fc_hypercentral}. This is a much larger class of groups, which, in the finitely generated setting, coincides with the set of groups with polynomial growth.
	
	\begin{thmintro}[Theorem~\ref{thm:mainFCHSimpleMinimal} and Theorem~\ref{thm:equivariant_proper_outer_equivalence}]
		\label{thmIntro:B}
		Let $G$ be an FC-hypercentral group acting minimally on a unital C*-algebra $A$. Then $A \rtimes_\lambda G$ is not simple if and only if the action has characterization (C1) involving $I(A)$. If the C*-algebra $A$ is separable, then this is furthermore equivalent to the action having characterization (C2) involving $A$ itself.
	\end{thmintro}

        The proof of this theorem was originally given via different techniques, but it was kindly pointed out to us by Echterhoff that it follows from a combination of Theorem~\ref{thmIntro:C} below, together with one of his results. The proof of both can be found in Section~\ref{sec:primality_minimal}. Theorem~\ref{thmIntro:C} stems from the observation that most of the arguments presented in the proof of Theorem~\ref{thmIntro:A} can be modified to prove, in full generality, when a crossed product $A \rtimes_\lambda G$ is prime, as long as the action is minimal. Again, minimality is very crucial for a specific part of the argument, and we do not see an easy way to reduce it down to just assuming that $A$ is $G$-prime.
	
	\begin{thmintro}[Theorem~\ref{thm:DiscreteMinimal} and Theorem~\ref{thm:equivariant_proper_outer_equivalence}]
		\label{thmIntro:C}
		\label{thm:primality_minimal}
		Let $G$ be any discrete group acting minimally on a unital C*-algebra $A$. Then $A \rtimes_\lambda G$ is not prime if and only if the action has characterization (C1) involving $I(A)$. If the C*-algebra $A$ is separable, then this is furthermore equivalent to the action having characterization (C2) involving $A$ itself.
	\end{thmintro}
	
	If the group $G$ is amenable (for example, the FC-hypercentral groups of Theorem~\ref{thmIntro:B}), the proof of Theorem~\ref{thmIntro:C} is not as difficult. However, the non-amenable setting makes use of a highly non-trivial technical lemma, namely Lemma~\ref{lem:infinitely_many_distinct_translates_unbounded_sum}.
	
	Finally, we obtain an immediate corollary to Theorem~\ref{thmIntro:C}, one which is perhaps a bit surprising at first. It is reminiscent of a likely well-known but somewhat folklore analogue in the theory of tracial von Neumann algebras. Consider an ICC group $G$ acting on a tracial von Neumann algebra $(M,\tau)$, and assume that both the action on $M$ is trace-preserving and the action on $Z(M)$ is ergodic. In this case, the von Neumann crossed product $M \closure{\rtimes} G$ is a factor. For group C*-algebras, we also have the equally folklore result that $C^*_\lambda(G)$ is prime if (and only if) $G$ is ICC \cite[Proposition~2.3]{Murphy_prim}. The C*-crossed product version is the following:
	
	\begin{corintro}[Corollary~\ref{cor:ICC}]
		\label{cotIntro:D}
		If $G$ is a discrete ICC group acting minimally on a unital C*-algebra $A$, then the reduced crossed product $A\rtimes_\lambda G$ is prime.
	\end{corintro}
	
	\section{Preliminaries}
	
	\subsection{Monotone complete C*-algebras}
	\label{sec:preliminaries:monotone_complete}
	
	As made abundantly clear in the introduction, the injective envelopes $I(A)$ and $I_G(A)$ (reviewed in Section~\ref{sec:preliminaries:injective_envelopes}) played a crucial role in all recent ideal-structure results for $A \rtimes_\lambda G$. One of the most important facts about injective algebras that make them especially convenient to work with is that they are what is known as \emph{monotone complete}. This means that every bounded increasing net of self-adjoint elements admits a supremum.
	
	Every von Neumann algebra is monotone complete, and while the converse is false, such algebras will still in many ways behave as if they were von Neumann algebras. Quite important to us is the fact that polar decompositions still work inside monotone complete C*-algebras (for a more general version, we refer to \cite[Section~21]{Berberian}).
	
	\begin{proposition}[{\cite[Lemma~2.1]{YenPolarDec}} and {\cite[Lemma~4.2]{SasakiPolarDec}}] \label{prop:monotonecompletePolarDec}
		Let $A$ be a monotone complete C*-algebra, and let $a \in A$. There exists a unique projection $p\in A$, called the \emph{right projection} of $a$, such that $ap = a$, and $ab = 0$ if and only if $pb = 0$, for every $b\in A$. This projection is denoted by $\RP(a)$. The \emph{left projection} of $a$ is defined similarly, and is denoted by $\LP(a)$.
		Moreover, there is a unique partial isometry $u \in A$ with the property that $a = u \abs{a}$ and $u^*u = \RP(a)$.
	\end{proposition}
	
	It is important to recognize a subtle point, and that is that the left and right projections $\LP(a)$ and $\RP(a)$ were defined \emph{intrinsically} in terms of the monotone complete C*-algebra $A$, and \emph{not} in terms of any representation $A \subseteq B(H)$. In the setting of von Neumann algebras $M \subseteq B(H)$, the polar decomposition $x = u \abs{x}$ of an element $x \in M$ is typically done with requiring that the support projection of $u$ is the projection onto $\closure{\operatorname{ran}} \abs{x}$. In this setting, this projection lies in $M$ and coincides with $\RP(x)$,
	but in the general monotone complete setting, there is no reason to expect them to coincide.
	
	Also worth mentioning is the fact that, just like von Neumann algebras canonically admit a weak*-topology, monotone complete C*-algebras admit notions of order convergence, with the one of interest to us given by Hamana in \cite[Section~1]{hamana82_mc_tensor_products_I}. Note that Hamana defines a notion of convergence for nets, without necessarily showing that it arises out of a topology. While we will not actually need to use this notion of convergence most of the time, it is still worth keeping its existence in mind, as in particular it is a very necessary part of the definition of Hamana's monotone complete crossed product, which we make heavy use of.

	Many times, certain properties of a non-monotone complete C*-algebra can be studied more easily by embedding it into a certain monotone completion. This is done, for example, in \cite{hamana81_regular_embeddings}, where the \emph{regular monotone completion} $\closure{A}$ of $A$ is constructed. This object admits nice abstract properties that describe it uniquely, but is perhaps a bit difficult to get a concrete handle on. If we instead consider a crossed product C*-algebra $A \rtimes_\lambda G$, then, under certain conditions, there is another monotone complete C*-algebra that it embeds into, and is far easier to explicitly write down. In \cite[Section~3]{hamana82_mc_tensor_products_II}, the \emph{monotone complete crossed product} is defined as follows:
	
	Assume $G$ is a discrete group acting on a unital C*-algebra $A \subseteq B(H)$, and recall that a reduced crossed product $A \rtimes_\lambda G$ can be viewed as bounded operators acting on the Hilbert space $H \tensor \ell^2(G)$. One may concretely view every operator $T \in B(H \tensor \ell^2(G))$ as a matrix $T = [T_{r,s}]_{r,s \in G}$ over $B(H)$. With respect to this representation, every finitely supported element $\sum_{t \in G} a_t \lambda_t\in A\rtimes_\lambda G$ embeds as the matrix $[r^{-1} \cdot a_{rs^{-1}}]_{r,s \in G}$. 
	Using the fact that such symmetry is still present after taking limits, every element $a\in A\rtimes_\lambda G$ can be written as a formal sum $\sum_{t\in G}a_t\lambda_t$, in a unique way. 
	However, there is nothing stopping us from instead considering the following, perfectly valid, operator subsystem of $B(H \tensor \ell^2(G))$.
	\[M(A,G) \defeq \setbuilder{\sum_{t \in G} a_t \lambda_t \text{ (formal sum)}}{[r^{-1} \cdot a_{rs^{-1}}]_{r,s \in G}\in B(H \tensor \ell^2(G))}. \]
	
	It is immediate that the coefficients $(a_t)_{t\in G}$ uniquely determine the elements of $M(A,G)$. What is not immediately clear is why $M(A,G)$ is, in any way, closed under multiplication. We summarize below results from
	\cite[Section~6]{hamana82_mc_tensor_products_I} and \cite[Section~3]{hamana82_mc_tensor_products_II} (see also \cite[Section~3]{hamana85-injective_envelopes_equivariant}), which show that things mostly still behave how one would expect, as long as we specify the correct multiplication structure to use.
	
	\begin{proposition}
		Let $G$ be a discrete group acting on a monotone complete C*-algebra $A$. Then $M(A,G)$ is a C*-algebra, when inheriting the involution and Banach space structure from $B(H \tensor \ell^2(G))$, but equipped with the \emph{new} multiplication
		\[ [ x_{r,s} ]_{r,s\in G} \cdot [ y_{r,s} ]_{r,s\in G} \coloneqq \left[ O-\sum_{t \in G} x_{r,t} y_{t,s} \right]_{r,s\in G} \]
		where $[x_{r,s}]_{r,s \in G}$ and $[y_{r,s}]_{r,s \in G}$ are the matrix representations with respect to the Hilbert space $H \tensor \ell^2(G)$, and $O-\sum_{t \in G} x_{r,t} y_{t,s}$ denotes the order-limit in $A$ of the finite sums. The multiplication of the formal sums is reflected in the following way:
		\[ \left( \sum_{g \in G} a_g \lambda_g \right) \cdot \left( \sum_{g \in G} b_g \lambda_g \right) = \sum_{g \in G} \left( O-\sum_{t \in G} a_{t^{-1}} (t^{-1} \cdot b_{tg}) \right) \lambda_g. \]
		Moreover, $M(A,G)$ is monotone complete.
	\end{proposition}
	
	\begin{remark}
		As remarked in \cite[Section~3]{hamana85-injective_envelopes_equivariant}, the above multiplication does not necessarily coincide with the usual multiplication in $B(H \tensor \ell^2(G))$, due to the fact that we are taking an order-limit in $A$ instead of, say, a strong limit in $B(H)$.
	\end{remark}
	
	Most of the time, we will not be multiplying arbitrary elements in $M(A,G)$ together, and thus the reader should not worry too much about order-convergent sums. However, one fact that we absolutely will be making heavy use of is the fact that $M(A,G)$ is itself monotone complete, whenever $A$ is.
	
	\begin{remark}\label{rem:CondExpeMonComp}
		Just as for reduced crossed products, there is an equivariant faithful conditional expectation $E \colon M(A,G) \to A$, extending the canonical conditional expectation on $A \rtimes_\lambda G$. Every $x \in M(A,G)$ is given by the formal sum $x=\sum_{g\in G}E(x\lambda_g^*)\lambda_g$.
	\end{remark}

	Moving back to the more basic theory of monotone complete C*-algebras, they admit fairly canonical monotone complete subalgebras. It is an easy exercise to show that the fixed point set $A^G$ of a group $G$ acting on a monotone complete C*-algebra $A$ by *-automorphisms is itself monotone complete, and in particular, we obtain the following result:

	\begin{proposition}
		\label{prop:monotone_complete_center}
		Let $A$ be a monotone complete C*-algebra. Then the center $Z(A)$, is automatically monotone complete.
	\end{proposition}

	We will often be interested with working more hands-on with the center of a monotone complete C*-algebra. It is a well-known result (we include a more modern citation) that, in the commutative setting, monotone complete C*-algebras are determined by a topological condition on their spectra.
	
	\begin{theorem}[{\cite[Theorem~2.3.7]{SWMonotoneComplete}}]
		\label{thm:commutative_monotone_complete_iff_extremally_disconnected}
		Let $X$ be a compact Hausdorff space. The C*-algebra $C(X)$ is monotone complete if and only if $X$ is \emph{extremally disconnected}, i.e.\ the closure of any open set $U \subseteq X$ is in fact clopen.
	\end{theorem}
	
	\subsection{Injective envelopes}
	\label{sec:preliminaries:injective_envelopes}
	
	The theory of injective envelopes of C*-algebras was introduced by Hamana in the 1970s and 1980s, originally in \cite{hamana79_injective_envelopes_cstaralg} and \cite{hamana79_injective_envelopes_opsys}, along with an equivariant version in \cite{hamana85-injective_envelopes_equivariant}. 
	When dealing with Hamana's theory, the right category to work in is the category of \emph{$G$-operator systems}. That is, the category where objects are operator systems with an action of a discrete group $G$ by unital complete order isomorphisms (these automatically become *-isomorphisms whenever the operator system $G$ is acting upon is a C*-algebra). The non-equivariant version can be obtained by letting $G = \set{e}$. However, the objects we deal with in this paper will always end up being C*-algebras. Therefore, for convenience, we recall some of Hamana's theory here in a way that requires only very basic theory of operator systems (see \cite{paulsen_cb_maps} for an extensive introduction).
	
	\begin{definition}
		Let $S$ and $T$ be operator systems. We say that $\phi\colon S\to T$ is a \emph{complete order isomorphism} if $\phi$ is a unital completely positive linear isomorphism with completely positive inverse.
	\end{definition}
	
	When two operators systems $S$ and $T$ are equipped with an action of a discrete group $G$ by complete order isomorphisms, we say that $S$ embeds equivariantly into $T$, and write $(S,G)\subseteq (T,G)$, if there is a $G$-equivariant unital completely positive map $\phi: S\to T$ which is a complete order isomorphism onto its range. Such a map is called an \emph{equivariant embedding}.
	
	\begin{theorem}
		Let $G$ be a discrete group acting on a unital C*-algebra $A$. Then there exists a unique $G$-C*-algebra, denoted $I_G(A)$, with the following properties:
		\begin{enumerate}
			\item $(A,G)\subseteq (I_G(A),G)$ (inclusion of C*-algebras).
			\item Let $S$ and $T$ be $G$-operator systems such that $(S,G)\subseteq (T,G)$. Then every unital $G$-equivariant completely positive map $\phi\colon S\to I_G(A)$ extends to a unital $G$-equivariant completely positive map $\tilde{\phi}\colon T\to I_G(A)$. 
			\item Let $S$ be a $G$-operator system and assume that $\phi\colon I_G(A)\to S$ is a $G$-equivariant unital completely positive map which restricts to an embedding on $A$. Then $\phi$ is an embedding.
		\end{enumerate}
		
	\end{theorem}
	
	When $G=\set{e}$, one obtains the \emph{injective envelope} of $A$, which is denoted by $I(A)$.
	
	\begin{definition}
		\label{def:GinjectiveGrigidetc}
		Property (2) expresses the fact that $I_G(A)$ is a \emph{$G$-injective} object in the category of $G$-operator systems. Property (3) is known as \emph{$G$-essentiality} of the inclusion $A\subseteq I_G(A)$. Moreover, the inclusion $A\subseteq I_G(A)$ is \emph{$G$-rigid}. Namely, the identity map on $I_G(A)$ is the only $G$-equivariant unital completely positive map $\phi\colon I_G(A)\to I_G(A)$ which restricts to the identity map on $A$ (see \cite[Lemma~2.4]{hamana85-injective_envelopes_equivariant}). If we set $G = \set{e}$, we obtain the non-equivariant versions of these properties, and we simply call them \emph{injectivity}, \emph{essentiality}, and \emph{rigidity}.
	\end{definition}
	
	Very importantly, we may apply all of our theory of monotone complete C*-algebras to the theory of injective envelopes, as the following well-known proposition shows.
	
	\begin{proposition}
		\label{prop:injective_is_mc_and_ginjective_is_injective}
		An injective C*-algebra is automatically monotone complete, and in particular this applies to $I(A)$ for any unital C*-algebra $A$. Moreover, a $G$-injective C*-algebra is automatically injective, and is therefore monotone complete as well. In particular, this applies to $I_G(A)$ for any unital C*-algebra $A$.
	\end{proposition}

	As remarked in \cite[Remark~2.3]{hamana85-injective_envelopes_equivariant}, $I_G(A)$ is always an injective C*-algebra. However, as $I(A)$ is an essential extension of $A$, it is not hard to construct the following chain of inclusions, where the embeddings are operator system embeddings:
	\[ A\subseteq I(A)\subseteq I_G(A). \]
	The next proposition shows that we may in fact obtain such inclusions in a $G$-C*-algebra sense.
	
	\begin{proposition}[{\cite[Section~3]{hamana85-injective_envelopes_equivariant}}, {\cite[Corollary~4.3]{hamana79_injective_envelopes_cstaralg}}, and {\cite[Lemma~6.2]{hamana85-injective_envelopes_equivariant}}]
		\label{prop:injective_envelope_inclusions}
		Let $G$ be a discrete group acting on a unital C*-algebra $A$. The action of $G$ on $A$ extends uniquely to an action of $G$ on $I(A)$, and there is a $G$-equivariant injective *-homomorphism from $I(A)$ to $I_G(A)$ restricting to the identity map on $A$. In other words, we may view
		\[ (A,G) \subseteq (I(A),G) \subseteq (I_G(A),G). \]
        These inclusions automatically restrict to inclusions on the center of each algebra. That is,
		\[ (Z(A),G) \subseteq (Z(I(A)),G) \subseteq (Z(I_G(A)),G). \]
	\end{proposition}

	A slightly surprising version of Proposition~\ref{prop:injective_envelope_inclusions} for crossed products is the following theorem by Hamana, which will be instrumental in passing from $I_G(A)$ to $I(A)$.
	
	\begin{theorem}[{\cite[Theorem~3.4]{hamana85-injective_envelopes_equivariant}}]
		\label{thm:crossed_product_injective_envelope_inclusions}
		Let $G$ be a discrete group acting on a unital C*-algebra $A$. We have a *-homomorphic embedding $I_G(A) \rtimes_\lambda G \injectsinto I(A \rtimes_\lambda G)$ that restricts to the identity on the copy of $A \rtimes_\lambda G$ in both algebras. Consequently, we may view
		\[ A \rtimes_\lambda G \subseteq I(A) \rtimes_\lambda G \subseteq I_G(A) \rtimes_\lambda G \subseteq I(A \rtimes_\lambda G). \]
	\end{theorem}
	
	Keeping the above in mind, the following result is a very easy but very important observation, which follows from the uniqueness of the injective envelope. It will let us transfer properties between $A \rtimes_\lambda G$, $I(A) \rtimes_\lambda G$, and $I_G(A) \rtimes_\lambda G$.
	
	\begin{proposition}
		\label{prop:intermediate_algebra_shares_injective_envelope}
		Let $A$ and $B$ be unital C*-algebras such that $A \subseteq B \subseteq I(A)$. Then there is a *-isomorphism $I(B) \isoto I(A)$ which restricts to the identity on $B$.
	\end{proposition}
	
	Hamana proves in \cite[Theorem~7.1]{hamana81_regular_embeddings} the equivalence between a C*-algebra $B$ being prime, and the regular monotone completion $\closure{B}$ being a factor. This is also true when considering the injective envelope $I(B)$ instead, which can be proven using the exact same proof, or the fact that $Z(\closure{B}) = Z(I(B))$ (see \cite[Theorem~6.3]{hamana81_regular_embeddings}).
	Together with Theorem~\ref{thm:crossed_product_injective_envelope_inclusions} and Proposition~\ref{prop:intermediate_algebra_shares_injective_envelope}, we obtain the following proposition:
	
	\begin{proposition}
		\label{proposition:primeifffactor}
		Let $B$ be a unital C*-algebra. Then $B$ is prime if and only if $I(B)$ is a factor.
		In particular, for a discrete group $G$ acting on a unital C*-algebra $A$, the following are equivalent:
		\begin{enumerate}
			\item $A\rtimes_\lambda G$ is prime.
			\item $I(A)\rtimes_\lambda G$ is prime.
			\item $I_G(A)\rtimes_\lambda G$ is prime.
		\end{enumerate}
	\end{proposition}
	
	Although $I(A)\rtimes_\lambda G$ and $I_G(A)\rtimes_\lambda G$ embed into the monotone complete C*-algebra $I(A\rtimes_\lambda G)$, which is also their common injective envelope, we may alternatively consider the embeddings of $I(A)\rtimes_\lambda G$ and $I_G(A)\rtimes_\lambda G$ into the monotone complete C*-algebras $M(I(A),G)$ and $M(I_G(A),G)$, respectively. This is usually more helpful, because it allows us to work with elements that still admit well-behaved series $x = \sum_{t \in G} x_t \lambda_t$, with $x_t$ inside $I(A)$ or $I_G(A)$.
	
	We will also need the following easy observation.
	
	\begin{proposition}
		\label{prop:minimality_transfers_to_injective_envelopes_and_centers}
		Let $G$ be a discrete group acting minimally on a unital C*-algebra $A$. The induced actions on the C*-algebras $I(A)$, $I_G(A)$, $Z(A)$, $Z(I(A))$, and $Z(I_G(A))$ are all minimal as well.
	\end{proposition}
	
	\begin{proof}
		Assume that $J \subseteq I_G(A)$ is a non-trivial $G$-invariant ideal of $I_G(A)$. The quotient map $q \colon I_G(A) \to I_G(A) / J$ is necessarily injective on $A$ by minimality, but is not injective on $I_G(A)$. This contradicts $G$-essentiality of the inclusion $A\subseteq I_G(A)$.
		
		The exact same proof will work for $I(A)$, since the inclusion $A\subseteq I(A)$ is essential and, in particular, $G$-essential.

		
		Since $A$ was arbitrary, if we show that $(Z(A),G)$ is minimal, this will automatically hold for $I(A)$ and $I_G(A)$ as well. To see this, identify $Z(A)$ with $C(X)$, for some compact space $X$, and assume that $U \subseteq X$ is a non-zero proper invariant open subset for the induced action on $X$.
		
		Choose $x\in X\setminus U$, and extend the evaluation map $\delta_x\colon Z(A)\to\mathbb{C}$ to a state $\rho\colon A\to \mathbb{C}$. Then $\rho$ vanishes on the ideal generated by $C_0(U)$ in $A$, namely $\closure{\operatorname{span}} (C_0(U) \cdot A)$, since $Z(A)$ lies in the multiplicative domain (see \cite[Proposition~1.5.7]{Brown_ozawa} and \cite[Definition~1.5.8]{Brown_ozawa}) of $\rho$. (As a side note, the ideal $\closure{\operatorname{span}} (C_0(U) \cdot A)$ is just $C_0(U) \cdot A$ by the Cohen-Hewitt factorization theorem). In any case, as $\rho$ is non-trivial, this ideal (which is clearly $G$-invariant and non-zero) cannot be all of $A$, contradicting minimality of the action.
	\end{proof}

	We also briefly mentioned in the introduction of this paper that a large part of the difficulty in characterizing simplicity of $A \rtimes_\lambda G$ is the difficulty in passing from results on $I_G(A)$ to results on $I(A)$. The following proposition, which we prove here for convenience, was observed by Hamana in \cite[Remark~3.8]{hamana85-injective_envelopes_equivariant}.
	
	\begin{proposition}
		Let $G$ be a discrete amenable group and let $M$ be an injective von Neumann algebra. Then $M$ is automatically $G$-injective.
	\end{proposition}

	This is highly contingent on $M$ being a von Neumann algebra, thus having the unit ball be compact under a nice topology (and, for example, being able to apply fixed point theorems for amenable groups acting on compact convex sets). In general, $I(A)$ is only monotone complete, and thus the closest analogue is a notion of order convergence. See, for example, the discussion near the start of \cite[Section~1]{hamana82_mc_tensor_products_I}. It is unlikely that the unit ball of $I(A)$ is compact under any order topology, and thus this crucial piece of the puzzle is missing in the above proof. In other words, one expects $I_G(A) \neq I(A)$. With that being said, this does not rule out the existence of some other deep reason for the equality $I_G(A) = I(A)$ to still hold, and such a result would indeed trivialize half of our work. 
	
	\subsection{Properly outer automorphisms}
	\label{sec:preliminaries:properly_outer}
	This subsection serves two purposes. The first is simply for intuition on the statements of the main theorems (Theorem~\ref{thmIntro:A}, Theorem~\ref{thmIntro:B}, and Theorem~\ref{thmIntro:C}), as all of them characterize simplicity (or primality) using an equivariant version of proper outerness, or the lack thereof. However, observe that in each theorem, \emph{two} characterizations are given. One is a more elegant, but perhaps more intractable characterization involving the somewhat mysterious injective envelope $I(A)$. The other is a characterization involving the ideals of $A$ (and their multiplier algebras), and it is more down-to-earth and understandable, even if it is more technically cumbersome. These are both based off of two equivalent characterizations for what is known as \emph{proper outerness} of an action of $G$ on a C*-algebra $A$.
	
	Our story begins with von Neumann algebras. Recall that an automorphism $\alpha$ of a measure space $(X,\mu)$ is said to be \emph{essentially free} if the set of fixed points of $\alpha$ has measure zero, and this notion is very important when studying actions on abelian von Neumann algebras $L^\infty(X,\mu)$. In \cite{kallman_free_actions}, Kallman generalizes this notion to the noncommutative setting, and introduces the notion of a \emph{freely acting} automorphism of a von Neumann algebra. However, it was observed by Hamana in \cite[Proposition~5.1]{hamana82_mc_tensor_products_I} that because monotone complete C*-algebras in many ways behave like von Neumann algebras, in particular with being able to take polar decompositions, all of Kallman's results work exactly the same in this more general setting. The following series of results are those originally observed by Kallman and Hamana:
	
	\begin{definition}\label{def:Prop_outer_monotone_complete}
		Let $A$ be a monotone complete C*-algebra. We say that a *-automorphism $\alpha \in \Aut(A)$ is \emph{freely acting} if whenever $x \in A$ satisfies $xy = \alpha(y) x$ for all $y \in A$, we have $x = 0$. We say that $\alpha$ is \emph{properly outer} if there is no non-zero $\alpha$-invariant central projection $p \in A$ with $\alpha|_{Ap}$ inner.
	\end{definition}
	
	The following result shows that in the setting of monotone complete C*-algebras, there is no difference between the notions defined above. Actually, there is a bit of a subtle point in the statement of the theorem. Namely, in the definition of proper outerness, one can ask whether $\alpha|_{pAp}$ is ever inner for $p$ not necessarily central, but this does not actually change the definition.
	
	\begin{theorem}[{\cite[Proposition~5.1]{hamana82_mc_tensor_products_I}}]
		\label{thm:PropInnerDecoMonCom}
		Let $A$ be a monotone complete C*-algebra, and let $\alpha \in \Aut(A)$. There is a largest $\alpha$-invariant projection $p \in A$ such that $\alpha|_{pAp}$ is inner. Moreover, $p$ is central and $\alpha|_{A(1-p)}$ is freely acting. In fact, this decomposition of $A \isoto Ap \directsum A(1-p)$ into inner and freely acting parts is unique.
	\end{theorem}
	
	For the proof of Theorem~\ref{thm:PropInnerDecoMonCom}, Hamana uses the following theorem, which we explicitly state below as we will make use of it as well. This can be proven by closely following the proof of Kallman in \cite[Theorem~1.1]{kallman_free_actions}, in the setting of von Neumann algebras.
	
	\begin{theorem}
		\label{thm:PropOuterCorner}
		Let $A$ be a monotone complete C*-algebra, let $\alpha \in \Aut(A)$, and assume that $x \in A$ implements the automorphism in the sense that $xy = \alpha(y) x$ for all $y \in A$. Let $x = u \abs{x}$ be the polar decomposition of the element $x$, and let $p = u^*u$ be the domain projection of the partial isometry $u$. We have:
		\begin{enumerate}
			\item The projection $p$ is an $\alpha$-invariant central projection in $A$;
			\item The partial isometry $u$ is a unitary in $U(Ap)$, and $\alpha|_{Ap} = \Ad u$.
		\end{enumerate}
	\end{theorem}

	We state the following corollary (which we mentioned earlier) to Theorem~\ref{thm:PropInnerDecoMonCom}.
	
	\begin{corollary}
		Given a *-automorphism $\alpha \in \Aut(A)$ on a monotone complete C*-algebra $A$, it is freely acting if and only if it is properly outer.
	\end{corollary}
	
	Everything that was done above shows that proper outerness is a very intuitive and well-behaved property in the setting of monotone complete C*-algebras. Our discussion now takes a turn towards the setting of general C*-algebras, given that an equivariant version of proper outerness forms the second half of each of Theorem~\ref{thmIntro:A}, Theorem~\ref{thmIntro:B}, and Theorem~\ref{thmIntro:C}. As mentioned in the introduction, such notions (in their non-equivariant versions) have played an important role in studying the simplicity of reduced crossed products $A \rtimes_\lambda G$. Multiple generalizations of proper outerness to automorphisms on arbitrary C*-algebras were given by various authors, and they generally all coincide when the underlying C*-algebra is separable. In particular, consider the properties listed in Theorem~\ref{thm:properly_outer_equivalence} below. The proofs of these implications are somewhat subtle and trace back through several papers (see, for example, \cite[Remark~2.2]{elliott_properly_outer} and \cite[Theorem~2.1]{hamana82_okayasu_saito}, \cite[Section~1]{hamana82_centre}, \cite[Theorem~7.3]{hamana85-injective_envelopes_equivariant}, and \cite[Theorem~6.6]{olesen_pedersen_III}).

	The precise definition of the Borchers spectrum itself can be found, for example, in \cite[Section~8.8]{pedersen_cstar_algebras_and_their_automorphism_groups} (alternatively, see \cite[Section~3]{olesen_pedersen_III} or \cite[Section~1]{Kishimoto_freely_acting}). The Borchers spectrum of a single automorphism $\alpha \in \Aut(A)$ is defined as the Borchers spectrum for the corresponding $\Z$-action on $A$, where $n \in \Z$ acts by $\alpha^n$. 

	\begin{theorem}
		\label{thm:properly_outer_equivalence}
		Let $A$ be a unital C*-algebra, and let $\alpha \in \Aut(A)$, with its unique extension to $I(A)$ also denoted by $\alpha$. Consider the following conditions:
		\begin{enumerate}
			\item
			\label{thm:properly_outer_equivalence:borchers}
			There is a non-zero $\alpha$-invariant ideal $J \subseteq A$ such that $\Gamma_B(\alpha|_J) = \set{e}$, where $\Gamma_B(\cdot)$ represents the Borchers spectrum of the automorphism.
			
			\item
			\label{thm:properly_outer_equivalence:injective_envelope}
			There is a non-zero $\alpha$-invariant central projection $p \in I(A)$, and a unitary $u \in U(I(A)p)$ such that $\alpha$ acts by $\Ad u$ on $I(A)p$.
			
			\item
			\label{thm:properly_outer_equivalence:elliott}
			There is a non-zero $\alpha$-invariant ideal $J \subseteq A$ and a unitary $u \in M(J)$, such that $\norm{\alpha|_J - (\Ad u)|_J} < 2$.
			
			\item
			\label{thm:properly_outer_equivalence:elliott_strong}
			Given any $\varepsilon\in (0,2)$, there exists a non-zero $\alpha$-invariant ideal $J \subseteq A$ and a unitary $u \in M(J)$ such that $\norm{\alpha|_J - (\Ad u)|_J} < \varepsilon$.
		\end{enumerate}
		We have that (\ref{thm:properly_outer_equivalence:borchers}) and (\ref{thm:properly_outer_equivalence:injective_envelope}) are equivalent, and both are implied by (\ref{thm:properly_outer_equivalence:elliott}) and (\ref{thm:properly_outer_equivalence:elliott_strong}). If $A$ is separable, all four are equivalent.
	\end{theorem}
	
	In all of the above properties, the intuition behind them is that the automorphism $\alpha \in \Aut(A)$ is ``almost inner'' on an ideal $J \subseteq A$. Given that proper outerness is supposed to mean ``not inner on any piece of the C*-algebra'', the following definition makes sense.
	
	\begin{definition}
    \label{def:properly_outer}
		Let $A$ be a unital C*-algebra. A *-automorphism $\alpha \in \Aut(A)$ is said to be \emph{properly outer in the sense of Kishimoto} if it satisfies the \emph{negation} of condition~(\ref{thm:properly_outer_equivalence:borchers}) in Theorem~\ref{thm:equivariant_proper_outer_equivalence} (originally introduced in \cite{Kishimoto_freely_acting} as ``freely acting''), and hence equivalently the negation of condition~(\ref{thm:properly_outer_equivalence:injective_envelope}) in that theorem. It is called \emph{properly outer in the sense of Elliott} if it satisfies the \emph{negation} of condition~(\ref{thm:properly_outer_equivalence:elliott}) in Theorem~\ref{thm:equivariant_proper_outer_equivalence}. Typically, ``properly outer'' on its own means properly outer in the sense of Kishimoto. 
		
		An action $\alpha\colon G\to \Aut(A)$ of a discrete group $G$ on the C*-algebra $A$ is said to be properly outer if each $\alpha_t$ corresponding to $t \in G \setminus \set{e}$ is properly outer.
	\end{definition}
	
	Recall that $I(A)$ is in many ways small enough to still remember many of the basic properties of $A$. 
	Thus, given that in the monotone complete setting, it is clear what the ``correct'' definition of proper outerness is (recall Definition~\ref{def:Prop_outer_monotone_complete}), the ``correct'' definition for an automorphism $\alpha\in \Aut(A)$ on a general C*-algebra $A$ should be the one that always coincides with the (unique) extension of $\alpha$ to $I(A)$ being properly outer. Hence, Kishimoto's definition is in some sense preferable to Elliott's definition (which only coincides with being properly outer on $I(A)$ in the separable setting). However, in the separable setting, where all of these definitions coincide, Elliott's definition is likely the more easy of the two to verify. Kishimoto's definition involves the use of the Borchers spectrum, and it is substantially more cumbersome to work with.
	
	\subsection{FC-hypercentral groups}
	\label{sec:preliminaries:fc_hypercentral}
	
	The center of a group $G$ can be phrased as the set of elements having a conjugacy class of size $1$. The FC-center $\FC(G)$ is a slight generalization of this concept, where we instead consider the set of elements admitting conjugacy classes of finite size. It is not hard to check that this is indeed a normal subgroup of $G$. This concept can once again be taken further. Consider now the quotient $G/\FC(G)$. There is no guarantee that there are no non-trivial elements of finite conjugacy class in this new quotient. That is, it may be the case that $\FC(G/\FC(G))$ is again non-trivial, and so we may quotient by it again.
	
	The FC-central series is what is obtained after performing the above process ordinal-many times. Start with $F_1 = \FC(G)$, and define the sets $F_\alpha$ for the ordinal numbers $\alpha$ as follows:
	
	\begin{enumerate}
		\item For successor ordinals $\alpha + 1$, define $F_{\alpha + 1}$ as the (necessarily normal) subgroup of $G$ satisfying $F_{\alpha + 1} / F_\alpha = \FC(G / F_\alpha)$, that is, $F_{\alpha + 1}$ is the preimage of $\FC(G / F_\alpha)$ under the quotient map $G \surjectsonto G / F_\alpha$.
		
		\item For limit ordinals $\beta$, define $F_\beta = \bigcup_{\alpha < \beta} F_\alpha$ (again, necessarily normal in $G$).
	\end{enumerate}
	
	\begin{definition}
		The \emph{FC-hypercenter} of $G$, denoted $\FCH(G)$, is the union of all of the $F_\alpha$ above. We say that $G$ is \emph{FC-hypercentral} if $\FCH(G) = G$. We say that $G$ is an \emph{FC-group} if $\FC(G)=G$.
	\end{definition}
	
	Recall that a group $G$ is called \emph{ICC} if every $g\in G\setminus \{e\}$ has infinite conjugacy class. FC-hypercentral groups are precisely the class of groups which have no non-trivial ICC quotients. This class contains all virtually nilpotent groups, and, in the finitely generated setting, coincides with the class of virtually nilpotent groups, and thus, also with the class of polynomial growth groups by Gromov's Theorem (see \cite[Theorem~2]{McLain_remarks_upper_central_series} and \cite[Theorem~2]{DM56_FcNipFcSol}). Likewise, it is also the case (and substantially easier to show) that, in the finitely generated setting, the class of FC-groups is precisely the class of groups $G$ where the center $Z(G)$ has finite index in $G$.
    
	In recent breakthrough results, the class of countable (not necessarily finitely generated) FC-hypercentral groups was identified with the class of strongly amenable groups \cite{FrischTamuzPooya} and with the class of groups which satisfy the \emph{Choquet-Deny property}, by \cite[Theorem 4.8]{Jaworski} and \cite{FrischHartmanTamuzPooya}. 
	We refer the reader to \cite{FrischThesis} for a smooth introduction to these exciting results. 
	
	\section{Pseudo-expectations and central elements}
	\label{sec:pseudoexpectations}
	
	Recall that there is always a canonical conditional expectation $E\colon A \rtimes_\lambda G \to A$, determined by mapping a finitely supported element $\sum_{t \in G} x_t \lambda_t\in A\rtimes_\lambda G$ to $x_e$. All of the existing ideal structure results for crossed products from the last few years pass through the machinery of what are known as \emph{pseudo-expectations}, which is an analogue obtained by expanding the codomain to the $G$-injective envelope $I_G(A)$. We recall the precise definition below.
	
	\begin{definition}[{\cite[Definition~6.1]{kennedy_schafhauser_noncommutative_crossed_products}}]
		\label{def:pseudoexpectation}
		An \emph{equivariant pseudo-expectation} for a reduced crossed product $A \rtimes_\lambda G$ is a $G$-equivariant unital and completely positive map $F \colon A \rtimes_\lambda G \to I_G(A)$, such that $F|_A = \id_A$.
	\end{definition}
	
	Pseudo-expectations were first introduced in the non-equivariant setting by Pitts in \cite{Pitts}, and even in that setting were found to be helpful for understanding the ideal structure of $A\rtimes_\lambda G$ (see \cite{Zarikian}). However, we will focus on equivariant pseudo-expectations, following the work of Kennedy and Schafhauser \cite{kennedy_schafhauser_noncommutative_crossed_products}.
	
	Before we proceed further, the following was not explicitly mentioned in \cite{kennedy_schafhauser_noncommutative_crossed_products}, but is a very convenient result that is used implicitly. We prove it for convenience, and will certainly make use of it as well.
	
	\begin{proposition}
		\label{prop:pseudoexpectations_conditional_expectations_bijection}
		Equivariant conditional expectations $F'\colon I_G(A) \rtimes_\lambda G \to I_G(A)$ are in canonical bijection (given by restriction) with equivariant pseudo-expectations $F\colon A \rtimes_\lambda G \to I_G(A)$. 
	\end{proposition}
	
	\begin{proof}
		It is clear that equivariant conditional expectations on $I_G(A)\rtimes_\lambda G$ restrict to pseudo-expectations on $A \rtimes_\lambda G$. The map $F' \mapsto F'|_{A \rtimes_\lambda G}$ is injective, since $I_G(A)$ lies in the multiplicative domain of $F'$, and so $F'$ is uniquely determined by the values it takes on the unitaries $\{\lambda_t\}_{t\in G}$.
		The restriction map is also surjective, as every equivariant pseudo-expectation $F\colon A \rtimes_\lambda G \to I_G(A)$ extends to some $G$-equivariant completely positive map $F'\colon I_G(A) \rtimes_\lambda G \to I_G(A)$, thanks to $G$-injectivity of $I_G(A)$. This extension, being the identity on $A$, is necessarily the identity on $I_G(A)$ by $G$-rigidity (see Definition~\ref{def:GinjectiveGrigidetc}).
	\end{proof}
	
	There are many reasons why pseudo-expectations are useful for studying simplicity of $A \rtimes_\lambda G$. One evidence is the following proposition, which directly characterizes simplicity, or more generally the intersection property, in terms of pseudo-expectations. 
	\begin{proposition}[{\cite[Theorem~6.6]{kennedy_schafhauser_noncommutative_crossed_products}}]
		\label{prop:intersection_property_iff_faithful}
		Let $G$ be a discrete group acting on a unital C*-algebra $A$. Then $A \rtimes_\lambda G$ has the intersection property if and only if every equivariant pseudo-expectation $F\colon A \rtimes_\lambda G \to I_G(A)$ is faithful.
	\end{proposition}
	
	This requirement that every equivariant pseudo-expectation be faithful is perhaps a bit mysterious. If $A$ is commutative, it follows from a combination of Proposition~\ref{prop:intersection_property_iff_faithful}, \cite[Theorem~3.4]{kawabe_crossed_products} and \cite[Theorem~6.4]{kennedy_schafhauser_noncommutative_crossed_products} that this is in fact equivalent to having a \emph{unique} equivariant pseudo-expectation, namely the canonical conditional expectation. (See also \cite[Theorem~4.6]{PittsZarikian} for the non-equivariant setting.) In the noncommutative setting, it is not known whether they are equivalent. This is one of the obstructions to obtaining a nice characterization of simplicity in the noncommutative case, as precisely what makes a pseudo-expectation non-faithful is poorly understood, while in contrast we have a better idea of what arbitrary pseudo-expectations look like. Proposition~\ref{prop:pseudoexpectation_coefficients} paints a picture on the latter (see also \cite{Ursu}, where similar ideas are used by the second coauthor in the context of traces on crossed products).
	We would also like to remark that in an unpublished work of Zarikian
	it was proven that, in the non-equivariant setting, having all pseudo-expectations $A \rtimes_\lambda G\to I(A)$ faithful implies that there is a unique pseudo-expectation.
	
	\begin{proposition}
		\label{prop:pseudoexpectation_coefficients}
		Let $G$ be a discrete group acting on a unital C*-algebra $A$ and let $F \colon A \rtimes_\lambda G \to I_G(A)$ be an equivariant pseudo-expectation. Then $F$ is uniquely determined by the ``coefficients'' $x_t \defeq F(\lambda_t)$,  which satisfy the following properties:
		\begin{enumerate}
			\item $x_e = 1$;
			\item $x_t y = (t \cdot y) x_t$, for all $t\in G$ and all $y \in I_G(A)$;
			\item $s \cdot x_t = x_{sts^{-1}}$, for all $s,t\in G$;
			\item The matrices $[x_{st^{-1}}]_{s,t \in \mathcal{F}}$ are positive for every finite set $\mathcal{F} \subseteq G$.
		\end{enumerate}
	\end{proposition}
	
	\begin{proof}
		By Proposition~\ref{prop:pseudoexpectations_conditional_expectations_bijection}, we may transfer the discussion to equivariant conditional expectations $F\colon I_G(A) \rtimes_\lambda G \to I_G(A)$. These are uniquely determined by the coefficients $(x_t)_{t\in G}$ since $I_G(A)$ lies in the multiplicative domain of $F$.
		
		We now turn to proving each individual property of the coefficients. It is clear that $x_e = 1$. Now let $y \in I_G(A)$ and $t\in G$ be given, and using multiplicative domain again, observe that:
		\[ x_t y = F(\lambda_t) y = F(\lambda_t y) = F((t \cdot y) \lambda_t) = (t \cdot y) F(\lambda_t) = (t \cdot y) x_t. \]
		Furthermore, $s \cdot x_t = x_{sts^{-1}}$ is due to $G$-equivariance of $F$. Finally, the last positivity condition is due to the fact that
		\[ F^{(n)} \left( \begin{bmatrix} \lambda_{s_1} \\ \vdots \\ \lambda_{s_n} \end{bmatrix} \begin{bmatrix} \lambda_{s_1} \\ \vdots \\ \lambda_{s_n} \end{bmatrix}^* \right) = \left[x_{s_i s_j^{-1}}\right]_{i,j = 1, \dots, n}. \]
	\end{proof}
	
	\begin{remark}
		It is not immediately obvious, but it can actually be shown that every time one has such coefficients $(x_t)_{t \in G}$ as above, this will define a pseudo-expectation $F\colon A \rtimes_\lambda G \to I_G(A)$ through $F(a \lambda_t) \coloneqq a x_t$. This is essentially what is proven, in a very special case, in \cite[Lemma~9.1]{kennedy_schafhauser_noncommutative_crossed_products}. However, we will not make use of this converse, and will not make use of the positive-definiteness requirement either.
	\end{remark}
	
	The following highlights another, seemingly unrelated, area in which such coefficients $(x_t)_{t \in G}$ show up. This will be crucial for us when constructing non-trivial central elements.

	\begin{proposition}
		\label{prop:crossed_product_center_coefficients}
		Let $G$ be a discrete group acting on a unital C*-algebra $A$, and consider an element $x\coloneqq \sum_{t \in G} x_t^* \lambda_t\in A\rtimes_\lambda G$. (The stars on $x_t$ are intentional). Then, $x\in Z(A\rtimes_\lambda G)$ if and only if
		\begin{enumerate}
			\item $x_t y = (t \cdot y) x_t$, for all $t\in G$ and all $y \in A$;
			\item $s \cdot x_t = x_{sts^{-1}}$, for all $s,t\in G$.
		\end{enumerate}
	\end{proposition}
	
	\begin{proof}
		Given $a\in A$, it is not hard to verify that
		\[ a\cdot x=\sum_{t\in G}a x_t^*\lambda_{t} \quad \text{and} \quad x\cdot a=\sum_{t\in G} x_t^*(t\cdot a)\lambda_{t}. \]
		Using the fact that elements in $A\rtimes_\lambda G$ are uniquely determined by their coefficients, we see that $x$ commutes with the copy of $A$ inside $A\rtimes_\lambda G$ if and only if condition~(1) of the proposition is satisfied.
		
		Now given $g\in G$, we also have
		\[ \lambda_s\cdot x=\sum_{t\in G}(s \cdot x_{s^{-1}t}^*)\lambda_t \quad \text{and} \quad x\cdot \lambda_s = \sum_{t\in G} x_{ts^{-1}}^*\lambda_t. \]
		Comparing coefficients, we see that $x$ commutes with the set of unitaries $\setbuilder{\lambda_s}{s \in G}$ if and only if condition~(2) of the proposition is satisfied.
		The result then follows.
	\end{proof}
	
	The following closely related proposition was proven by Hamana.
	
	\begin{proposition}[{\cite[Lemma~10.3]{hamana85-injective_envelopes_equivariant}}]
		\label{prop:mc_crossed_product_commutant_coefficients}
		Let $G$ be a discrete group acting on a monotone complete C*-algebra $A$, and let $x\coloneqq \sum_{t \in G} x_t^* \lambda_t\in M(A,G)$ be given. The following conditions are equivalent.
		\begin{enumerate}
			\item $x\in Z(M(A,G))$;
			\item $x\in (A\rtimes_\lambda G)'\cap M(A,G)$;
			\item\begin{enumerate}
				\item $x_t y = (t \cdot y) x_t$, for all $t\in G$ and all $y \in A$;
				\item $s \cdot x_t = x_{sts^{-1}}$, for all $s,t\in G$.
			\end{enumerate}
		\end{enumerate}
	\end{proposition}
	
	The equivalence between properties (2) and (3) will be the one which is relevant for us, and its proof is similar to the proof of Proposition~\ref{prop:crossed_product_center_coefficients}. The equivalence to condition~(1) will not be used later, but it is interesting nonetheless.
	
	\begin{proposition}
		\label{prop:properly_outer_equivariant_polar_decompositions}
		Let $G$ be a discrete group acting on a monotone complete C*-algebra $A$. Assume that $(x_t)_{t \in G}$ are elements of $A$ satisfying the following properties
		\begin{enumerate}
			\item $x_t y = (t \cdot y) x_t$, for all $t\in G$ and all $y \in A$;
			\item $s \cdot x_t = x_{sts^{-1}}$, for all $s,t\in G$.
		\end{enumerate}
		Let $x_t = u_t \abs{x_t}$ be the polar decomposition of $x_t$, with $p_t = u_t^*u_t$ being the domain projection of $u_t$. Then the following relations hold:
		\begin{enumerate}
			\item The projection $p_t$ is a $t$-invariant central projection in $A$, the partial isometry $u_t$ is a unitary in $U(Ap_t)$, and $t$ acts by $\Ad u_t$ on the corner $Ap_t$;
			\item $s \cdot u_t = u_{sts^{-1}}$, for all $s,t\in G$.
		\end{enumerate}
	\end{proposition}
	
	\begin{proof}
		The first condition is an immediate application of Theorem~\ref{thm:PropOuterCorner}. For the second condition, first note that if $\alpha \colon A \to B$ is a *-isomorphism between monotone complete C*-algebras, and $x \in A$ admits polar decomposition $x = u \abs{x}$, then $\alpha(x)$ admits polar decomposition $\alpha(u) \abs{\alpha(x)}$. In particular, applying this to the *-automorphism of $s \in G$ tells us that the element $x_{sts^{-1}} = s \cdot x_t$ admits the polar decomposition $(s \cdot u_t) \abs{s \cdot x_t}$. Uniqueness of polar decomposition then tells us that $s \cdot u_t = u_{sts^{-1}}$.
	\end{proof}

	For convenience, we will sometimes want to go from a subset of pairs $(p_t,u_t)_{t\in C}$ indexed over a conjugacy class $C\subseteq G$, and satisfying the conditions above, to a single pair $(p,u)$, and vice versa.
	
	\begin{proposition}
		\label{prop:equivariant_unitaries_to_single_unitary}
		Let $G$ be a discrete group acting on a unital C*-algebra $A$, and let $C \subseteq G$ be the conjugacy class of a fixed element $r \in C$. The following are in canonical bijection with each other:
		\begin{enumerate}
			\item Sets of pairs $\set{(p_t, u_t)}_{t \in C}$, where $p_t$ is a $t$-invariant central projection in $A$, $u_t \in U(A p_t)$ is a unitary such that $t$ acts by $\Ad u_t$ on the corner $Ap_t$, and moreover $s \cdot u_t = u_{sts^{-1}}$, for all $s \in G$ and $t \in C$.
			
			\item A single $r$-invariant central projection $p$ in $A$ and a unitary $u \in U(Ap)$ such that $r$ acts by $\Ad u$ on the corner $Ap$, and moreover $s \cdot u = u$, for all $s \in C_G(r)$.
		\end{enumerate}
		The map from the first setting to the second setting is given by $p = p_r$ and $u = u_r$.
	\end{proposition}
	
	\begin{proof}
		Since we are dealing with a conjugacy class, the pairs $\{(p_t,u_t)\}_{t\in C}$, are uniquely determined by the values of $p_r$ and $u_r$, due to the fact that $p_{srs^{-1}} = s \cdot p_r$ and $u_{srs^{-1}} = s \cdot u_r$. Thus, the map between the two settings mentioned above is injective.
		
		To show that this map is surjective, we need to check that the elements $p_{srs^{-1}} \defeq s \cdot p$ and $u_{srs^{-1}} \defeq s \cdot u$, for $s\in G$, are well-defined and satisfy the required properties. Assume that $s_1 r s_1^{-1} = s_2 r s_2^{-1}$, for some $s_1,s_2\in G$. This is equivalent to the requirement that $s_2^{-1} s_1$ belongs to $C_G(r)$, and so $s_2^{-1} s_1 \cdot u = u$, or, in other words, $s_1 \cdot u = s_2 \cdot u$. Thus, the aforementioned elements give rise to well-defined pairs $\{(p_t,u_t)\}_{t\in C}$. Observe that we indeed have $p_r = p$ and $u_r = u$.
		It is straightforward to verify that each $p_t$ is a $t$-invariant central projection in $A$, that $u_t\in U(Ap_t)$, and that $s\cdot u_t=u_{sts^{-1}}$, for all $s\in G$ and all $t\in C$. 
		Finally, we check that the action of $t$ is given by $\Ad u_t$ on the corner $Ap_t$, for every $t\in C$. Given $srs^{-1} \in C$ and $x \in A p_{srs^{-1}}$ we have that $s^{-1} x \in A p$, and so
		\[ srs^{-1} \cdot x = sr \cdot (s^{-1} \cdot x) = s \cdot (u (s^{-1} \cdot x) u^*) = (s \cdot u) x (s \cdot u)^* = u_{srs^{-1}} x u_{srs^{-1}}^*. \]
	\end{proof}

	\section{The intersection property for crossed products by FC-groups}
	\label{sec:intersection_property_fc}
	
	The goal of this section is to prove the first half of Theorem~\ref{thmIntro:A}, which characterizes the intersection property for the crossed product $A \rtimes_\lambda G$ in terms of the dynamics of $G$ on $I(A)$. Specifically, the lack of the intersection property is equivalent to the existence of a triple $(p,u,r)$, where
	\begin{itemize}
		\item $r \in G \setminus \set{e}$;
		\item $p$ is an $r$-invariant central projection in $I(A)$;
		\item $u$ is a unitary in $I(A)p$ such that $r$ acts by $\Ad u$ on $I(A)p$;
		\item $u$ is $C_G(r)$-invariant.
	\end{itemize}
	We would first like to give some brief intuition and sketch the proof of this theorem in the minimal setting. The setting of the intersection property relies on some technical results that make perhaps not as much intuitive sense without the appropriate context.
	
	Assume that the action of $G$ on $A$ is minimal, and assume that $A \rtimes_\lambda G$ is not simple. Through the technology of pseudo-expectations (see Section~\ref{sec:pseudoexpectations}), it is possible to write down a non-trivial element of $Z(I_G(A) \rtimes_\lambda G)$, which in particular implies that $I_G(A) \rtimes_\lambda G$ is not prime, and, by Proposition~\ref{proposition:primeifffactor}, neither is $I(A) \rtimes_\lambda G$. From here, letting $I$ and $J$ be two non-trivial ideals of $I(A) \rtimes_\lambda G$ with $I\cdot J = 0$, we may take a supremum of the positive elements $x\in I$ with $\norm{x}<1$ inside of the monotone complete crossed product $M(I(A),G)$. Denote this supremum by $q$. We will see in Proposition~\ref{prop:ideal_sup_projection} that this element must commute with the copy of $I(A) \rtimes_\lambda G$, and it cannot be a scalar (since it must still be orthogonal to $J$ - see Proposition~\ref{prop:ideal_sup_orthogonality}). By Proposition~\ref{prop:mc_crossed_product_commutant_coefficients}, the coefficients of $q= \sum_{r \in G} x_r \lambda_r\in M(I(A),G)$ satisfy certain properties that we are after, and moreover, there is at least one $r\in G\setminus\{e\}$ for which $x_r\neq 0$ (otherwise, $q$ would be a non-trivial element of $Z(I(A))^G$, and this is impossible due to minimality - see Proposition~\ref{prop:minimality_transfers_to_injective_envelopes_and_centers}). Following the procedure carried out in Proposition~\ref{prop:properly_outer_equivariant_polar_decompositions} and Proposition~\ref{prop:equivariant_unitaries_to_single_unitary}, we deduce that the polar decomposition of $x_r$ inside $I(A)$ gives rise to a triple $(p,u,r)$ that has the desired properties. 
	
	In order to deal with the general case, we will at some point need to define a suitable variant of being prime (see Lemma~\ref{lemma:stronglynonprime}), one that is appropriate for dealing with the intersection property when the action is not minimal.
	
	We would now like to also motivate the converse direction of the Theorem, relative to the ideas of Kennedy and Schafhauser in \cite{kennedy_schafhauser_noncommutative_crossed_products} (specifically for the case of amenable groups, which slightly simplifies their proof).
	The existence of a triple $(p,u,r)$ as in the theorem is clearly stronger than requiring the action on $I(A)$ (equivalently, on $A$) to be non properly outer. However, in \cite[Theorem~9.5]{kennedy_schafhauser_noncommutative_crossed_products}, it is instead assumed that the action on $A$ is not properly outer \emph{and} that the system $(A,G)$ has \emph{vanishing obstruction}. We now briefly explain what this means and how they use it in order to prove the existence of a non-faithful pseudo-expectation on $A\rtimes_\lambda G$ (which then violates the intersection property for $(A,G)$, by Proposition~\ref{prop:intersection_property_iff_faithful}). Let $u_t \in U(I(A)p_t)$ be the unitaries arising out of the decompositions $I(A) = I(A) p_t \directsum I(A) (1-p_t)$ into inner and properly outer parts, as described in Theorem~\ref{thm:PropInnerDecoMonCom}.
	If it were possible to ``untwist'' these unitaries in such a way so that the map $t \mapsto u_t$ behaves essentially like a group homomorphism (or, more precisely, if this map is a \emph{partial representation} of $G$, see \cite[Definition~9.1]{Exel_book} or \cite[Definition~8.2]{kennedy_schafhauser_noncommutative_crossed_products}), then the proofs in the commutative case can be mimicked, and a new non-faithful equivariant pseudo-expectation $F'\colon A \rtimes_\lambda G \to I_G(A)$ can be defined by $F'(a \lambda_t) = a u_t$. 
	(One can also start with proper outerness on $I_G(A)$, and untwisting the resulting unitaries from the decomposition of $I_G(A)$ instead. In the case of non-amenable groups, doing this becomes mandatory, due to coefficients from $I(A)$ only defining pseudo-expectations on the universal crossed product $A \rtimes G$, and not the reduced $A\rtimes_\lambda G$. See the proof of \cite[Theorem~9.1]{kennedy_schafhauser_noncommutative_crossed_products}.)
	This ability to untwist the unitaries (vanishing obstruction) trivially holds when $A$ is commutative, but as mentioned before is a very strong property that frequently fails in the noncommutative case. Likewise, even if these unitaries $u_t$ arise out of the polar decomposition of elements $x_t \in I_G(A)$ satisfying the positive-definiteness condition mentioned in Proposition~\ref{prop:pseudoexpectation_coefficients} (i.e.\ $[x_{st^{-1}}]_{s,t \in \mathcal{F}} \geq 0$ for all finite subsets $\mathcal{F} \subseteq G$), there is no guarantee that the matrices $[u_{st^{-1}}]_{s,t \in \mathcal{F}}$ are positive, a necessary condition for defining the aforementioned pseudo-expectation.
	
	Let us now explain the different route taken in our proof: Starting with a triple $(p,u,r)$ as in the theorem, we obtain a set of pairs $\{(p_t,u_t)\}_{t\in C}$ as in Proposition~\ref{prop:equivariant_unitaries_to_single_unitary}, where $C$ denotes the (finite) conjugacy class of $r$. These do not necessarily allow us to construct a non-faithful pseudo-expectation for $A\rtimes_\lambda G$, as explained above and done in \cite{kennedy_schafhauser_noncommutative_crossed_products}. However, a direct application of Proposition~\ref{prop:crossed_product_center_coefficients} shows that the element $\sum_{t\in C}u_t^*\lambda_t$ lies in the center of $I(A)\rtimes_\lambda G$, but does not belong to the copy of $I(A)$. Using this fact, we then prove that $I(A)\rtimes_\lambda G$ is not prime in a strong way (see Lemma~\ref{lemma:stronglynonprime} for the precise statement), which allows to deduce that $(I(A),G)$ does not have the intersection property, and so neither does $(A,G)$ \cite[Theorem~3.2]{bryder_injective_envelopes}.
	
	We now start building up the technical results that are necessary in order to give a complete proof. The following proposition is very similar to \cite[Lemma~1.1]{hamana82_centre} (see also the proof of \cite[Theorem~7.1]{hamana81_regular_embeddings}), which is done in the context of having $B$ below being either $\closure{A}$, Hamana's regular monotone completion, or $I(A)$. For our purposes, we will usually consider $A \rtimes_\lambda G$ and $M(A,G)$ as the monotone complete algebra, which is in general neither isomorphic to $\closure{A \rtimes_\lambda G}$ nor to $I(A \rtimes_\lambda G)$ (this is easy to see in the case of $A = \C$ and $G$ abelian).
	
	\begin{proposition}
		\label{prop:ideal_sup_projection}
		Let $A \subseteq B$ be an inclusion of unital C*-algebras, and assume that $B$ is monotone complete. If $I \subseteq A$ is any ideal of $A$, we let $I^+_1 \defeq \setbuilder{x \in I}{\norm{x} < 1}$. Recalling that this is an upwards-directed set, we have that $\sup^B I^+_1$ is a projection in $A' \cap B$, which acts on $I$ like a unit.
	\end{proposition}
	
	\begin{proof}
		For convenience, we will denote $\sup^B I^+_1$ by $p$. First, we show that it commutes with all of $A$. To see this, note that any *-isomorphism $\alpha \colon B \to C$ between monotone complete C*-algebras will preserve supremums. In other words, if $(b_\lambda)_\lambda$ is a bounded increasing net of self-adjoint elements in $B$, then $(\alpha(b_\lambda))_\lambda$ is such a net in $C$, and
		\[ \sup_\lambda \alpha(b_\lambda) = \alpha( \sup_\lambda b_\lambda). \]
		In particular, if we let $u \in U(A)$ be any unitary element, then the *-automorphism $\alpha = \Ad u$ will also preserve the supremum of $I^+_1$. However, $x \mapsto uxu^*$ is a bijection on this set. Thus,
		\[ upu^* = u ( \sup{^B} I^+_1 ) u^* = \sup{^B} ( u I^+_1 u^* ) = \sup{^B} I^+_1 = p. \]
		Since $U(A)$ spans $A$, we conclude that $p \in A' \cap B$.
		
		To see that $p$ is a projection, note that for any $x \in I^+_1$, we have that $x^{1/2} \in I^+_1$ as well, and so $p \geq x^{1/2}$. Since $p$ and $x$ commute, we also have $p^2\geq x$.
		In other words, $p^2$ is an upper bound to $I^+_1$. But since $p \leq 1$, we have that $p^2 \leq p$. Given that $p$ was the \emph{least} upper bound, it follows that $p^2 = p$.
		
		Finally, we show that $p$ acts on $I$ like a unit. Let $j\in I_+^1$. Since $j\leq p$, we have
		\[(1-p)j(1-p)\leq (1-p)p(1-p)=0. \]
		This implies $(1-p)j=0$. Since $I^+_1$ spans $I$, we have $pj=j$, for every $j\in I$.
	\end{proof}

	It is interesting to note that in \cite[Theorem~7.1]{hamana81_regular_embeddings}, Hamana does not take a supremum of $I^+_1$ directly, but rather a supremum of the left projections of the elements in this set. It is possible that he wished for the result to be a projection, but at least in this earlier paper, overlooked the fact that $\sup I^+_1$ is itself always a projection. In \cite[Lemma~1.1]{hamana82_centre}, approximate units are used directly.
	
	\begin{lemma}[{\cite[Lemma~1.9 and Corollary~4.10]{hamana81_regular_embeddings}}]
		\label{lem:sup_bFbstar}
		Let $B$ be any C*-algebra, and let $\mathcal{F} \subseteq B$ be a bounded set of self-adjoint elements that admits a supremum. Then, for every $b\in B$, $\sup ( b \mathcal{F} b^* )$ exists, and
		\[ b ( \sup \mathcal{F} ) b^* = \sup ( b \mathcal{F} b^* ). \]
	\end{lemma}
	
	\begin{proposition}
		\label{prop:ideal_sup_orthogonality}
		Let $A \subseteq B$ be an inclusion of unital C*-algebras, where $B$ is monotone complete. Let $I,J \subseteq A$ be non-trivial ideals of $A$ with $I\cdot J = 0$. Let $p \coloneqq \sup^B I^+_1$ and $q \coloneqq \sup^B J^+_1$. Then $pq = 0$. Moreover, $p\cdot J=0$ and $q\cdot I=0$.
	\end{proposition}
	
	\begin{proof}
		Let $j \in J$ be any positive element. Using Lemma~\ref{lem:sup_bFbstar} and the fact that $p$ is a projection that commutes with $A$ (see Proposition~\ref{prop:ideal_sup_projection}), we have that
		\[ jp = j^{1/2} p j^{1/2} = j^{1/2} (\sup{^B} I^+_1) j^{1/2} = \sup{^B} (j^{1/2} I^+_1 j^{1/2}) = 0. \]
		Since positive elements span a C*-algebra, we have that $p\cdot J=0$. Similarly, $q\cdot I=0$.
		
		Applying Lemma~\ref{lem:sup_bFbstar} once again, we have
		\[ pqp = p (\sup{^B} J^+_1) p = \sup{^B} (p J^+_1 p) = 0, \]
		and thus $pq=0$.
	\end{proof}
	
	One of Hamana's goals with these sorts of results was to prove the equivalence between a C*-algebra $B$ being prime, and the regular monotone completion $\closure{B}$ being a factor, although the same proof directly works with the injective envelope $I(B)$ instead. One direction is clear from the above result. Orthogonal ideals $I$ and $J$ in $B$ give orthogonal projections in $I(B)$ that commute with $B$, and therefore lie in $Z(I(B))$ (see \cite[Corollary~4.3]{hamana79_injective_envelopes_cstaralg}).
	
	We will also need the following concrete computation that shows we may move back and forth between projections in $I(B)$ and ideals in $B$, in a certain sense. This is essentially just part of \cite[Lemma~1.3(i)]{hamana82_centre}, but Hamana phrases it in terms of injective envelopes of non-unital C*-algebras. To avoid referencing this theory just yet (and leave it contained to Section~\ref{sec:proper_outerness_on_A}), we give a short re-proof here.
	
	\begin{proposition}
		\label{prop:ConcreteSupProj}
		Let $B$ be a unital C*-algebra and let $r\in Z(I(B))$ be a projection. Then, for $J\coloneqq r\cdot I(B)\cap B$, one has $\sup^{I(B)}J_1^+=r$.
	\end{proposition}
	
	\begin{proof}
		Denote $p=\sup^{I(B)}J_1^+$. By Proposition~\ref{prop:ideal_sup_projection} and \cite[Corollary~4.3]{hamana79_injective_envelopes_cstaralg}, we have that $p\in I(B)\cap B'=Z(I(B))$.
		In addition, given that every $x \in J_1^+$ satisfies $x \leq r$, we have $p \leq r$.
		On the other hand, since $(r-p)\cdot I(B)\cap B\subseteq J$ and $p$ acts on $J$ like a unit (see Proposition~\ref{prop:ideal_sup_projection}), it follows that $p$ acts on $(r-p)\cdot I(B)\cap B$ as a unit. Since $p\perp (r-p)$, we conclude that $(r-p)\cdot I(B)\cap B=0$, and therefore also $(r-p)\cdot I(B)=0$ by essentiality of the inclusion $B\subseteq I(B)$. This finishes the proof.
	\end{proof}
	
	As mentioned in the beginning of this section, in the minimal setting, the proof of the main theorem of this section (Theorem~\ref{thm:mainSec4}) involves showing that $I_G(A) \rtimes_\lambda G$ is not prime and conclude that $I(A) \rtimes_\lambda G$ is not prime, since they share the same injective envelope. In the context of the intersection property, we need a stronger notion of ``not prime'' that plays well relative to the subalgebras $A$, $I(A)$, and $I_G(A)$.
	
	\begin{lemma}
		\label{lemma:stronglynonprime}
		Let $G$ be a discrete group acting on a unital C*-algebra $A$. Assume that the canonical inclusion $Z(I_G(A))^G\subseteq Z(I_G(A) \rtimes_\lambda G)$ is proper. Then there exist non-trivial ideals $J,K\subseteq I(A) \rtimes_\lambda G$ with the following properties:
		\begin{enumerate}
			\item $J\cdot K=0$;
			\item If $J\cdot a=0$ for some $a\in I(A)$, then $a=0$.
		\end{enumerate}
		
	\end{lemma}
	
	\begin{proof}
		View the commutative C*-algebras $Z(I_G(A))^G$ and $Z(I_G(A) \rtimes_\lambda G)$ as $C(X)$ and $C(Y)$, respectively. The inclusion $C(X) \subsetneqq C(Y)$ is dual to a surjective, non-injective, continuous map $q : Y \to X$. Choose distinct elements $y_1,y_2$ in $Y$ mapping to the same point $x \in X$ under $q$, i.e.\ $x = q(y_1) = q(y_2)$.
		Let $U_1,U_2\subseteq Y$ be open disjoint neighborhoods of $y_1$ and $y_2$, respectively, and let $z_1,z_2\in C(Y)$ be positive contractions such that $z_i$ is supported on $U_i$ and $z_i(y_i)=1$, for $i=1,2$. Clearly, $z_1z_2=0$. Moreover, we claim that $z_i\notin C(X)$, for $i=1,2$. Indeed, the inclusion $C(X)\subseteq C(Y)$ is given by viewing a function $f\in C(X)$ as a function of $C(Y)$ which is constant (and equal to $f(x)$) on the fiber $q^{-1}(\{x\})$, for every $x\in X$. However, the functions $z_i\in C(Y)$, $i=1,2$, are clearly not constant along the fiber $q^{-1}(\{x\})$, and therefore do not belong to the copy of $C(X)$. 
		
		Note that $E(z_1)\in C(X)$, since the canonical conditional expectation $E\colon I_G(A) \rtimes_\lambda G\to I_G(A)$ maps $Z(I_G(A) \rtimes_\lambda G)$ onto $Z(I_G(A))^G$. Moreover, $Z(I_G(A))^G$ is a monotone complete C*-algebra, and so $X$ is an extremally disconnected space (see the discussion at the end of Section~\ref{sec:preliminaries:monotone_complete}). This implies that
		\[ \closure{\supp (E(z_1))} = \closure{\setbuilder{x \in X}{E(z_1)(x) \neq 0}} \] 
		is a clopen subset of $X$. Let $p_1\in C(X)$ be the characteristic function of $\overline{\supp(E(z_1))}$ (in other words, $p_1$ is the support projection of $E(z_1)$). Set $w_1\coloneqq (1-p_1)+z_1$ and $w_2\coloneqq z_2p_1$. We will show the following properties:
		\begin{enumerate}
			\item $w_1, w_2\in Z(I_G(A) \rtimes_\lambda G)$ are non-zero orthogonal positive contractions.
			\item If $w_1a=0$ for some $a\in I_G(A)$, then $a=0$.
		\end{enumerate}
		
		It is immediate that $w_1, w_2\in Z(I_G(A) \rtimes_\lambda G)$ are positive contractions and that $w_1\neq 0$. We have to show that $w_2\neq 0$. For this, we first observe that $p_1z_1=z_1$. This is due the fact that $E((1-p_1)z_1)=(1-p_1)E(z_1)=0$, and since $E$ is faithful, we get $(1-p_1)z_1=0$. By definition, we have that $z_1(y_1)=1$, and thus we must have that $p_1(y_1)=1$. As $p_1\in C(X)$ and $q(y_1)=q(y_2)$, we see that $p_1(y_1)=p_1(y_2)=1$ under the embedding $C(X) \subseteq C(Y)$. Now, $w_2(y_2)=z_2(y_2)p_1(y_2)=1$, so $w_2\neq 0$.
		Orthogonality of $w_1$ and $w_2$ follows from orthogonality of $z_1$ and $z_2$.
		
		Finally, consider the set
		\[ I = \setbuilder{a\in I_G(A)}{w_1a=0}. \]
		Using that $w_1\in Z(I_G(A) \rtimes_\lambda G)$, it is easy to check that $I$ is a closed two-sided $G$-invariant ideal of $I_G(A)$. Let $q=\sup^{I_G(A)}I_1^+$. By Proposition~\ref{prop:ideal_sup_projection}, we know that $q$ is a projection inside $Z(I_G(A))$, which is also $G$-invariant since
		\[ g \cdot q = g \cdot (\sup I_1^+) = \sup(g \cdot I_1^+) = \sup I_1^+ = q. \]
		By Lemma~\ref{lem:sup_bFbstar}, we have $w_1qw_1=w_1(\sup I_1^+)w_1=\sup(w_1I_1^+w_1)=0$. Thus $w_1q=0$.
		This in turn implies that $E(w_1q)=0$, or, equivalently, $E(w_1)q=0$, which means that $((1-p_1)+E(z_1))q=0$ inside $Z(I_G(A))^G = C(X)$. However, by definition of $p_1$, we see that the function $(1-p_1)+E(z_1)$ is non-vanishing on 
		\[ \supp(E(z_1)) \cup (X \setminus \closure{\supp(E(z_1))}), \]
		which is a dense subset of $X$. This forces $q$ to be the zero element, and consequently, $I=0$ as required.
		
		Next, recalling that $I(I_G(A)\rtimes_\lambda G)=I(A\rtimes_\lambda G)$ (see Theorem~\ref{thm:crossed_product_injective_envelope_inclusions} and Proposition~\ref{prop:intermediate_algebra_shares_injective_envelope}) and that $Z(B)\subseteq Z(I(B))$ for every unital C*-algebra $B$ \cite[Corollary~4.3]{hamana79_injective_envelopes_cstaralg}, we have that $Z(I_G(A)\rtimes_\lambda G) \subseteq Z(I(A\rtimes_\lambda G))$. This allows us to view $w_1,w_2$ as elements of $Z(I(A\rtimes_\lambda G))$, which is again monotone complete. Denote by $r_1$ and $r_2$ the right (or equivalently left) projections of $w_1$ and $w_2$, respectively, inside of this commutative and monotone complete C*-algebra (recall the definition of these projections from Proposition~\ref{prop:monotonecompletePolarDec}). The following properties hold:
		
		\begin{enumerate}
			\item $r_1,r_2\in Z(I(A\rtimes_\lambda G))$ are non-zero orthogonal projections.
			\item If $r_1a=0$ for some $a\in I(A)\subseteq I(A\rtimes_\lambda G)$, then $a=0$.
		\end{enumerate}
		
		The first property follows from the fact that, because $w_1 w_2 = 0$, then $w_1 r_2 = 0$, and so $r_1 r_2 = 0$ as well. The second property is also not hard to see, as if $r_1 a = 0$ for some $a \in I(A)$, then
		\[ a^* w_1^* w_1 a \leq a^* r_1 a = 0, \]
		and so $w_1 a = 0$, which in turn we know implies $a = 0$.
		
		Now we may finally construct the ideals that we are after. Define
		\begin{align*}
		&J \defeq (r_1 \cdot I(A \rtimes_\lambda G)) \cap (I(A) \rtimes_\lambda G), \\
		&K \defeq (r_2 \cdot I(A \rtimes_\lambda G)) \cap (I(A) \rtimes_\lambda G).
		\end{align*}
		By essentiality of the inclusion $I(A)\rtimes_\lambda G\subseteq I(A\rtimes_\lambda G)$, it follows that $J$ and $K$ are non-zero ideals inside $I(A) \rtimes_\lambda G$, and they are clearly orthogonal as well. Let $a \in I(A)$ be given, and assume that $ja=0$ for every $j\in J$. By Proposition~\ref{prop:ConcreteSupProj}, we have that $\sup^{I(A\rtimes_\lambda G)} J_1^+ = r_1$, and so $a^*r_1a=0$ by Lemma~\ref{lem:sup_bFbstar}, and so $r_1 a = 0$. Due to condition~(2) above, we obtain $a=0$. In summary, if $J\cdot a=0$ for some $a\in I(A)$, then $a=0$.		
	\end{proof}
	
	\begin{remark}
		\label{remark:stronglynonprimeinjective}
		Consider the statement of Lemma~\ref{lemma:stronglynonprime}, and assume that instead of requiring the inclusion $Z(I_G(A))^G \subseteq Z(I_G(A) \rtimes_\lambda G)$ to be proper, we required $Z(I(A))^G \subseteq Z(I(A) \rtimes_\lambda G)$ to be proper. This would automatically imply that the former inclusion is proper as well. Indeed, assume $x \in Z(I(A) \rtimes_\lambda G) \setminus Z(I(A))^G$. It is clear that $x \notin Z(I_G(A))^G$. Moreover, if we consider the inclusions
		\[ A \rtimes_\lambda G \subseteq I(A) \rtimes_\lambda G \subseteq I_G(A) \rtimes_\lambda G \subseteq I(A \rtimes_\lambda G) \]
		from Theorem~\ref{thm:crossed_product_injective_envelope_inclusions}, then we see that any element of $Z(I(A) \rtimes_\lambda G)$ in particular commutes with $A \rtimes_\lambda G$, and therefore with all of $I(A \rtimes_\lambda G)$ by \cite[Corollary~4.3]{hamana79_injective_envelopes_cstaralg}. Hence, it must lie in $Z(I_G(A) \rtimes_\lambda G)$.
	\end{remark}

	\begin{theorem}
		\label{thm:mainSec4}
		Let $G$ be an FC-group acting on a unital C*-algebra $A$. Then $(A,G)$ does not have the intersection property if and only if there exist $r\in G\setminus\{e\}$, a non-zero $r$-invariant central projection $p\in I(A)$ and a unitary $u\in U(I(A)p)$ such that
		\begin{enumerate}
			\item $r$ acts by $\Ad u$ on $I(A)p$;
			\item $s \cdot p = p$ and $s\cdot u=u$ for all $s\in C_G(r)$.
		\end{enumerate}
	\end{theorem}
	
	\begin{proof}
		First, assume that there exists a triple $(p,u,r)$ satisfying properties $(1)$ and $(2)$.
		By Proposition~\ref{prop:equivariant_unitaries_to_single_unitary}, we obtain well-defined pairs of elements $\{(p_t,u_t)\}_{t\in C}$, where $C=\setbuilder{grg^{-1}}{g \in G}$ denotes the conjugacy class of $r$, each $p_t$ is a $t$-invariant central projection in $I(A)$, each $u_t\in I(A)$ is a unitary element in the corner $I(A)p_t$ such that $t$ acts by $\Ad u_t$ on this corner, and moreover $s \cdot u_t=u_{sts^{-1}}$ for all $s\in G$ and $t\in C$.
		
		Letting $u_g=0$ for every $g\in G\setminus C$, we check that the elements $(u_g)_{g\in G}$ satisfy the conditions stated in Proposition~\ref{prop:crossed_product_center_coefficients}. Indeed, for every $g\in G\setminus C$ and every $y\in I(A)$, it is clear that $u_gy=(g\cdot y)u_g$. However, if $g\in C$, we have that 
		\[ u_gy =u_gp_gy = (u_g(p_gy)u_g^*)u_g = (g \cdot (p_g y)) u_g = p_g (g\cdot y) u_g = (g\cdot y)u_g, \]
		and so condition~(1) of Proposition~\ref{prop:crossed_product_center_coefficients} is satisfied, and condition~(2) is immediate by construction.
		Therefore,
		\[z\coloneqq\sum\limits_{c\in C} u_c^*\lambda_c\in Z(I(A)\rtimes_\lambda G). \]
		However, $z\notin Z(I(A))^G$, since $u_r=u$ and $r\neq e$. Hence, $Z(I(A))^G\subseteq Z(I(A)\rtimes_\lambda G)$ is a proper inclusion and we can apply Lemma~\ref{lemma:stronglynonprime} together with Remark~\ref{remark:stronglynonprimeinjective} in order to obtain non-trivial orthogonal ideals $J,K\subseteq I(A)\rtimes_\lambda G$ with the property that $J^\perp\cap I(A)=\{0\}$. In particular, $K$ is a non-trivial ideal of $I(A)\rtimes_\lambda G$ which intersects $I(A)$ trivially. This violates the intersection property for $(I(A),G)$, and thus for $(A,G)$ as well by \cite[Theorem~3.2]{bryder_injective_envelopes}.

		Conversely, assume that $(A,G)$ does not have the intersection property, and let $F\colon A\rtimes_\lambda G\to I_G(A)$ be a non-canonical pseudo expectation, which exists by Proposition~\ref{prop:intersection_property_iff_faithful}. The map $F$ gives us non-trivial coefficients $(x_t)_{t \in G}$ in $I_G(A)$ by letting $x_t\coloneqq F(\lambda_t)$, for every $t\in G$. These satisfy the properties listed in Proposition~\ref{prop:pseudoexpectation_coefficients}. Let $r\in G \setminus \set{e}$ be such that $x_r \neq 0$ and denote by $C \subseteq G$ its conjugacy class. We obtain that $\sum_{t \in C} x_t^* \lambda_t$ lies in $Z(I_G(A) \rtimes_\lambda G)$ by Proposition~\ref{prop:crossed_product_center_coefficients}, and so we conclude that the canonical inclusion
		\[Z(I_G(A))^G\subseteq Z(I_G(A) \rtimes_\lambda G) \]
		is proper.
		
		Applying Lemma~\ref{lemma:stronglynonprime}, we obtain non-trivial orthogonal ideals $J,K\subseteq I(A) \rtimes_\lambda G$ with the property that whenever $J\cdot a=0$, for some $a\in I(A)$, then $a=0$. Denote $q \defeq \sup^{M(I(A),G)}K_1^+$. By Proposition~\ref{prop:ideal_sup_projection}, we have that $q$ is a non-zero projection which commutes with the copy of $I(A) \rtimes_\lambda G \subseteq M(I(A),G)$. 
		Moreover, $J\cdot q=0$, by Proposition~\ref{prop:ideal_sup_orthogonality}, which implies that $q\notin I(A)$.
		
		Let $(q_t)_{t\in G}$ denote the coefficients in $I(A)$ of $q = \sum_{t \in G} q_t \lambda_t \in M(I(A),G)$, and let $z_t \defeq q_t^*$, for every $t\in G$. By Proposition~\ref{prop:mc_crossed_product_commutant_coefficients}, we have that $z_t y =(t\cdot y) z_t$, for all $t \in G$ and all $y \in I(A)$, and moreover that $s\cdot z_t=z_{sts^{-1}}$, for all $s,t\in G$.
		
		Given that $q \notin I(A)$, at least one coefficient $z_r$ is non-zero for some $r \in G \setminus \set{e}$. Letting $z_r=u\abs{z_r}$ be the polar decomposition of this element, and $p=u^*u$, we conclude by Proposition~\ref{prop:properly_outer_equivariant_polar_decompositions} and Proposition~\ref{prop:equivariant_unitaries_to_single_unitary} that the triple $(p,u,r)$ satisfies the desired properties.
	\end{proof}

	Since, under minimality of the action of $G$ on $A$, the intersection property is equivalent to simplicity of the associated crossed product, the theorem above also gives a characterization of simplicity of crossed products by FC-groups. Moreover, in general (not assuming the action is minimal), we almost immediately obtain a characterization for when crossed products by FC-groups are prime, and we state it separately below. This is in slight contrast to what is done in Section~\ref{sec:primality_minimal}, which characterizes exactly when a crossed product $A \rtimes_\lambda G$ is prime \emph{regardless} of what $G$ is, but assuming the action is minimal. 
	
	Recall that a C*-algebra $A$ equipped with a $G$-action is called \emph{$G$-prime} if whenever $I,J\subseteq A$ are non-trivial $G$-invariant ideals in $A$, their product $I\cdot J$ is non-trivial.

	\begin{theorem}\label{thm:mainSec4Prime}
		Let $G$ be an FC-group acting on a unital, $G$-prime C*-algebra $A$. Then $A\rtimes_\lambda G$ is not prime if and only if there exist $r\in G\setminus\{e\}$, a non-zero $r$-invariant central projection $p\in I(A)$ and a unitary $u\in U(I(A)p)$ such that
		\begin{enumerate}
			\item $r$ acts by $\Ad u$ on $I(A)p$;
			\item $s \cdot p = p$ and $s\cdot u=u$, for all $s\in C_G(r)$.
		\end{enumerate}
	\end{theorem}
	
	\begin{proof}
		Assume that there is a triple $(p,u,r)$ satisfying the properties listed in the Theorem. Exactly as in the proof of Theorem~\ref{thm:mainSec4}, the triple gives rise to a central element of $I(A)\rtimes_\lambda G$ which does not lie in $Z(I(A))^G$. In particular, such an element is non-scalar, and so $I(A)\rtimes_\lambda G$ is not prime. By Proposition~\ref{proposition:primeifffactor}, $A\rtimes_\lambda G$ is not prime either.
		
		Conversely, assume that $A\rtimes_\lambda G$ is not prime. We claim that the intersection property for $(A,G)$ cannot hold. Indeed, since $A\rtimes_\lambda G$ is not prime, we can find two non-trivial orthogonal ideals $J,K\subseteq A\rtimes_\lambda G$. As $J\cap A$ and $K\cap A$ are $G$-invariant orthogonal ideals in $A$, and $A$ is $G$-prime, we have that at least one of them must be the zero ideal. In other words, at least one of $J$ or $K$ violates the intersection property for $(A,G)$. By Theorem~\ref{thm:mainSec4}, we obtain a triple $(p,u,r)$ which satisfies the desired properties.		
	\end{proof}
	
	\begin{remark}
		Theorem~\ref{thm:mainSec4Prime} had already been proven in \cite[Theorem~10.1]{hamana85-injective_envelopes_equivariant} in the special case of finite groups.
	\end{remark}

	\section{Primality for minimal actions and simplicity for FC-hypercentral groups}
	\label{sec:primality_minimal}
	
	This section is essentially an observation that most of the arguments in the previous section, which deal with characterizing when a reduced crossed product has the intersection property (a generalization of simplicity), actually also apply in the context of characterizing when a reduced crossed product is prime (Theorem~\ref{thmIntro:C}), for all minimal actions of any discrete groups. Afterwards, we show via an application of a result of Echterhoff from \cite{echterhoff_jot} that Theorem~\ref{thmIntro:B}, which characterizes simplicity of $A \rtimes_\lambda G$ in the case of FC-hypercentral $G$, immediately follows.
	
	First, we require a seemingly arbitrary lemma with quite a convoluted proof, and it will have its usefulness become clear near the end of the proof of the main theorem. The proof of the lemma makes use of the basic properties of the Furstenberg boundary $\boundary_F G$. This was originally introduced by Furstenberg several decades ago in \cite{furstenberg_proceedings} (see also \cite{furstenberg_poisson}), as a topological boundary to be used largely in the study of Lie groups. It was more recently studied in the C*-simplicity results \cite{kalantar_kennedy_boundaries} and \cite{breuillard_kalantar_kennedy_ozawa_c_simplicity}, where in particular it was observed that for discrete groups, $I_G(\mathbb{C}) \isoto C(\partial_F G)$ as $G$-C*-algebras. If $G$ is a discrete group, the Furstenberg boundary $\boundary_F G$ is the universal compact Hausdorff $G$-space with the following properties:
	
	\begin{enumerate}
		\item It is \emph{minimal}.
		\item It is \emph{strongly proximal}, in the sense that for any probability measure $\nu \in P(\boundary_F G)$, there is a net of group elements $(g_\lambda)_\lambda$ such that $\lim_\lambda g_\lambda \nu = \delta_x$, for some $x \in X$.
	\end{enumerate}
	
	As noted in Proposition~\ref{prop:injective_is_mc_and_ginjective_is_injective}, $I_G(\mathbb{C})$ is monotone complete, and hence $\partial_F G$ is an extremally disconnected space (see, for example, \cite[Theorem~2.3.7]{SWMonotoneComplete}).
	
	Before we proceed with the proof of the lemma, it might also interest the reader to realize that in the case of amenable groups, we are guaranteed the existence of an invariant state on any unital C*-algebra, and the proof of the lemma becomes extremely easy. The significant difficulty in the non-amenable case is that now, we are only guaranteed the existence of $G$-equivariant unital completely positive maps into $I_G(\C)$.
	
	\begin{lemma}
		\label{lem:infinitely_many_distinct_translates_unbounded_sum}
		Let $G$ be a discrete group acting minimally on a compact Hausdorff space $X$, and let $f \in C(X)$ be a non-zero positive function. Consider the set
		\[ \setbuilder{g \cdot f}{g \in G}, \]
		without counting repetition among the elements.
		If this set is infinite, then the sum of its elements cannot be uniformly bounded from above, in the sense that there cannot exist some scalar $k \geq 0$ such that every sum of finitely many elements is bounded by $k$ from above.
	\end{lemma}

	\begin{proof}
		This proof will, somewhat interestingly, first assume that $G$ is countable, and from there deduce the case of uncountable groups.
		
		For convenience, let $W \subseteq G$ be a choice of elements of $G$ so that
		\[ \setbuilder{g \cdot f}{g \in W} = \setbuilder{g \cdot f}{g \in G}, \]
		and $g \cdot f$ are distinct for distinct values of $g \in W$. We will do a proof by contradiction, and start with the assumption that $\sum_{g \in W} g \cdot f$ is uniformly bounded.        
		
		Since $I_G(\mathbb{C})$ is $G$-injective and isomorphic to $C(\partial_F G)$, we know that there is at least one $G$-equivariant unital completely positive map $\phi \colon C(X) \to C(\boundary_F G)$. We also know that the left-ideal
		\[ \setbuilder{h \in C(X)}{\phi(h^*h) = 0} \]
		is a two-sided $G$-invariant ideal in $C(X)$. Since $1$ does not lie in this ideal, by minimality, this ideal must necessarily be zero. In other words, the map $\phi$ is faithful.
		
		We know that $\phi(f) \neq 0$, by faithfulness. We now divide the proof into two cases:
		
		If it were the case that $\phi(f) \geq \delta$ for some $\delta > 0$, then the sum $\sum_{g \in W} g \cdot f$ could not be bounded, as neither could $\sum_{g \in W} \phi(g \cdot f)$.
		
		Assume therefore that $\phi(f)$ is not bounded from below. From here, choose some fixed $\delta > 0$ so that $\phi(f)(y) > \delta$ for some $y \in \boundary_F G$. Let $U \coloneqq \phi(f)^{-1}(\delta,\infty)$. Given that $\boundary_F G$ is extremally disconnected, we have that $\closure{U}$ is a clopen subset, and $\phi(f)(z) \geq \delta$ for every $z\in \closure{U}$, while $\phi(f)(z) \leq \delta$ for every $z\in \partial_FG\setminus \closure{U}$. Note that, by assumption of $\phi(f)$ not being bounded below by any strictly positive scalar, we have $\closure{U} \neq \partial_FG$. For convenience, we will also let $p \in C(\boundary_F G)$ be the non-trivial projection coming from this clopen subset, so that $\phi(f) \geq \delta \cdot p$.
		
		Another observation we make is that if $g_1 \cdot p \neq g_2 \cdot p$, this implies that $g_1 U \neq g_2 U$, which are exactly the set of points on which $g_1 \cdot \phi(f)$ and $g_2 \cdot \phi(f)$ are strictly greater than $\delta$, respectively. Consequently, it must be the case that $g_1 \cdot \phi(f) \neq g_2 \cdot \phi(f)$. 
		In other words, distinct translates of $p$ give rise to distinct translates of $\phi(f)$. Thus, if we show that the sum of the distinct translates of $p$ is unbounded, then, since $\phi(f) \geq \delta \cdot p$, this will automatically imply the same for the distinct translates of $\phi(f)$.
		
		Since $G$ is countable, the Furstenberg boundary $\boundary_F G$ is separable. This is a direct consequence of minimality of the space. Let $(y_n)_{n=1}^\infty$ be a dense subset of this space. Observe that there cannot exist any $g \in G$ with the property that $g \cdot y_n \in \closure{U}$ for all $n\in\mathbb{N}$. Otherwise, we would have $y_n \in g^{-1} \closure{U}$ for all $n\in\mathbb{N}$, implying $g^{-1} \closure{U} = \boundary_F G$ by density. This is impossible, since we have already shown $\closure{U} \neq \boundary_F G$.
		
		Thus, if we construct a probability measure $\omega \in P(\boundary_F G)$ by $\omega = \sum_{n=1}^\infty \frac{1}{2^n} \delta_{y_n}$ (which we will also view as a state on $C(\boundary_F G)$), the previous paragraph essentially tells us that $(g\cdot \omega)(p)$ is never exactly equal to $1$ for any $g \in G$. However, choosing any $z \in \closure{U}$, we know that since $\boundary_F G$ is a strongly proximal, there is a net $(g_\lambda)_\lambda$ in $G$ such that $\lim_\lambda g_\lambda \cdot \omega = \delta_z$. In other words, even though $(g\cdot \omega)(p) \neq 1$ for any $g\in G$, we can still find infinitely many elements $g\in G$ for which $(g\cdot \omega)(p) = \omega(g^{-1} \cdot p)$ is arbitrarily close to $1$ and, say, greater than $\frac{1}{2}$. Thus, we have infinitely many $g\in G$ for which $g^{-1} p$ are distinct, and $\omega(g^{-1} \cdot p) \geq \frac{1}{2}$.
		
		Consequently, the sum of the distinct translates of $p$ cannot be bounded. By what was mentioned earlier, the same must be true for the corresponding distinct translates of $\phi(f)$, and thus the distinct translates of $f \in C(X)$.
		
		The case of uncountable groups is not that hard to deduce from the countable case. Assume $G$ is uncountable, and consider our starting function $f \in C(X)$. Let $U \subseteq X$ be a nonempty open set such that $f(x) > 0$ for all $x \in U$. As $\cup_{g\in G} g\cdot U$ is an open invariant set, it must cover $X$ by minimality. Compactness then tells us that there is a finite collection $g_1 U, \dots, g_n U$ covering $X$. In terms of our function translates, this tells us that $g_1 \cdot f, \dots, g_n \cdot f$ have the property that for any $x \in X$, there is some $i\in \{1,\ldots,n\}$ with $(g_i \cdot f)(x) > 0$.
		
		By assumption, we may also choose countably many distinct group elements $(k_i)_{i=1}^\infty$ such that the translates $k_i \cdot f$ are all distinct. Now let $H$ be the necessarily countable subgroup of $G$ generated by $g_1, \dots, g_n$ and $(k_i)_{i=1}^\infty$. Although the restricted action of $H$ on $X$ is not necessarily minimal, we can pass to a minimal $H$-subsystem $Y\subseteq X$ by Zorn's lemma. Consider the restriction map $\pi \colon C(X) \to C(Y)$. The function $\pi(f)$ is non-zero, by construction.
		
		If the set $\setbuilder{h \cdot \pi(f)}{h \in H}$ is finite, then by the pigeonhole principle, because $\setbuilder{h \cdot f}{h \in H}$ is infinite, there are infinitely many distinct translates $h \cdot f$ mapping to $h_0 \cdot \pi(f)$, for some $h_0\in H$. Thus, the sum of the distinct translates $h \cdot f$ mapping to this value is unbounded, as under the image of $\pi$, the sum of infinitely many copies of $h_0 \cdot \pi(f)$ is certainly unbounded.
		
		If the set $\setbuilder{h \cdot \pi(f)}{h \in H}$ is infinite, then we may simply apply the lemma we have already proven in the countable setting to deduce that the sum of the infinitely many distinct translates $h \cdot \pi(f)$ in $C(Y)$ is unbounded. Thus, the sum of the infinitely many distinct translates in the original space $C(X)$, i.e.\ the sum of the distinct $h \cdot f$, is certainly unbounded as well.
	\end{proof}

    \begin{lemma}
    \label{lem:minimal_central_element_FC_support}
        Let $G$ be a discrete group acting minimally on a unital C*-algebra $A$, and assume $z = \sum z_t^*\lambda_t$ is an element of the center of $M(I(A),G)$. If $z_t \neq 0$ for some $t \in G$, then $t$ necessarily has finite conjugacy class.
    \end{lemma}

    \begin{proof}
        To see this, let $E : M(I(A),G) \to I(A)$ be the canonical conditional expectation, and consider
		\[ E(zz^*) = O-\sum_{t \in G} z_t^* z_t, \]
		so that the net of finite sums is order-convergent to a concrete element in $I(A)$. The precise details of order convergence are not important. What is important is that this in particular implies that
		\[\sum_{t \in \mathcal{F}} z_t^* z_t\leq E(zz^*), \]
		for every finite subset $\mathcal{F}\subseteq G$ (see \cite[Lemma~1.2.(iv)]{hamana82_mc_tensor_products_I}).
		
		We recall from Proposition~\ref{prop:mc_crossed_product_commutant_coefficients} that $z_t y = (t \cdot y) z_t$, for all $y \in I(A)$. Consequently,
		\[ z_t^* z_t y = z_t^* (t \cdot y) x_t = y z_t^* z_t, \]
		for all $y \in I(A)$, or, in other words, $z_t^*z_t$ lies in $Z(I(A))$.  Proposition~\ref{prop:mc_crossed_product_commutant_coefficients} also tells us that $s \cdot z_t = z_{sts^{-1}}$, and so $s \cdot (z_t^*z_t) = z_{sts^{-1}}^*z_{sts^{-1}}$ for all $s,t\in G$.
		
		Let us restrict our attention to an \emph{infinite} conjugacy class $C \subseteq G$, and show that the coefficients $z_c$ must be zero for every $c\in C$. Assume otherwise, so that one, hence all, of these coefficients are non-zero. If $\setbuilder{z_c^*z_c}{c \in C}$ is a finite set, then by the pigeonhole principle, there are infinitely many $c \in C$ for which $z_c^*z_c$ are all the same value. Given that $z_c^*z_c \neq 0$, this clearly contradicts the assumption that the sum $\sum_{t \in G} z_t^*z_t$ is uniformly bounded above.
		Hence, we are left to conclude that $\setbuilder{z_c^*z_c}{c \in C}$ is an infinite set. Equivalently, if we choose a fixed $c_0 \in C$, the set
		\[ \setbuilder{g \cdot (z_{c_0}^*z_{c_0})}{g \in G} \]
		has infinitely many distinct elements in $Z(I(A))$. By Lemma~\ref{lem:infinitely_many_distinct_translates_unbounded_sum}, these once again cannot have a bounded sum, as $Z(I(A))$ is minimal by Proposition~\ref{prop:minimality_transfers_to_injective_envelopes_and_centers}. Therefore $z_c = 0$.
    \end{proof}
	\begin{theorem}\label{thm:DiscreteMinimal}
		Let $G$ be a discrete group acting minimally on a unital C*-algebra $A$. Then $A \rtimes_\lambda G$ is not prime if and only if there exist $r \in \FC(G) \setminus \set{e}$, a non-zero $r$-invariant central projection $p \in I(A)$, and a unitary $u \in U(I(A)p)$ such that
		\begin{enumerate}
			\item $r$ acts by $\Ad u$ on $I(A)p$;
			\item $s \cdot p = p$ and $s \cdot u = u$ for all $s \in C_G(r)$.
		\end{enumerate}
	\end{theorem}
	
	\begin{proof}
		First, if there is a triple $(p,u,r)$ satisfying the properties listed above, then exactly as in the proof of Theorem~\ref{thm:mainSec4Prime} (see also Theorem~\ref{thm:mainSec4}), we conclude that $A\rtimes_\lambda G$ is not prime.
		
		Conversely, assume that $A \rtimes_\lambda G$ is not prime. Hence neither is $I(A) \rtimes_\lambda G$ (by Proposition~\ref{proposition:primeifffactor}). Choose non-trivial orthogonal ideals $J,K \subseteq I(A) \rtimes_\lambda G$, and let \[q\coloneqq \sup {}^{M(I(A),G)}K_1^+.\] By Proposition~\ref{prop:ideal_sup_projection} and Proposition~\ref{prop:ideal_sup_orthogonality}, we have that $q$ is a non-zero projection which commutes with the copy of $I(A) \rtimes_\lambda G \subseteq M(I(A),G)$, and satisfies $J\cdot q=0$. Moreover, $q\notin I(A)$, as otherwise $q$ would be an element of $Z(I(A))^G$. This is impossible, since the induced system $(Z(I(A)),G)$ is minimal (see Proposition~\ref{prop:minimality_transfers_to_injective_envelopes_and_centers}), and so $Z(I(A))^G=\mathbb{C}$.
		
		Letting $q = \sum_{t\in G}z_t^*\lambda_t$ as an element of $M(I(A),G)$, we know that there exists $r\in G\setminus \{e\}$ so that $z_r\neq 0$. By Lemma~\ref{lem:minimal_central_element_FC_support}, $r$ necessarily has finite conjugacy class. The proof now proceeds exactly as in Theorem~\ref{thm:mainSec4}. More precisely, if $z_r=u\abs{z_r}$ is the polar decomposition of $z_r$, with $p=u^*u$, then the triple $(p,u,r)$ has the desired properties.
	\end{proof}

    Now we state the following immediate, and somewhat surprising, corollary to Theorem~\ref{thm:DiscreteMinimal}.
	
	\begin{corollary}
		\label{cor:ICC}
		If $G$ is a discrete ICC group acting minimally on a unital C*-algebra $A$, then the reduced crossed product $A\rtimes_\lambda G$ is prime.
	\end{corollary}
	
	We remark that the hypotheses that the action is minimal and $G$ is ICC are far from being necessary. For example, by \cite[Proposition~2.8]{MatuiRor}, the action of $\mathbb{Z}$ on its one-point compactification gives rise to a prime crossed product.
	
	Finally, we want to take the ideas from Section~\ref{sec:intersection_property_fc}, and show that at least in the minimal setting, they can be generalized to the case of FC-hypercentral groups.
	After the first public release of this preprint, Siegfried Echterhoff kindly pointed out to us that Theorem~\ref{thm:mainFCHSimpleMinimal} immediately follows from a combination of the main result of his paper \cite{echterhoff_jot} together with Theorem~\ref{thm:DiscreteMinimal}. Previously, this was done using different techniques inspired by \cite{bedos_omland_fc_hypercentral_simplicity}, generalized to the context of pseudo-expectations of crossed products.

    \begin{lemma}[Special case of {\cite[Theorem~3.1]{echterhoff_jot}} or {\cite[Satz~5.3.1]{echterhoff_thesis}}]
		\label{lem:fc_simplicity_equivalence}
		Let $G$ be an FC-hypercentral group acting minimally on a unital C*-algebra $A$. Then
		$A \rtimes_\lambda G$ is simple if and only if $A \rtimes_\lambda G$ is prime.  
	\end{lemma}
	
	\begin{theorem}\label{thm:mainFCHSimpleMinimal}
		Let $G$ be an FC-hypercentral group acting minimally on a unital C*-algebra $A$. Then $A \rtimes_\lambda G$ is not simple if and only if there exist $r \in \FC(G) \setminus \set{e}$, a non-zero $r$-invariant central projection $p \in I(A)$, and a unitary $u \in U(I(A)p)$ such that
		\begin{enumerate}
			\item $r$ acts by $\Ad u$ on $I(A)p$;
			\item $s \cdot p = p$ and $s \cdot u = u$ for all $s \in C_G(r)$.
		\end{enumerate}
	\end{theorem}

	\section{From the injective envelope to the original C*-algebra}
	\label{sec:proper_outerness_on_A}
	
	The first dynamical characterizations of simplicity and primality that we obtain in our paper are initially written in terms of the dynamics of $G$ on $I(A)$, the injective envelope of $A$. While this gives the most elegant characterizations from a theory perspective, the injective envelope is still a somewhat mysterious object that is not that easy to describe concretely in many cases. Strictly speaking, it was also possible in all of our results to simply use Hamana's \emph{regular monotone completion} $\overline{A}$, which is significantly smaller in many cases, but at the end of the day it suffers from the same problem of still being relatively difficult to write down concretely. It would be desirable to have results which relate back to the dynamics on the original C*-algebra $A$.
	
	It was mentioned in the introduction that the notion of a \emph{properly outer} automorphism plays an important role in the study of simplicity of crossed products, and indeed, our results are an equivariant version of this notion. Recall the classical, non-equivariant notion first, which was discussed in Section~\ref{sec:preliminaries:properly_outer} (especially Theorem~\ref{thm:properly_outer_equivalence}, which we state again for convenience). Let $A$ be a unital C*-algebra, and let $\alpha \in \Aut(A)$, with its unique extension to $I(A)$ also denoted by $\alpha$. Consider the following conditions:
	\begin{enumerate}
		\item
		There is a non-zero $\alpha$-invariant ideal $J \subseteq A$ such that $\Gamma_B(\alpha|_J) = \set{e}$.
		
		\item
		There is a non-zero $\alpha$-invariant central projection $p \in I(A)$, and a unitary $u \in U(I(A)p)$ such that $\alpha$ acts by $\Ad u$ on $I(A)p$.
		
		\item
		There is a non-zero $\alpha$-invariant ideal $J \subseteq A$ and a unitary $u \in M(J)$, such that $\norm{\alpha|_J - (\Ad u)|_J} < 2$.
		
		\item
		Given any $\varepsilon\in (0,2)$, there exists a non-zero $\alpha$-invariant ideal $J \subseteq A$ and a unitary $u \in M(J)$ such that $\norm{\alpha|_J - (\Ad u)|_J} < \varepsilon$.
	\end{enumerate}
	We have that (\ref{thm:properly_outer_equivalence:borchers}) and (\ref{thm:properly_outer_equivalence:injective_envelope}) are equivalent, and both are implied by (\ref{thm:properly_outer_equivalence:elliott}) and (\ref{thm:properly_outer_equivalence:elliott_strong}). If $A$ is separable, all four are equivalent.

	We will ultimately require some notion of invariance when it comes to all of the pieces involved in the above theorem. In particular, our characterization of simplicity, or lack thereof, is a modified version of condition~(\ref{thm:properly_outer_equivalence:injective_envelope}), where we require the unitary $u$ corresponding to some automorphism $\alpha_r$ (where $r \in \FC(G) \setminus \set{e}$) to be $C_G(r)$-invariant. Note that this is in contrast to the usual definition of proper outerness, where the automorphisms are considered individually without any regard to the rest of the group action.
	
	Let us briefly investigate the feasibility and outcome of generalizing each characterization of proper outerness on the original C*-algebra $A$. First, consider (\ref{thm:properly_outer_equivalence:borchers}). The Borchers spectrum is normally defined for actions of abelian groups (and for single automorphisms $\alpha$, it secretly considers the corresponding $\Z$-action). In the setting of abelian groups, \cite[Theorem~7.3]{hamana85-injective_envelopes_equivariant} essentially gives us the exact result we are after, and makes use of the Borchers spectrum for the action of the \emph{entire} group (as opposed to a single automorphism). Note, however, that the definition of the Borchers spectrum involves the dual group $\what{G}$. This strongly hints to the fact that any invariant Borchers-type characterization that is meant to work in the non-abelian setting will involve the use of the non-abelian dual group. Such an item is borderline impossible to get a concrete handle on, in practice, for most infinite groups. Moreover, the definition of the Borchers spectrum also involves considering all possible hereditary C*-subalgebras, and these are already non-trivial enough in general as-is. Thus, while it could \emph{in theory} be possible to obtain an appropriate generalization, we highly doubt it would be one that people would wish to use in practice.
	
	We also mention that \cite[Theorem~6.6]{olesen_pedersen_III} contains several other characterizations of proper outerness. However, the vast majority (with a couple of exceptions) also involve considering either all possible hereditary C*-subalgebras of $A$, or at least the invariant ones. Again, this would make for conditions that are perhaps mysterious and hard to check in practice, and so we choose to also skip those. It is worth noting though that all of these conditions are very likely still possible to generalize.
	
	The obvious characterization remaining is Elliott's characterization, i.e.\ condition~(\ref{thm:properly_outer_equivalence:elliott}). While it requires separability of the underlying C*-algebra $A$ to be a true characterization, most C*-algebras that people are interested in end up being separable anyways. Moreover, it is only necessary to consider the space of invariant ideals of $A$, as opposed to the set of all invariant hereditary subalgebras. Thus, in our opinion, it is the most worthwhile characterization to generalize.
	
	We would like to first give an initial and incorrect guess as to how the generalization would proceed. One might guess that ``invariant'' proper outerness of $\alpha_t$ would mean that for any $C_G(t)$-invariant ideal $I \subseteq A$, and any $C_G(t)$-invariant unitary $u \in M(I)$, we have $\norm{\alpha|_I - (\Ad u)|_I} = 2$. However, this turns out to be just slightly too weak of a notion, as there is simply no reason to expect enough suitable \emph{invariant} ideals $I \subseteq A$ and \emph{invariant} unitaries $u \in M(I)$ for approximating the automorphism. As it turns out, the correct notion to use is \emph{approximately invariant} ideals and unitaries, in the appropriate sense.
	
	Before proceeding further, we remark that all of the theory of injective envelopes that we use was in the setting of \emph{unital} C*-algebras. In \cite[Section~6]{hamana82_mc_tensor_products_I}, Hamana defines the injective envelope of a non-unital C*-algebra as follows (and studies it further in \cite{hamana82_centre}):
	
	\begin{definition}
    \label{def:nonunital_injective_envelope}
		Let $A$ be a C*-algebra, not necessarily unital. We define $A^+$ to just be $A$ if $A$ is already unital, and the unitization otherwise. In particular, if $J \subseteq A$ is an ideal, the unitization $J^+$ will still be $J$ if, coincidentally, $J$ has its own unit. The injective envelope of $A$ is defined as $I(A) \defeq I(A^+)$.
	\end{definition}
	
	We also state the following result of Hamana, which gives a correspondence between ideals of a C*-algebra $A$ and central projections in $I(A)$. This, and other results, usually work in the setting of $A$ being non-unital, but to avoid unnecessary subtleties, we will just stick with the unital setting unless necessary.
	
	\begin{proposition}
		\label{prop:ideal_injective_envelope}
		Let $J$ be an ideal of a unital C*-algebra $A$. Let $p = \sup^{I(A)} J_1^+$. Then
		\begin{enumerate}
			\item $p$ is a central projection in $I(A)$;
			\item $I(A)p \isoto I(J)$.
		\end{enumerate}
		Conversely, if we start with some central projection $p \in I(A)$ and define an ideal $J \subseteq A$ as $J = I(A)p \cap A$, then $I(J) \isoto I(A)p$.
	\end{proposition}
	
	\begin{proof}
		The forward direction was proven in \cite[Lemma~1.1]{hamana82_centre}, in the more general setting of hereditary subalgebras of $A$. For the converse direction, note that if $J = I(A)p \cap A$ for a central projection $p\in I(A)$, then $\sup^{I(A)} J_1^+=p$ by Proposition~\ref{prop:ConcreteSupProj}.
	\end{proof}
	
	This now hints at how the proof of the aforementioned equivalences would work in general. Any automorphism $\alpha \in \Aut(A)$ that is ``close enough'' to being inner on an $\alpha$-invariant ideal $J \subseteq A$ will give something that is genuinely inner on $I(J)$, which is a central corner of $I(A)$.
	
	We may also realize the multiplier algebra of a C*-algebra $A$ as an idealizer of $A$ inside its injective envelope $I(A)$. 
	
	\begin{proposition}[{\cite[Section~1]{hamana82_centre}}]
		\label{prop:multiplier_algebra_contained_in_injective_envelope}
		Assume $A$ is a not necessarily unital C*-algebra, and denote by $M(A)$ its multiplier algebra. We have
		\[ M(A) = \setbuilder{x \in I(A)}{xa \in A \text{ and } ax \in A \text{ for all } a \in A}. \]
	\end{proposition}
	
	Recall that an ideal $K$ is a C*-algebra $A$ is essential if and only if whenever $K \cdot a = 0$ for some $a\in A$ (which is equivalent to $a \cdot K = 0$), then $a=0$.
	As mentioned earlier, we will also be dealing with ideals that are ``almost invariant'' in the appropriate sense. The following observation allows us to establish what this means rigorously.
	
	\begin{proposition}
		\label{prop:essential_ideal_containment_implies_inj_env}
		Let $A$ be a unital C*-algebra, and let $J,K \subseteq A$ be ideals with $J \subseteq K$. Then
		\begin{enumerate}
			\item Assume that $J$ is essential in $K$. Then recalling Proposition~\ref{prop:ideal_injective_envelope}, $I(J)$ and $I(K)$ both share the same central support projection $p\in I(A)$, so that we canonically have $I(J) \cong I(K) \cong I(A)p$.
			\item If $\sup^{I(A)}J_1^+=\sup^{I(A)}K_1^+$, then $J$ is essential in $K$.
		\end{enumerate} 
	\end{proposition}
	
	\begin{proof}
		Let $p=\sup^{I(A)} J_1^+$ and $q=\sup^{I(A)} K_1^+$. Since $J\subseteq K$, we have that $p\leq q$. We need to show that $q\leq p$. Let $I=I(A)(q-p)\cap A$. Observe that $I$ is an ideal in $A$ which has trivial intersection with $J$. On the other hand, since $J$ is essential in $K$, $I$ must have trivial intersection with the ideal $K$ as well. This implies that $I\cdot K=0$, and by Proposition~\ref{prop:ideal_sup_orthogonality} we obtain that $q(q-p)=0$, and so $q\leq p$.
		
		For the second statement, assume that $p=\sup^{I(A)}J_1^+=\sup^{I(A)}K_1^+$ and let $k\in K$ be an element which is orthogonal to $J$. Then $kp=0$ by Proposition~\ref{lem:sup_bFbstar}. However, by Proposition~\ref{prop:ideal_sup_projection}, $p$ acts as the identity on $K$, and we conclude that $k=0$, as desired.
	\end{proof}
	
	The following two observations are well-known, but we recall them nevertheless.
	
	\begin{remark}\label{rem:multipliers_inclusion_essential_ideal}
		Let $A$ be a unital C*-algebra, and let $J,K \subseteq A$ be ideals with $J \subseteq K$, and such that $J$ is essential in $K$. Then $M(K) \subseteq M(J)$ as a unital inclusion of the multiplier algebras.
	\end{remark}
	
	\begin{remark}
		\label{rem:intersection_of_essential_ideals}
		Assume $A$ is a C*-algebra, not necessarily unital, and assume that $I$ and $J$ are two essential ideals of $A$. Then $I \cap J$ is also essential in $A$.
	\end{remark}
	
	Before proceeding, we mention what our suitable replacement for having an ideal $J \subseteq A$ invariant on the nose. Such an ideal will be considered ``almost invariant'' with respect to an action of a group $H$ on $A$ if, while we do not necessarily have $h \cdot J = J$, we at least have that $J \cap h \cdot J$ is essential in both $J$ and $h \cdot J$.
	
	We will also require a stronger version of Elliott's proper outerness characterization than the one presented in Theorem~\ref{thm:properly_outer_equivalence}. There, it can be observed that if $\alpha$ is an automorphism of a (not necessarily unital) separable C*-algebra $A$ such that $\alpha$ extends to an inner automorphism on \emph{all of} $I(A)$ (such automorphisms are called \emph{quasi-inner}), then given any $\varepsilon > 0$, there is some ideal $J \subseteq A$ and unitary $u \in M(J)$ such that $\norm{\alpha|_J - (\Ad u)|_J} < \varepsilon$. However, \cite[Corollary~6.7]{olesen_pedersen_III} more or less observes that in the case of quasi-inner automorphisms, the ideal $J$ can be required to be essential. Proposition~\ref{prop:elliott_strong_essential} uses essentially the same argument, but we recall two lemmas before proving it. 
	
	\begin{lemma}[{\cite[Lemma~3.1]{hamana82_centre}}]
		\label{lem:direct_sum_multiplier_algebra_injective_envelope}
		Let $(A_\lambda)_{\lambda \in \Lambda}$ be a family of C*-algebras, not necessarily unital. Consider the $c_0$-direct sum $\directsum_\lambda A_\lambda$. The injective envelope of this direct sum is the $\ell^\infty$-direct sum $\prod_\lambda I(A_\lambda)$, while the multiplier algebra is $\prod_\lambda M(A_\lambda)$.
	\end{lemma}
	
	\begin{proof}
		The first statement is proven in \cite[Lemma~3.1]{hamana82_centre}. The claim about the multiplier algebra is well-known, but also follows from the first statement and the identification given in Proposition~\ref{prop:multiplier_algebra_contained_in_injective_envelope}.
	\end{proof}
	
	\begin{lemma}
		\label{lem:direct_sum_of_orthogonal_ideals}
		Assume $(J_\lambda)_{\lambda \in \Lambda}$ is a family of pairwise orthogonal ideals in a (not necessarily unital) C*-algebra $A$. The ideal generated by this family is $J = \closure{\operatorname{span}} J_\lambda$, which is canonically isomorphic to the $c_0$-direct sum $\directsum_\lambda J_\lambda$.
	\end{lemma}
	
	\begin{proposition}
		\label{prop:elliott_strong_essential}
		Let $\alpha \in \Aut(A)$ be an automorphism on a separable, not necessarily unital, C*-algebra $A$. If $\alpha$ is quasi-inner, i.e.\ inner on $I(A)$, then for every $\varepsilon > 0$, there exists an \emph{essential} $\alpha$-invariant ideal $J \subseteq A$ and a unitary $u \in M(J)$ with the property that $\norm{\alpha|_J - (\Ad u)|_J} < \varepsilon$.
	\end{proposition}
	
	\begin{proof}
		Fix $\varepsilon > 0$. By Zorn's lemma, there is a maximal family of pairs $(J_\lambda, u_\lambda)_{\lambda\in\Lambda}$ with the property that $J_\lambda$ is a non-zero $\alpha$-invariant ideal of $A$, $u_\lambda$ is a unitary in $M(J_\lambda)$ with the property that $\norm{\alpha|_{J_\lambda} - (\Ad u_\lambda)|_{J_\lambda}} < \varepsilon$, and $(J_\lambda)_{\lambda \in \Lambda}$ are all pairwise orthogonal.
		
		Observe that $J \defeq \closure{\operatorname{span}} J_\lambda$ is a new $\alpha$-invariant ideal in $A$, which is isomorphic to $\directsum_\lambda J_\lambda$ by Lemma~\ref{lem:direct_sum_of_orthogonal_ideals}. We then know by Lemma~\ref{lem:direct_sum_multiplier_algebra_injective_envelope} that its multiplier algebra is the $\ell^\infty$-direct sum $\prod_\lambda M(J_\lambda)$. In particular, we are free to set the individual coordinates to anything bounded with no other restriction, and so we have that $u \defeq (u_\lambda)_{\lambda\in\Lambda}$ is an element of the multiplier algebra $M(J)$ satisfying
		\[ \norm{\alpha|_J - (\Ad u)|_J} \leq \varepsilon. \]
		
		Maximality of the family $(J_\lambda,u_\lambda)_{\lambda\in\Lambda}$ gives us the final result we are after, namely that $J$ above is essential. If it were not, then $J^{\perp}$ would be $\alpha$-invariant and orthogonal to every $J_\lambda$. Moreover, using the fact that $I(J^{\perp})$ is a corner of $I(A)$ by Proposition~\ref{prop:ideal_injective_envelope}, we have that $\alpha$ is inner on $I(J^{\perp})$. Applying Theorem~\ref{thm:properly_outer_equivalence} to the unitization $(J^\perp)^+$, there at least exists some non-trivial $\alpha$-invariant ideal $L \subseteq (J^\perp)^+$ (and $L \cap J^+$ is a non-trivial $\alpha$-invariant ideal of $A$), and some unitary $v \in M(L) \subseteq M(L \cap J^+)$ with the property that $\norm{\alpha|_L - (\Ad v)|_L} < \varepsilon$. The pair $(L \cap J^+,v)$ could then be added to our maximal family from before, contradicting the fact that it is actually maximal.
	\end{proof}

	Unlike the proof of the equivalence between Elliott's characterization of proper outerness (condition~(3) in Theorem~\ref{thm:properly_outer_equivalence}) and the injective envelope version (condition~(2) in Theorem~\ref{thm:properly_outer_equivalence}), our approach will completely avoid the theory of derivations, and instead take a different path (one which genuinely appears to be necessary for the equivalence between the \emph{equivariant} notions of proper outerness). Essentially, what was shown in the proof of Theorem~\ref{thm:properly_outer_equivalence} is that if $\norm{\beta - \id} < 2$ on an ideal $J\subseteq A$ for some automorphism $\beta\in\Aut(J)$, then we have that $\beta$ extends to an inner automorphism on $I(J)$. What we need is some control over the unitary itself, in $I(J)$. Intuitively, if $\beta$ is ``close'' to the identity, then the resulting unitary should be ``close'' to $1$. In the context of von Neumann algebras at least, this intuition is indeed true, and the following lemma can be found in an important paper of Kadison and Ringrose.
	
	\begin{lemma}[{\cite[Lemma~5]{kadison_ringrose}}]
		\label{lem:kadison_ringrose_unitary_vn_alg}
		Let $\alpha \in \Aut(M)$ be an automorphism of a von Neumann algebra $M$ such that $\norm{\alpha - \id} < 2$. Then there is a unitary $u \in M$ such that $\alpha = \Ad u$, and moreover $u$ has the following restriction on its spectrum:
		\[ \sigma(u) \subseteq \setbuilder{z \in \mathbb{T}}{\Re z \geq \frac{1}{2} \sqrt{4 - \norm{\alpha - \id}^2}}. \]
	\end{lemma}
	
	Note that we indeed have that if $\norm{\alpha - \id}$ gets closer to zero, then $\norm{u - 1}$ indeed gets closer to zero as well. We suspect that the above lemma might work for general monotone complete C*-algebras. However, we work around this in a sneaky manner and can derive the result on just injective envelopes quite easily. For the proof, we will need the following lemma. We would also like to apologize to Zarikian, as when we had originally released the first preprint of our paper, we did not notice that this same exact result could be found in one of his papers.
	
	\begin{lemma}[{\cite[Lemma~2.2]{Zarikian}}]
		\label{lem:non_free_on_A_on_inj}
		Let $A$ be a unital C*-algebra and let $\alpha\in \Aut(A)$, with its unique extension to an automorphism of $I(A)$ also denoted by $\alpha$. Assume that $x\in I(A)$ is such that
		$xy=\alpha(y)x$ for every $y\in A$. Then $xy=\alpha(y)x$ for every $y\in I(A)$.
	\end{lemma}
	
	\begin{lemma}
		\label{lem:kadison_ringrose_unitary_injective_envelope}
		Let $\alpha \in \Aut(A)$ be an automorphism of a not necessarily unital C*-algebra $A$, such that $\norm{\alpha - \id} < 2$. Then the unique extension of $\alpha$ to $I(A)$ (obtained by first extending canonically to $A^+$) is inner and of the form $\Ad u$, where $u$ is a unitary in $I(A)$ which can be required to satisfy the following restriction on its spectrum:
		\[ \sigma(u) \subseteq \setbuilder{z \in \mathbb{T}}{\Re z \geq \frac{1}{2} \sqrt{4 - \norm{\alpha - \id}^2}}. \]
	\end{lemma}
	
	\begin{proof}
		We already know that such a result holds on von Neumann algebras. Let $M$ be a von Neumann algebra containing $A$ as a weak*-dense subalgebra, and moreover having the property that the automorphism $\alpha \in \Aut(A)$ extends (necessarily uniquely) to an automorphism $\wtilde{\alpha} \in \Aut(M)$. It does not appear that any sort of separability is required for $M$ in either the statement or the proof of Lemma~\ref{lem:kadison_ringrose_unitary_vn_alg} of Kadison and Ringrose (so in particular, one could take $M = A^{**}$). However, just for peace of mind of the reader, it is observed in that same paper that any automorphism $\alpha \in \Aut(A)$ with $\norm{\alpha - \id} < 2$ \emph{automatically} extends to any such enveloping von Neumann algebra $M$. See \cite[Theorem~7]{kadison_ringrose}, along with the definition of \emph{permanently weakly inner} or \emph{$\pi$-weakly inner} in \cite[Page~35]{kadison_ringrose}. It appears that Kadison and Ringrose deal with unital C*-algebras in their paper, but this is not a problem. The canonical extension of $\alpha$ to $A^+$ (call it $\alpha^+$) also satisfies $\norm{\alpha^+ - \id_{A^+}} < 2$. Hence, if $A$ is separable, one can easily take any faithful non-degenerate representation $A \subseteq B(H)$ where $H$ is separable, and then let $M = A''$.
		
		In any case, let $M$ and $\wtilde{\alpha} \in \Aut(M)$ be as above, and observe that $\norm{\wtilde{\alpha} - \id_M} = \norm{\alpha - \id_A}$, via a straightforward application of weak*-continuity of $\wtilde{\alpha}$ and Kaplansky's density theorem. Let $u \in M$ be a unitary obtained from Lemma~\ref{lem:kadison_ringrose_unitary_vn_alg}. In other words, $\wtilde{\alpha} = \Ad u$, and
		\[ \sigma(u) \subseteq \setbuilder{z \in \mathbb{T}}{\Re z \geq \frac{1}{2} \sqrt{4 - \norm{\alpha - \id}^2}}. \]
		For convenience, denote the bound $\frac{1}{2} \sqrt{4 - \norm{\alpha - \id}^2}$ by $r$. We have $\Re \sigma(u) \subseteq [r,1]$, and so by the continuous functional calculus, we genuinely have $\Re u \geq r$, where $\Re u = \frac{1}{2}(u + u^*)$.
		
		We proceed to construct a similar element in $I(A)$ using the only method we know. We know that $A^+ \subseteq M$, and by injectivity, there is a unital and completely positive map $E \colon M \to I(A)$ such that it is the identity map on $A^+$. Observe that, because $\Re u \geq r$ with $r > 0$, we have that
		\[ \Re E(u) = E(\Re u) \geq E(r) = r. \]
		In particular, we have that $E(u)$ is non-zero. Moreover, observe that for any $b \in A^+$ (which lies in the multiplicative domain of $E$), we have
		\[ E(u) b = E(ub) = E(\alpha(b) u) = \alpha(b) E(u). \]
		By Proposition~\ref{lem:non_free_on_A_on_inj}
		this automatically implies that $E(u) z = \alpha(z) E(u)$ for all $z \in I(A)$.
		
		For convenience, write $y = E(u)$. We have no reason to expect that $y$ is in any way a unitary element. However, this is not too hard to patch up. Let $y = v \abs{y}$ be the polar decomposition of $y$ in $I(A)$, with $p = v^*v$ being the support projection of $v$. Recall from Theorem~\ref{thm:PropOuterCorner} that
		\begin{itemize}
			\item $p$ is an $\alpha$-invariant central projection in $I(A)$.
			\item $v$ is a unitary in $I(A)p$, and on this corner, we in fact have $\alpha|_{I(A)p} = \Ad v$.
		\end{itemize}
		It was also shown in the proof of Theorem~\ref{thm:PropOuterCorner} that $y^*y$ (and therefore also $\abs{y}$) will always lie in the center of $I(A)$.  Given that $v$ is a unitary on its respective corner, it follows that $C^*(1,v,\abs{y})$ is a commutative C*-algebra of the form $C(X)$.
		
		Writing $y = v \abs{y}$ inside of $C(X)$, and using that $\Re y \geq r$ for $r > 0$, we note that $y$ is invertible in $C(X)$. Consequently, so is $\abs{y}$, and it follows that the support projection $p$ of $v$ was in fact $1$. In other words, $v$ is a unitary in $I(A)$ implementing $\alpha$. Finally, given any $x \in X$, we have $\Re y(x) \geq r$, and $0 < \abs{y}(x) \leq 1$ (which follows from $\norm{y} = \norm{E(u)} \leq 1$). Thus, $v(x) = \frac{y(x)}{\abs{y}(x)}$ also satisfies $\Re v(x) \geq r$ for all $x\in X$. As the spectrum of $v$ as an element of $C(X)$ is given by $\setbuilder{v(x)}{x \in X}$, and coincides with its spectrum as an element of $A$, this finishes the proof.
	\end{proof}
	
	Clearly, if $u$ is a unitary in any unital C*-algebra with $\Re \sigma(u) \subseteq [r,1]$ for $r$ close to $1$, then it is the case that $u$ is close to $1$ as well by the continuous functional calculus. The following lemma establishes the precise norm estimate:
	
	\begin{lemma}
		\label{lem:unitary_spectrum_to_norm}
		Let $k \in [0,2]$, and let $u \in U(A)$ be a unitary in a unital C*-algebra $A$ with $\Re \sigma(u) \subseteq [\frac{1}{2} \sqrt{4 - k^2},1]$. Then we have
		\[ \norm{u - 1} \leq \sqrt{2 - \sqrt{4-k^2}}. \]
	\end{lemma}
	
	\begin{proof}
		This is essentially an exercise in Euclidean geometry. Consider Figure~\ref{fig:unitary_norm}.
		\begin{figure}
			\centering
			\includegraphics[width=0.45\textwidth]{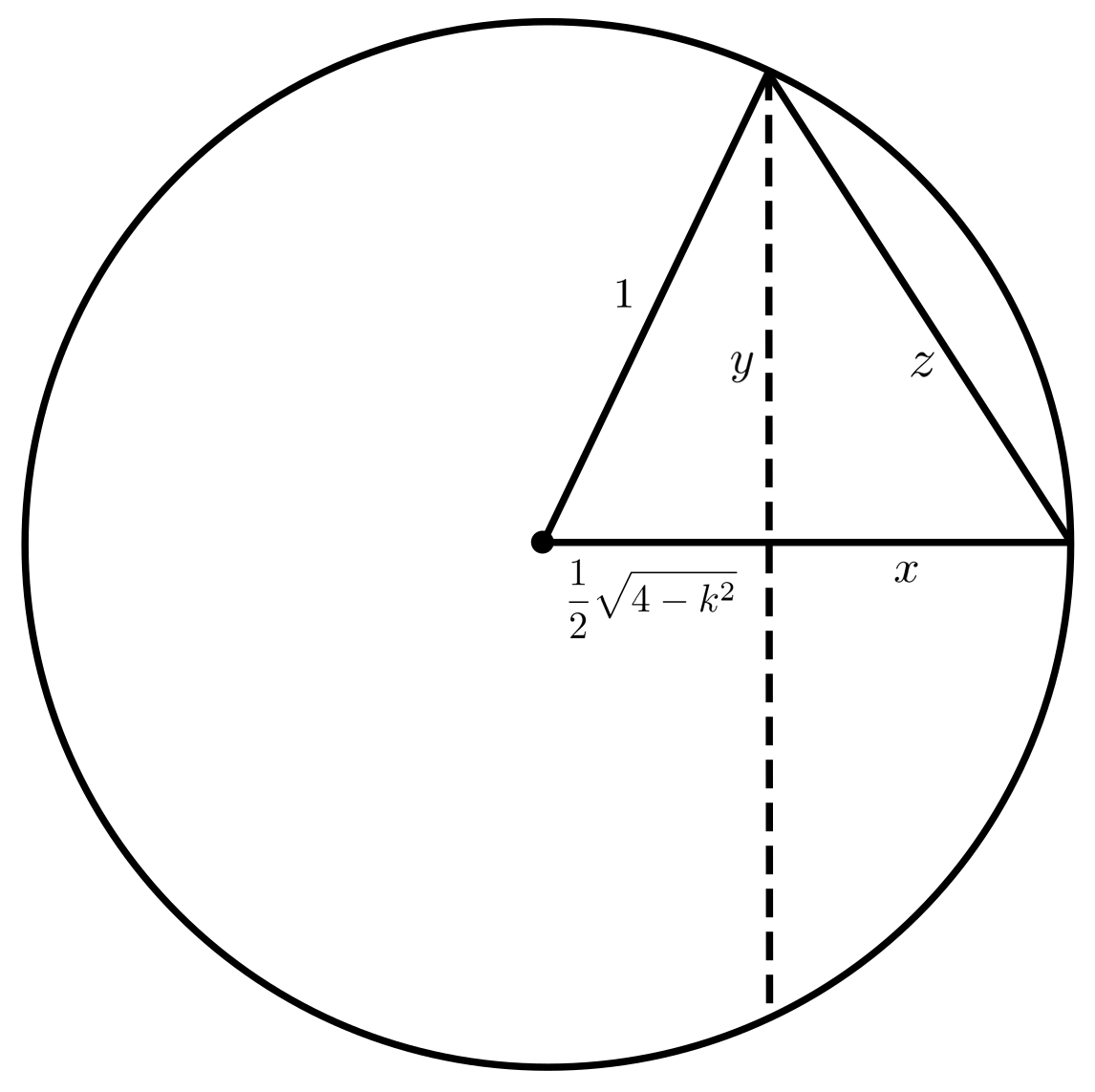}
			\caption{Estimate of $\norm{u - 1}$ based on estimate of $\Re \sigma(u)$.}
			\label{fig:unitary_norm}
		\end{figure}
		It depicts the unit circle $\mathbb{T}$, and the spectrum of $u$ lies on the arc of the circle that is to the \emph{right} of the dotted line. Some unknown side lengths are labeled as $x$, $y$, and $z$. It is clear that $x = 1 - \frac{1}{2} \sqrt{4 - k^2}$, and by the Pythagorean theorem that $y = \frac{1}{2} k$. One more application of the Pythagorean theorem gives us that
		\[ z = \sqrt{2 - \sqrt{4 - k^2}}. \]
		The value of $z$ represents the maximum distance of any given point in $\sigma(u)$ from $1$, and so by the continuous functional calculus, it must be the case that $\norm{u - 1} \leq z$, our desired result.
	\end{proof}
	
	A similar and even easier estimate characterizes when $\sigma(u)$ has a positive lower bound as well:
	
	\begin{lemma}
		\label{lem:unitary_spectrum_positive_to_norm_sqrt2}
		Let $u \in U(A)$ be a unitary in some unital C*-algebra $A$. Then there exists an $r > 0$ with $\Re \sigma(u) \subseteq [r,1]$ if and only if $\norm{u - 1} < \sqrt{2}$.
	\end{lemma}
	
	\begin{proof}
		Again, we perform basic Euclidean geometry. In Figure~\ref{fig:unitary_norm_sqrt2},
		\begin{figure}
			\centering
			\includegraphics[width=0.45\textwidth]{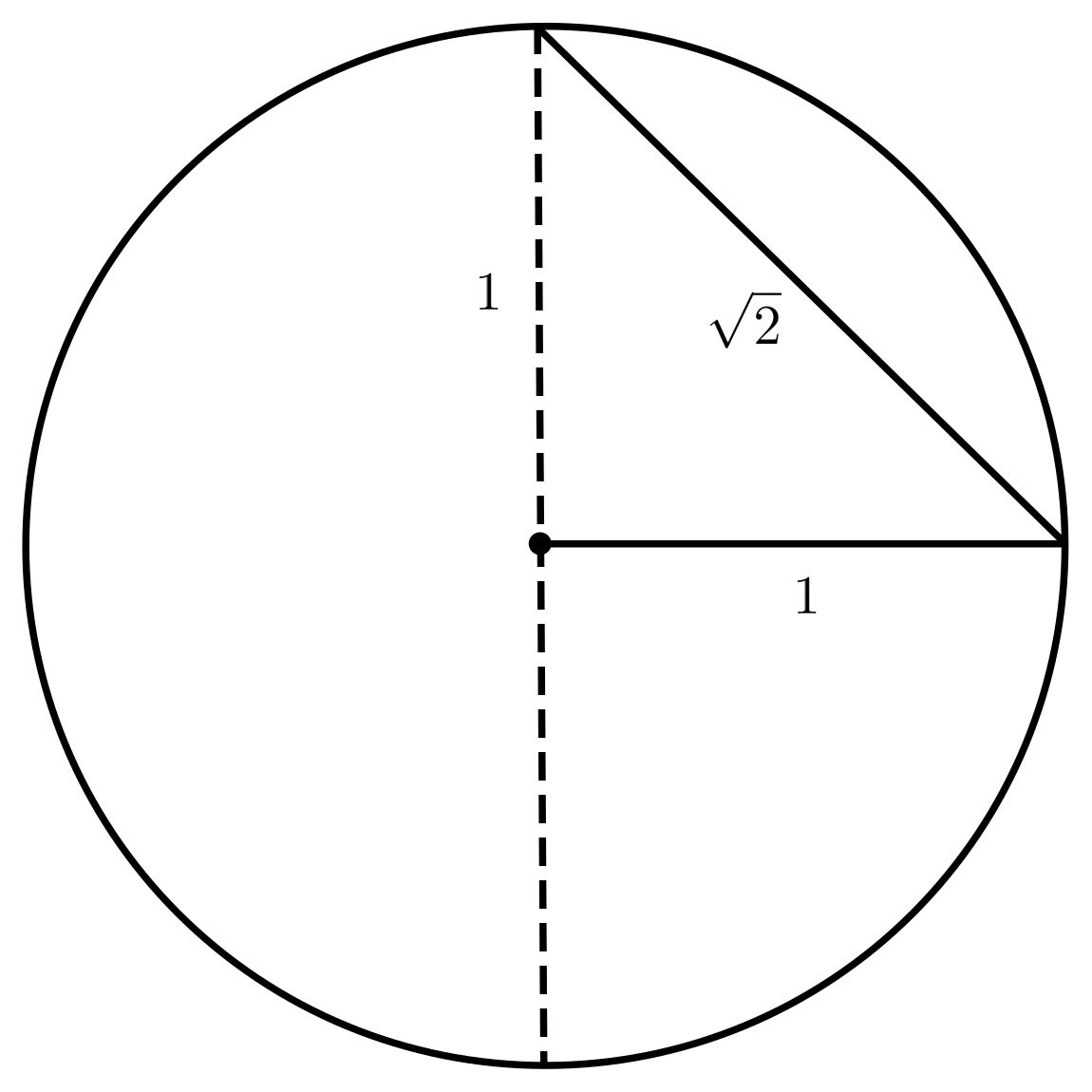}
			\caption{Estimate of $\norm{u-1}$ based on $\Re \sigma(u) > 0$ and vice versa.}
			\label{fig:unitary_norm_sqrt2}
		\end{figure}
		we again need $\sigma(u)$ to lie \emph{strictly} to the right of the dotted line, which is equivalent to any point in the spectrum being distance strictly less than distance $\sqrt{2}$ away from $1$. By continuous functional calculus, this is equivalent to $\norm{u-1} < \sqrt{2}$.
	\end{proof}
	
	Now we are ready to state the main result.
	
	\begin{theorem}
		\label{thm:equivariant_proper_outer_equivalence}
		Let $\alpha\colon G\to\Aut(A)$ be a group action on a unital C*-algebra $A$, such that we are in one of the following situations:
		\begin{itemize}
			\item $G$ is an FC-group.
			\item $G$ is an arbitrary discrete and $\alpha$ is a minimal action.
		\end{itemize}
		The following are equivalent:        
		\begin{enumerate}
			\item
			\label{thm:equivariant_proper_outer_equivalence:injective_envelope}
			There exists an $r \in \FC(G) \setminus \set{e}$, a non-zero $C_G(r)$-invariant central projection $p \in I(A)$, and a $C_G(r)$-invariant unitary $u \in U(I(A)p)$ such that $r$ acts by $\Ad u$ on $I(A)p$.
			
			\item
			\label{thm:equivariant_proper_outer_equivalence:elliott}
			There exists an $r \in \FC(G) \setminus \set{e}$, a non-zero $r$-invariant ideal $J \subseteq A$ with the property that $J \cap h \cdot J$ is essential in both $J$ and $h \cdot J$ for all $h \in C_G(r)$, and a unitary $u \in M(J)$, such that, if we let $k = \norm{\alpha_r|_J - (\Ad u)|_J}$ and $t = \sup_{h \in C_G(r)} \norm{h\cdot u - u}$, we have $2 \sqrt{2 - \sqrt{4 - k^2}} + t < \sqrt{2}$.
		\end{enumerate}
		
		Note that the $r \in \FC(G) \setminus \set{e}$ and $p \in Z(I(A))$ in one equivalence do not in any way need to correspond to the same $r$ or $J$ in the other equivalence! Moreover, note that in the second equivalence, we canonically have $M(J)$ and $h\cdot M(J)\cong M(h \cdot J)$ contained in $M(J \cap h \cdot J)$, for every $h\in C_G(r)$, by Remark~\ref{rem:multipliers_inclusion_essential_ideal}, and so the norm difference $\norm{h\cdot u - u}$ makes sense.
	\end{theorem}
	
	\begin{proof}
		First, we prove (\ref{thm:equivariant_proper_outer_equivalence:elliott}) $\implies$ (\ref{thm:equivariant_proper_outer_equivalence:injective_envelope}). To give some intuition first on what is about to happen, all of the estimates that we gave are for the express purpose of being able to write (the unique extension of) $\alpha_r$ as some inner automorphism $\Ad w$ on $I(J)$, with respect to a unitary $w$ that is \emph{almost} $C_G(r)$-invariant. Our aim is then to ``average'' the translates $h \cdot w$ ($h \in C_G(r)$) of this unitary to get something that is truly invariant (and importantly, non-zero, which is what the almost invariance is for). It is, of course, not immediately obvious how to do this, given that the group $C_G(r)$ is not amenable in general, nor does $I(A)$ have any nice compact topology on the unit ball. The techniques developed in the rest of the paper, however, will show that neither are necessary.
		
		Let us also remark that, by Proposition~\ref{prop:essential_ideal_containment_implies_inj_env}, all of the translates $h \cdot J$ (for $h\in C_G(r)$) share the same injective envelope. Specifically, they all share the same support projection $p \in Z(I(A))$ and $I(h \cdot J)$ (for $h\in C_G(r)$) are all canonically isomorphic to $I(A)p$. Observe that this central projection $p$ is also necessarily $C_G(r)$-invariant. Moreover, the norm difference $\sup_{h \in C_G(r)}\norm{h\cdot u - u}$ from the statement of the theorem makes sense as a difference of unitaries in this corner, but it can also be viewed in significantly smaller C*-algebras by the fact that Remark~\ref{rem:multipliers_inclusion_essential_ideal} also tells us $M(J), \ M(h \cdot J) \subseteq M(J \cap h \cdot J)$, for all $h\in C_G(r)$.
		
		Right off the bat, we notice that $k < 2$, and so $\norm{(\Ad u^* \circ \alpha_r)|_J - \id|_J} < 2$. Let $\wtilde{\alpha_r} \in \Aut(I(A))$ denote the unique extension of $\alpha_r$ to an automorphism on $I(A)$, and observe that $\wtilde{\alpha_r}|_{I(J)}$ is also the unique extension of $\wtilde{\alpha_r}|_{J^+}$ to $I(J)$, where $J^+ = C^*(J,p)$. By Lemma~\ref{lem:kadison_ringrose_unitary_injective_envelope}, we have that $(\Ad u^* \circ \wtilde{\alpha_r})|_{I(J)} = (\Ad v)|_{I(J)}$ for some unitary $v \in I(J)$ with $\Re \sigma(v) \subseteq [\frac{1}{2} \sqrt{4 - k^2},1]$. Using Lemma~\ref{lem:unitary_spectrum_to_norm}, we have that
		\[ \norm{v - p} \leq \sqrt{2 - \sqrt{4 - k^2}}, \]
		from which we also obtain the fact that for any $h \in C_G(r)$, we have
		\[ \norm{h\cdot v - v} \leq \norm{h\cdot v - p} + \norm{v - p} = 2 \norm{v - p} \leq 2 \sqrt{2 - \sqrt{4 - k^2}}. \]
		Also for convenience, we will let $w \defeq uv \in U(I(J))$, so that in particular, $\wtilde{\alpha_r}|_{I(J)} = (\Ad w)|_{I(J)}$. Finally, we obtain an estimate on $\norm{h \cdot w - w}$ for $h \in C_G(r)$:
		\[ \norm{h \cdot w - w} \leq \norm{h \cdot u - u} \norm{h \cdot v} + \norm{u} \norm{h \cdot v - v} \leq 2 \sqrt{2 - \sqrt{4 - k^2}} + t. \]
		
		The above norm estimate ends up being good enough for our purposes, namely to be able to ``average'' the translates and obtain something non-zero. Now we describe the averaging process itself. Consider the C*-algebra $\ell^\infty(G,I(A))$. It is equipped with a diagonal $G$-action given by $(s \cdot f)(g) = s \cdot f(s^{-1}g)$, for $s,g\in G$. Moreover, the embedding of $I(A) \injectsinto \ell^\infty(G,I(A))$ given by sending $y \in I(A)$ to the constant function $f_y(g) = y$, for all $g\in G$, is a $G$-equivariant *-homomorphic embedding.
		
		To construct a $C_G(r)$-invariant ``average'' of $w$, we proceed as follows. Let $T \subseteq G$ be a transversal of the space of right coset space $C_G(r) \backslash G$. That is, every $g_1 \in G$ can be written \emph{uniquely} as $g_1 = h g_2$ for some $h \in C_G(r)$ and $g_2 \in T$. Set $W \in \ell^\infty(G,I(A))$ by $W(h g_2) = h \cdot w$ for every $h \in C_G(r)$ and $g_2 \in T$. It is not hard to check that $W$ is $C_G(r)$-invariant. Moreover, as $\wtilde{\alpha_r}|_{I(J)} = (\Ad w)|_{I(J)}$, and this automorphism is $C_G(r)$-invariant, it follows that $\Ad W$ also implements the diagonal automorphism induced by $r$ on $I(J) \subseteq I(A) \subseteq \ell^\infty(G,I(A))$.
		
		Now we wish to push $W$ back to $I(A)$. We could use an expectation back down onto this C*-algebra, but this would lose us $C_G(r)$-invariance. Instead, we push it down onto the larger algebra $I_G(A)$, and use the far more roundabout techniques developed in the rest of the paper, where we are still able to obtain an element in $I(A)$ satisfying the same properties. Let $E \colon \ell^\infty(G,I(A)) \to I_G(A)$ be a $G$-equivariant unital completely positive map that is the identity map on $I(A)$. Given any $y \in I(J)$, we have
		\[ E(W) y = E(Wy) = E(\wtilde{\alpha_r}(y)W) = \wtilde{\alpha_r}(y) E(W), \]
		where we use the fact that $I(J)$ lies in the multiplicative domain of $E$. Moreover, by $G$-equivariance of $E$, we still have that $E(W)$ is $C_G(r)$-invariant.
		
		It is important to verify that the element $E(W)$ that we obtain in $I_G(A)$ is sufficiently non-trivial for our purposes. Since $p \in Z(I(A))$, by Proposition~\ref{prop:injective_envelope_inclusions}, we have $p \in Z(I_G(A))$ as well. We aim to prove the following observation: $E(W)$ is a non-zero element in $I_G(A)p$.
		
		To see that $E(W) \in I_G(A)p$, we remark that $p$ lies in the multiplicative domain of $E$ and $W$ lies entirely in $\ell^\infty(G,I(A))p$. Thus,		
		\[E(W)=E(Wp)=E(W)p\in I_G(A)p.\]		
		It remains to show that $E(W) \neq 0$. Recalling that $W$ was constructed as having some translates $h \cdot w$ (for $h \in C_G(t)$) in every coordinate, and $w$ was a unitary in the corner $I(J) = I(A)p$, it follows that $W$ is a unitary in the corner $\ell^\infty(G,I(A))p$. Moreover, on this corner, we have
		\[ \norm{W - w} = \sup_{h \in C_G(r)} \norm{h \cdot w - w} \leq 2 \sqrt{2 - \sqrt{4 - k^2}} + t < \sqrt{2}. \]
		Equivalently, $\norm{w^*W - p} < \sqrt{2}$. But by Lemma~\ref{lem:unitary_spectrum_positive_to_norm_sqrt2}, we have that $\Re \sigma(w^*W) \subseteq [\varepsilon,1]$ for some $\varepsilon > 0$, where the spectrum is considered inside of the C*-algebra $\ell^\infty(G,I(A))p$ and \emph{not} all of $\ell^\infty(G,I(A))$. By the continuous functional calculus, this is equivalent to having $\Re (w^*W) \geq \varepsilon\cdot p$ for this same value of $\varepsilon$. Thus, keeping in mind that $w$ lies in the multiplicative domain of $E$, we have
		\[ 0 < \varepsilon\cdot p \leq E(\Re (w^*W)) = \Re E(w^*W)= \Re (w^*E(W)), \]
		which immediately implies $E(W)\neq 0$.
		
		To summarize, $E(W)$ was some non-zero element in $I_G(A)p$ implementing the action $\wtilde{\alpha_r}$ on $I(J) = I(A)p$. That is, given any $y \in I(J)$, we have
		\[ E(W) y = \wtilde{\alpha_r}(y) E(W). \]
		The importance of having $E(W)$ in the corner $I_G(A)p$ is reflected in the fact that, if $y \in I(A)(1-p)$, then $E(W)y = 0$ and $\wtilde{\alpha_r}(y) E(W) = 0$. In other words, we may conclude that for any $y \in I(A)$, and not just $y \in I(J)$, we have
		\[ E(W) y = \wtilde{\alpha_r}(y) E(W). \]
		
		Our setup is now beginning to look quite similar to the proofs of the previous major theorems in this paper. For convenience, we will significantly simplify notation by letting $x_r \defeq E(W)$ and simply writing $r \cdot y$ instead of $\wtilde{\alpha_r}(y)$. In this language, we have an element $x_r \in I_G(A)$ with the property that for any $y \in I(A)$ we have
		\[ x_r y = (r \cdot y) x_r. \]
		Recall Proposition~\ref{prop:equivariant_unitaries_to_single_unitary}. There is nothing special, per se, about the element $x_r$ needing to be unitary, and the same proof will show that, because $s \cdot x_r = x_r$ for all $s \in C_G(r)$,
		then defining $x_{srs^{-1}} \defeq s \cdot x_r$, for all $s\in G$ gives well-defined elements $(x_c)_{c\in C}$, where $C$ denotes the conjugacy class of $r$. Moreover, given any $c\in C$, we have $s \cdot x_c = x_{scs^{-1}}$, for all $s\in G$. A very similar proof to what is done in Proposition~\ref{prop:equivariant_unitaries_to_single_unitary} will also show that each element $x_{srs^{-1}}$ implements the automorphism corresponding to $srs^{-1}$, for $s\in G$. For completeness, we verify the details here. Given any $y \in I(A)$ and $s\in G$, we have
		\begin{align*}
		x_{srs^{-1}} y &= (s \cdot x_r) y = s \cdot (x_r (s^{-1} y)) \\
		&= s \cdot ((rs^{-1} \cdot y) x_r) = (srs^{-1} \cdot y) (s \cdot x_r) = (srs^{-1} \cdot y) x_{srs^{-1}}.
		\end{align*}
		Consequently, we have that
		\[ \sum_{c \in C} x_c^*\lambda_c \in (I(A) \rtimes_\lambda G)' \cap I_G(A) \rtimes_\lambda G. \]
		The proof of the above fact is straightforward, and if you wish, it is essentially identical to the proof of Proposition~\ref{prop:crossed_product_center_coefficients}. We claim that this element in fact lies in the center of $I_G(A) \rtimes_\lambda G$. To see this, recall from Theorem~\ref{thm:crossed_product_injective_envelope_inclusions} the inclusions
		\[ A \rtimes_\lambda G \subseteq I(A) \rtimes_\lambda G \subseteq I_G(A) \rtimes_\lambda G \subseteq I(A \rtimes_\lambda G). \]
		The element $\sum_{c \in C} x_c^*\lambda_c$ in particular commutes with all of $A \rtimes_\lambda G$. However, this automatically implies that the element lies in the center of $I(A \rtimes_\lambda G)$ (see \cite[Corollary~4.3]{hamana79_injective_envelopes_cstaralg}), and so in particular it lies in $Z(I_G(A) \rtimes_\lambda G)$. Clearly, because $r$ was non-trivial, this element does \emph{not} lie in $I_G(A)$.
		
		Now is when we split the proof of this direction into the two cases of when the group $G$ is an FC group, and when the action of $G$ on $A$ is minimal.
		
		The easier case is the minimal case. Given that the element $\sum_{c\in C} x_c^*\lambda_c$ is a non-trivial element lying in the center of $I_G(A) \rtimes_\lambda G$, we have that $I_G(A) \rtimes_\lambda G$ is not prime, and we know that this is equivalent to $A \rtimes_\lambda G$ is not prime by Proposition~\ref{proposition:primeifffactor}. This is a case that we have already completely solved in Section~\ref{sec:primality_minimal}. In particular, we automatically get by Theorem~\ref{thm:DiscreteMinimal} that there exists some $r_2\in\FC(G) \setminus \set{e}$, some $C_G(r_2)$-invariant central projection $p_2 \in Z(I(A))$, and some $C_G(r_2)$-invariant unitary $u_2 \in I(A)p_2$ with the property that $r_2$ acts by $\Ad u_2$ on $I(A)p_2$. Again, we emphasize that there is no reason to expect these elements, in particular $r_2$ and $p_2$, to have anything to do with the elements $r$ and $p$ from before.
		
		Next, consider the case where $G$ is an FC-group. Our aim is again to reduce it down to the results which we have already proven in Section~\ref{sec:intersection_property_fc}. The element $\sum_{c \in C} x_c^*\lambda_c$ that we have constructed above shows that the inclusion $Z(I_G(A))^G \subseteq Z(I_G(A) \rtimes_\lambda G)$ is proper. 
		Exactly as explained in the proof of Theorem~\ref{thm:mainSec4}, we can apply Lemma~\ref{lemma:stronglynonprime} and conclude that $(A,G)$ does not have the intersection property.
		
		By Theorem~\ref{thm:mainSec4}, we have that there exists some $r_2\in G \setminus \set{e} = \FC(G) \setminus \set{e}$, some $C_G(r_2)$-invariant central projection $p_2 \in Z(I(A))$, and some $C_G(r_2)$-invariant unitary $u_2 \in I(A)p_2$ with the property that $r_2$ acts by $u_2$ on $I(A)p_2$. Once again, these could be completely unrelated to the previous $r$ and $p$.
		
		Now we prove (\ref{thm:equivariant_proper_outer_equivalence:injective_envelope}) $\implies$ (\ref{thm:equivariant_proper_outer_equivalence:elliott}). Let $r\in \FC(G)$, let $p$ be a $C_G(r)$-invariant central projection in $I(A)$, and let $u$ be a $C_G(r)$-invariant unitary in $I(A)p$ such that $r$ acts by $\Ad u$ on $I(A)p$. By Proposition~\ref{prop:ideal_injective_envelope}, $K = A \cap I(A)p$ is an $r$-invariant ideal of $A$ with the property that $I(K) = I(A)p$.
		
		Now we show that, from $K$, we may obtain a smaller ideal that satisfies all of the properties, and importantly all of the norm estimates, that we want. We will start with some $\varepsilon > 0$, and occasionally at certain steps of the proof, mention that it can be made small enough to satisfy the properties that we want.
		
		Let $\varepsilon > 0$. We know by Proposition~\ref{prop:elliott_strong_essential} that there exists an \emph{essential} $r$-invariant ideal $J \subseteq K$ and a unitary $v \in M(J)$ with the property that $\norm{\alpha_r|_J - (\Ad v)|_J} \leq \varepsilon$. Equivalently, we have
		\[ \norm{(\Ad v^* \circ \alpha_r)|_J - \id_J} \leq \varepsilon. \]
		We may stick with $\varepsilon < 2$, so that by Lemma~\ref{lem:kadison_ringrose_unitary_injective_envelope} combined with Lemma~\ref{lem:unitary_spectrum_to_norm}, there is a unitary $w \in I(J)$ with the property that $\Ad v^* \circ \wtilde{\alpha_r} = \Ad w$ on $I(J)$ and $\norm{w-p} < \delta(\varepsilon)$, where $\delta(\varepsilon) \to 0$ as $\varepsilon \to 0$. Again, the precise estimate is not that relevant for our purposes.
		
		Given that $J$ was essential in $K$, we have by Proposition~\ref{prop:multiplier_algebra_contained_in_injective_envelope} and Proposition~\ref{prop:essential_ideal_containment_implies_inj_env} that $M(J) \subseteq I(J) = I(K)$. Hence, we may rewrite the above equality as $\Ad u = \Ad (vw)$ on $I(J)$. In other words, $u = \gamma vw$ for some $\gamma \in Z(I(K))$, and so $v = \gamma^* u w^*$. Our end goal is to show that, while $v$ itself might not be anywhere close to being $C_G(r)$-invariant, we may ``correct'' it by multiplying it with an appropriate central element (in some potentially smaller, but still essential, ideal of $K$). We already know that $w$ is close in norm to $p$, and is therefore approximately $C_G(r)$-invariant. Moreover, $u$ was $C_G(r)$-invariant by assumption. Thus, the only problematic element to deal with is $\gamma$.
		
		To this end, recall that $\gamma \in Z(I(K))$, and thus there is no reason to expect that $\gamma$ lies in $M(J)$. However, we claim that we can approximate it by something that does lie in some potentially different multiplier algebra. Write $Z(I(K))$ as $C(X)$. Given that $I(K)$ is monotone complete, so is $C(X)$, and thus $X$ is extremally disconnected (see the end of Section~\ref{sec:preliminaries:monotone_complete}).
		
		Using compactness of $X$, we can find a finite collection of pairs $(p_n,\beta_n)_{n=1}^{N}$, where each $p_n$ is a non-zero projection in $C(X)$, and each $\beta_n \in \mathbb{T}$ is a unimodular constant so that $\sum_{n=1}^{N}p_n=1_{C(X)}=p$ and $\norm{\gamma p_n - \beta_n p_n}\leq \varepsilon$ for all $n$. 
		
		For every $n\in\mathbb{N}$, let $L_n = A \cap I(A)p_n$. It is clear that the ideals $(L_n)_{n=1}^{N}$ are pairwise orthogonal. Letting $L$ be the ideal generated by all of them, we have $L = \operatorname{span} L_n$ and $L \isoto \directsum_{n=1}^{N} L_n$. It is a subtle, but very important, point that all $L_n$ and therefore $L$ as well are $r$-invariant, due to the fact that $\wtilde{\alpha_r}$ was inner on $I(K)$, and therefore acts trivially on $Z(I(K)) = C(X)$. Note that $L \subseteq K$, as we had $K = A \cap I(A)p$. Moreover, letting $q$ be the support projection of $L$ in $I(A)$, we have $q \leq p$, and also $q \geq p_n$ for all $n$. It follows that $q = p$. Equivalently, $L$ is essential in $K$ by Proposition~\ref{prop:essential_ideal_containment_implies_inj_env}, and so $I(L) = I(K) = I(A)p$.
		
		Note that $p_n\in M(L_n)\subseteq I(L_n)$ for every $n$, and moreover $M(L) \isoto \prod_{n=1}^N M(L_n)$ by Lemma~\ref{lem:direct_sum_multiplier_algebra_injective_envelope}. In fact, viewing $M(L) \subseteq I(A)p$, this isomorphism maps each copy of $M(L_n)$ to its canonical copy in $I(A)p$ as well. Thus,
		\[ f = \sum_{n=1}^N \beta_n p_n \]
		is a unitary in $M(L)$. By construction, we have $f \in Z(I(K))$ and $\norm{f - \gamma} \leq \varepsilon$.
		
		It is perhaps worthwhile to take a small break and to summarize the important points of our current setup. We have:
		\begin{itemize}
			\item $K \subseteq A$ was a $C_G(r)$-invariant ideal, with $r$ acting by $\Ad u$ on $I(K)$, and $u$ being $C_G(r)$-invariant as well.
			
			\item $J \subseteq K$ was an $r$-invariant ideal that is \emph{essential} in $K$, and $v \in M(J) \subseteq I(K)$ was a unitary with $\norm{\alpha_r|_J - (\Ad v)|_J} \leq \varepsilon$.
			
			\item $w \in I(K)$ was a unitary element with $\norm{w - p} \leq \delta(\varepsilon)$.
			
			\item $\gamma \in Z(I(K))$ was a central unitary element so that $v = \gamma^* u w^*$.
			
			\item $L \subseteq K$ was another $r$-invariant ideal, also essential in $K$, and $f \in M(L) \subseteq I(K)$ was a unitary element with $\norm{f - \gamma} \leq \varepsilon$. Moreover, $f \in Z(I(K))$.
		\end{itemize}
		
		With the above in hand, we are now ready to construct our approximately invariant ideal $I \subseteq A$, and approximately invariant unitary in $M(I)$. Let $I = J \cap L$ (and note that $I$ is in particular an ideal of $A$). This is $r$-invariant, as both $J$ and $L$ are $r$-invariant.
		Moreover, given that $J$ and $L$ were both essential in $K$, we have by Remark~\ref{rem:intersection_of_essential_ideals} that $I$ is also essential in $K$. However, $K$ was $C_G(r)$-invariant, and thus $h \cdot I$ is still essential in $K$, for every $h\in C_G(r)$. Applying Remark~\ref{rem:intersection_of_essential_ideals} again, we see that $I \cap h \cdot I$ is essential in $K$ for every $h\in C_G(r)$. In particular, this intersection must be essential in the smaller ideals $I$ and $h \cdot I$ as well.
		
		Now consider the equality $v = \gamma^* u w^*$. Our ``correction factor'' that we constructed for $I$ was the element $f\in M(L)$, which also belongs to $M(I)$ (applying Remark~\ref{rem:multipliers_inclusion_essential_ideal} and noting that $I$ is essential in $K$ and is therefore essential in $L$). Consider the new equality
		\[ fv = f \gamma^* u w^*. \]
		We have $\norm{f \gamma^* - p}\leq \varepsilon$, $\norm{w-p} \leq \delta(\varepsilon)$, and $u$ was completely $C_G(r)$-invariant to begin with. It follows that there is some new bound $\delta_2(\varepsilon)$ such that
		\[ \sup_{h \in C_G(r)}\norm{h \cdot (fv) - fv} \leq \delta_2(\varepsilon), \]
		where, while the precise bound $\delta_2(\varepsilon)$ is not important, what is important is that as $\varepsilon \to 0$, we also have $\delta_2(\varepsilon) \to 0$.
		
		Given that $f \in Z(I(K))$, we have that $fv$ is a unitary element of $M(I)$ with the property that $\Ad (fv)$ and $\Ad v$ implement the same automorphism on $I(K)$. The inequality $\norm{\alpha_r|_J - (\Ad v)|_J} \leq \varepsilon$ therefore implies that 
		\[\norm{\alpha_r|_I - (\Ad (fv))|_I} \leq \varepsilon. \]
		
		In summary, letting $\varepsilon$ be small enough, we have that
		\[ k = \norm{\alpha_r|_I - (\Ad (fv))|_I} \text{ \ and \ } t = \sup_{h \in C_G(r)} \norm{h \cdot (fv) - fv} \]
		are sufficiently small in order to satisfy $2 \sqrt{2 - \sqrt{4 - k^2}} + t < \sqrt{2}$.
	\end{proof}

    \section{The non-unital case}

    All of our previous results were obtained in the case of reduced crossed products $A \rtimes_\lambda G$ where the underlying C*-algebra $A$ is unital. Certain results genuinely made use of the fact that we were dealing with a minimal action on a unital C*-algebra. However, the results in Section~\ref{sec:intersection_property_fc} that characterize the intersection property in the case of FC-groups had no such assumption.

    As such, one might ask whether it is relatively straightforward to bootstrap these results to the non-unital case, perhaps by considering the corresponding action of $G$ on the unitization $A^+$. As a reminder, our convention is that $A^+ = A$ if $A$ is already unital. Note that if $A$ is genuinely non-unital, then the action on $A^+$ is never minimal, even if it is on $A$. Our aim is to show that we may still, however, relate intersection property results. Similar observations are made in \cite[Section~6]{kalantar_scarparo_locally_compact}, which among other things studies the intersection property for crossed products of the form $C_0(X) \rtimes_\lambda G$.

    \begin{lemma}
    \label{lem:nonunital_essential}
        Assume $A$ is a non-unital C*-algebra, and $G$ is a discrete group acting on $A$ by *-automorphisms. Let $A^+$ denote the unitization of $A$. We have that $A \rtimes_\lambda G$ is an essential ideal in $(A^+) \rtimes_\lambda G$.
    \end{lemma}

    \begin{proof}
        We may verify this directly by computing $(A \rtimes_\lambda G)^\perp$. Indeed, if $x \sim \sum b_g \lambda_g$ is the Fourier series of some element $x \in (A^+) \rtimes_\lambda G$, then for $a \in A \subseteq A \rtimes_\lambda G$ the Fourier series of the element $ax$ is $\sum (ab_g) \lambda_g$. In other words, $ax = 0$ tells us that $ab_g = 0$ for all $g \in G$ by uniqueness of Fourier series. But $A$ is an essential ideal in $A^+$, and so this just says that $b_g = 0$ for all $g \in G$, or in other words, $x = 0$.
    \end{proof}

    From here, it is quite straightforward to transfer the intersection property between the two algebras.

    \begin{proposition}
    \label{prop:nonunital_intersection_property_transfer}
        Assume $A$ is a non-unital C*-algebra, and $G$ is a discrete group acting on $A$ by *-automorphisms. Let $A^+$ denote the unitization of $A$. The reduced crossed product $A \rtimes_\lambda G$ has the intersection property if and only if $(A^+) \rtimes_\lambda G$ does.
    \end{proposition}

    \begin{proof}
        First, assume that $A \rtimes_\lambda G$ does not have the intersection property, so that there is some nontrivial ideal $J \triangleleft A \rtimes_\lambda G$ with $J \cap A = 0$. Since $A \rtimes_\lambda G$ is an ideal of $(A^+) \rtimes_\lambda G$, we have that $J$ is in fact an ideal of $(A^+) \rtimes_\lambda G$ as well. To check that it still has trivial intersection with $A^+$ is straightforward:
        \[ J \cap A^+ = (J \cap A \rtimes_\lambda G) \cap A^+ = J \cap (A \rtimes_\lambda G \cap A^+) = J \cap A = 0. \]

        Conversely, assume that $K$ is now a nontrivial ideal of $(A^+) \rtimes_\lambda G$ with the property that $K \cap A^+ = 0$. By Lemma~\ref{lem:nonunital_essential}, the intersection $J = K \cap (A \rtimes_\lambda G)$ is still a nontrivial ideal of $A \rtimes_\lambda G$. Moreover, it is also clear that
        \[ J \cap A \subseteq K \cap A^+ = 0. \]
    \end{proof}

    It remains to check that our two main characterizations (C1) and (C2) listed in Section~\ref{sec:introduction} actually are equivalent when interpreted on $A$ and $A^+$. We will not repeat these characterizations again here, as they are quite cumbersome, but rather just remark the key parts in the proof below.

    \begin{proposition}
        \label{prop:nonunital_characterizations_equivalent}
        Consider from Section~\ref{sec:introduction} the two characterizations (C1) and (C2) for an action of a discrete group on a C*-algebra. Assume $A$ is a non-unital C*-algebra on which $G$ acts by *-automorphisms, and denote its unitization by $A^+$. We have that
        \begin{enumerate}
            \item The action $G \actson A$ has characterization (C1) if and only if $G \actson A^+$ does.
            \item The action $G \actson A$ has characterization (C2) if and only if $G \actson A^+$ does.
        \end{enumerate}
    \end{proposition}

    \begin{proof}
        First, it is clear that characterization (C1), which is phrased purely in terms of the injective envelope, is equivalent for $G \actson A$ and $G \actson A^+$, as $I(A) = I(A^+)$ by definition (recall from Definition~\ref{def:nonunital_injective_envelope}, for example).

        For characterization (C2), which is written in terms of ideals of the original C*-algebra, it is clear that if $J \triangleleft A$ is a nontrivial ideal satisfying the required properties (together with some $r \in \FC(G) \setminus \{e\}$, etc...), then it is also an ideal of $A^+$ with the same properties as well.
        
        Conversely, if $K \triangleleft A^+$ is a nontrivial ideal with the required properties, then by essentiality of $A$ in $A^+$ we have that $J = K \cap A$ is a nontrivial ideal of $A$. Note that we still have that $J \cap sJ$ is essential in $J$ for the appropriate group elements, as
        \[ J \cap sJ = (K \cap A) \cap (sK \cap A) = (K \cap sK) \cap A, \]
        and we leave it as an exercise to the reader
        that because $K \cap sK$ is essential in $K$, then $(K \cap sK) \cap A$ is essential in $K \cap A$. Moreover, by similar logic, we have that $J$ is essential in $K$, and so our unitary $u \in M(K)$ also lies in $M(J)$, as we canonically have $M(K) \subseteq M(J)$. The rest of the required properties are immediate.
    \end{proof}

    \section{Examples, applications, and special cases}

    This section collects several basic examples and special cases. Some of these were already known, sometimes with additional assumptions (for example, separability), but we believe that re-proving them via our machinery highlights the power of these results.

    \subsection{The case of simple C*-algebras}

    The conditions of the main theorems simplify considerably when working with actions on simple C*-algebras. This is due to the following well-known result relating inner actions on $I(A)$ to inner actions on $A$.

    \begin{proposition}
    \label{prop:simple_inner}
        Assume $A$ is a simple C*-algebra, not necessarily unital, and let $\alpha \in \Aut(A)$ be an automorphism of $A$. Denote the multiplier algebra of $A$ by $M(A)$. If the unique extension of $\alpha$ to $I(A)$ is given by $\Ad u$ for some unitary $u \in I(A)$, then we in fact have $u \in M(A)$.
    \end{proposition}

    \begin{proof}
        First, let us quickly recall that for non-unital C*-algebras $A$ with unitization $A^+$, we have the inclusions:
        \[ A \subseteq A^+ \subseteq M(A) \subseteq I(A) = I(A^+) \]
        It is known that for simple C*-algebras $A$, any automorphism is inner on $M(A)$ if and only if it is inner on $I(A)$ - see \cite[Corollary~7.8]{hamana85-injective_envelopes_equivariant}. Thus, $\alpha_r \in \Aut(A)$ satisfies $\alpha_r = \Ad v$ for some $v \in U(M(A))$, and by uniqueness of extensions of automorphisms to the injective envelope, we have that the extension $\wtilde{\alpha_r} \in \Aut(I(A))$ also satisfies $\wtilde{\alpha_r} = \Ad v$.
        
        Since $A$ is simple, we have that $I(A)$ is a factor. Indeed, if we were to have a nontrivial central projection $p$, the ideal $I(A)p \cap A^+$ would be a nonzero ideal of $A^+$, and thus by essentiality would have to intersect $A$ nontrivially as well.  But $A$ is simple, and so reinterpreting this statement, we have that $ap = a$ for all $a \in A$. In other words, the quotient map $I(A) \to I(A)p$ would have to be an injective *-homomorphism on $A^+$, violating essentiality of $A^+ \subseteq I(A)$.
        
        Thus, $\Ad u = \Ad v$ on $I(A)$ just says that $u = \gamma v$ for some scalar $\gamma \in \mathbb{T}$, which immediately implies that $u$ was an element of $M(A)$ to begin with.
    \end{proof}

    Now, Theorem~\ref{thmIntro:A} and Theorem~\ref{thmIntro:B} can be reinterpreted as the following two corollaries, respectively:

	\begin{corollary}
    \label{cor:ip_fc}
		Let $G$ be an FC-group acting on a simple C*-algebra $A$, not necessarily unital. Then $A \rtimes_\lambda G$ is not simple if and only if there exist an $r \in G \setminus \{e\}$ and a unitary $u \in U(M(A))$ such that
		\begin{enumerate}
			\item $r$ acts by $\Ad u$ on $A$;
			\item $s \cdot u=u$, for all $s \in C_G(r)$.
		\end{enumerate}
	\end{corollary}

    \begin{corollary}
    \label{cor:simple_fch}
		Let $G$ be an FC-hypercentral group acting on a unital simple C*-algebra $A$. Then $A \rtimes_\lambda G$ is not simple if and only if there exist an $r \in \FC(G) \setminus \{e\}$ and a unitary $u \in U(A)$ such that
		\begin{enumerate}
			\item $r$ acts by $\Ad u$ on $A$;
			\item $s \cdot u=u$, for all $s \in C_G(r)$.
		\end{enumerate}
	\end{corollary}

    \subsection{A finite-dimensional example}
    \label{sec:examples:finite_dimensional}

    The following example is taken from \cite[Example~5.6]{kennedy_schafhauser_noncommutative_crossed_products} (see also \cite[Example~5.1]{Ursu}).
	In \cite[Example~5.6]{kennedy_schafhauser_noncommutative_crossed_products}, it is given as an example for which the methods presented in \cite{kennedy_schafhauser_noncommutative_crossed_products} for determining simplicity of crossed products, do not apply.
	
	\begin{example}\label{eg}
		Consider the unitaries
        \[ U = \begin{bmatrix}
		0 & 1\\
		1 & 0
		\end{bmatrix} \quad \text{and} \quad V=\begin{bmatrix}
		1 & 0\\
		0& -1
		\end{bmatrix}, \]
        and from here let $\alpha \colon (\mathbb{Z}/2\mathbb{Z})^2 \to \Aut(M_2)$ be the action determined by the automorphisms:
		\[ \alpha_{(1,0)}= \Ad U \quad \text{and} \quad \alpha_{(0,1)}= \Ad V. \]
        Given any automorphism $\beta = \Ad W \in \Aut(M_2)$, the choice of unitary $W$ is unique up to scalar multiplication. Together with the fact that the unitaries $U$ and $V$ themselves do not commute, but instead satisfy $UV = -VU$, it is quite easy to check that no matter which $r \in (\Z/2\Z)^2$ and which unitary $W \in U(M_2)$ satisfying $\alpha_r = \Ad W$ we choose, the unitary $W$ does not satisfy $s \cdot W = W$ for all $s \in (\Z/2\Z)^2$. By Theorem~\ref{thmIntro:B}, using the fact that $I(A) = A$ in this case, the crossed product $M_2 \rtimes (\Z/2\Z)^2$ is simple. Indeed, it is also quite easy to verify by hand that it is a 16-dimensional C*-algebra with trivial center, and thus isomorphic to $M_4$.
	\end{example}

    \subsection{The case of cyclic groups}

    One of the first cases of interest in crossed products were actions by a single automorphism, and the corresponding crossed products of $A \rtimes_\lambda \Z$ or $A \rtimes_\lambda \Z/n\Z$. For example, \cite[Theorem~2.5]{olesen_pedersen_II} characterizes the intersection property of $A \rtimes_\lambda G$ for arbitrary abelian groups in terms of the Connes spectrum of the action, while \cite[Theorem~10.4]{olesen_pedersen_III} gives several equivalent conditions in the case of $\Z$-actions, at least in the separable setting.
    
    More recent results, such as \cite[Theorem~2.5]{kwasniewski_meyer_aperiodicity}, remove any separability assumptions in the minimal setting and are able to characterize simplicity entirely based on if the action of $\Z$ or (certain) $\Z/n\Z$ is outer on $M(A)$. However, in the non-minimal setting, the same paper characterizes the intersection property in \cite[Theorem~9.12]{kwasniewski_meyer_aperiodicity} (for Fell bundle C*-algebras in fact, via a Morita equivalence to a crossed product) only under some additional assumptions, such as containing a separable essential ideal. See also \cite{kwasniewski_meyer_2022} for a subsequent paper studying implications among various related properties for inclusions of $A$ into $A \rtimes_\lambda G$ and other crossed-product-like constructions, again with many converses still requiring additional mild assumptions. It is worth noting that in these two (and previous works), they mostly use \emph{aperiodicity} instead of proper outerness (a condition which is equivalent for automorphisms and group actions), and say that an inclusion \emph{detects ideals} rather than saying it has the intersection property (two synonyms for the same concept).

    Here, we show that our results can bypass any sort of separability (or similar) requirement. In other words, we show that for $\Z$ and (certain) $\Z/n\Z$ actions, the intersection property corresponds to proper outerness. See Theorem~\ref{thm:properly_outer_equivalence} and Definition~\ref{def:properly_outer} to recall the notion of an automorphism/group action being properly outer (in the sense of Kishimoto) on a C*-algebra $A$, which is equivalent to proper outerness of the unique extension to $I(A)$.

    Since our main characterization, Theorem~\ref{thmIntro:A}, requires an invariance condition on any unitary $u \in I(A)p$ realizing a non-properly outer automorphism $\alpha_g$ ($g \in G$), the following lemma can be found quite useful. Similar techniques have been used several times in the literature, for example, in the proof of \cite[Theorem~4.6]{olesen_pedersen_II}.

    \begin{lemma}
        \label{lem:single_automorphism_untwist}
        Assume $\alpha \in \Aut(B)$ is an automorphism on a monotone complete C*-algebra $B$. Assume moreover that there is some $n \in \N$ with the property that $\alpha^n$ is not properly outer on $B$. Let $p$ be the largest $\alpha^n$-invariant central projection on which $\alpha^n|_{Bp}$ is inner. Then we have:
        \begin{enumerate}
            \item $p$ is $\alpha$-invariant.
            \item $\alpha^{n^2}$ is also inner on $Bp$, and more importantly, we may choose a unitary $v \in Bp$ with the property that $\alpha^{n^2}|_{Bp} = \Ad v$ and $\alpha(v) = v$.
        \end{enumerate}
    \end{lemma}

    \begin{proof}
        First, we claim that $\alpha^n$ is also inner on $B\alpha(p)$. Let $u \in U(Bp)$ be a unitary implementing $\alpha^n$ on $Bp$, and observe that
        \[ \alpha(u) (x \alpha(p)) \alpha(u)^* = \alpha(u \alpha^{-1}(x) p u^*) = \alpha(\alpha^n(\alpha^{-1}(x) p)) = \alpha^n(x\alpha(p)). \]
        In other words, $\alpha^n$ is implemented by $\alpha(u)$ on $B \alpha(p)$. But $p$ was the largest $\alpha^n$-invariant central projection on which $\alpha^n$ was inner, and so $\alpha(p) \leq p$. Chaining together enough inequalities, we get:
        \[ p = \alpha^n(p) \leq \alpha^{n-1}(p) \leq \dots \leq \alpha(p) \leq p, \]
        which proves that $\alpha(p) = p$.

        There is no reason to believe that $\alpha(u) = u$, or that we may even perturb $u$ so that this is true. However, we may build a unitary $v \in U(Bp)$ implementing $\alpha^{n^2}$ such that $\alpha(v) = v$ as follows. Let
        \[ v = u \alpha(u) \dots \alpha^{n-1}(u). \]
        We have shown previously that if $\alpha^n = \Ad u$ on $Bp$, then $\alpha^n = \Ad \alpha(u)$ on $B \alpha(p) = Bp$. Consequently, we have $\alpha^n = \Ad \alpha^k(u)$ on $Bp$ for any $k \geq 0$. This actually shows that $\Ad v = \alpha^{n^2}$ on $Bp$.
        
        To see $\alpha$-invariance of $v$, note that if $\Ad w_1 = \Ad w_2$ on some C*-algebra, then $w_1$ and $w_2$ commute. In other words, all the terms $u, \alpha(u), \alpha^2(u), \ldots$,  in the expression of $v$ commute. Thus, we may simply evaluate
        \[ \alpha(v) = \alpha(u) \alpha^2(u) \dots \alpha^n(u), \]
        where the last term $\alpha^n(u)$ is just $u$ again, as $\alpha^n = \Ad u$ on $Bp$. In other words, $\alpha(v) = v$.
    \end{proof}

    \begin{corollary}
    \label{cor:cyclic}
        Assume $G = \Z$ or $G = \Z/n\Z$ with $n$ being square-free (i.e. no nontrivial perfect square divides $n$). Assume that $G$ acts on some C*-algebra $A$, not necessarily unital. Then $A \rtimes_\lambda G$ has the intersection property if and only if the action is properly outer.
    \end{corollary}

    \begin{proof}
        Again, the characterization of proper outerness in terms of $I(A)$ is the most convenient for our purposes. By Theorem~\ref{thmIntro:A}, if the action is properly outer, the crossed product has the intersection property. Conversely, if the action is not properly outer, there is some nonzero $k \in \Z$ or $k \in \Z/n\Z = \set{0, \dots, n-1}$ and a largest nonzero $k$-invariant central projection $p \in I(A)$ with the property that $k = \Ad u$ for some $u \in U(I(A)p)$. By Lemma~\ref{lem:single_automorphism_untwist}, this implies that $p$ is $\alpha$-invariant, and $k^2 = \Ad v$ for some other unitary $v \in U(I(A)p)$ with the additional property that $\alpha(v) = v$. Importantly, if we are dealing with $G = \Z$, then clearly $k^2 \neq 0$ as well. If we are dealing with $\Z/n\Z$ for a square-free $n$, then the same conclusion that $k^2 \neq 0$ also holds. In either case, by Theorem~\ref{thmIntro:A}, the crossed product does not have the intersection property.
    \end{proof}

    This extra requirement that $\Z/n\Z$ should be restricted to the square-free case is quite necessary (with the reason being evident when reading through the above proof). We now construct counterexamples in the non-square free case.

    \begin{proposition}
        Assume $n \in \N$ is \emph{not} square-free. That is, $n$ is divisible by some nontrivial perfect square. Then there is an action of $\Z/n\Z$ on some unital, separable, and simple C*-algebra $A$ that is not properly outer (at least one nontrivial group element acts by an inner automorphism), but with the crossed product $A \rtimes_\lambda \Z/n\Z$ still being simple.
    \end{proposition}

    \begin{proof}
        We actually argue as in \cite[Section~5.2]{Ursu}, where we use the fact that the automorphisms of the hyperfinite $II_1$ factor $R$ are quite well-understood, and provide a plentiful source of automorphisms that are strange enough for our purposes. Connes' work, specifically \cite[Proposition~1.6]{connes_hyperfinite_II1_automorphisms}, tells us the following. Let $p \in \N$, let $\gamma$ be a $p$-th root of unity, and viewing $R = \closure{\tensor}_{n=1}^\infty M_p$, define a unitary $u \in R$ as
        \[ u = \begin{bmatrix} \gamma & & & \\ & \gamma^2 & & \\ & & \ddots & \\ & & & \gamma^p \end{bmatrix} \tensor (\otimes_{n=2}^\infty I). \]
        There is an automorphism $\alpha \in \Aut(R)$ with the following properties:
        
        \begin{enumerate}
            \item $p$ is the smallest positive integer with $\alpha^p$ inner
            \item $\alpha^p = \Ad u$
            \item $\alpha(u) = \gamma u$
        \end{enumerate}

        In the context of our counterexample, write the non-square-free number as
        \[ n = p_1 \dots p_k q_1^{n_1} \dots q_l^{n_l}, \]
        where $n_i \geq 2$ for all $i$, and let
        \[ m = p_1 \dots p_k q_1^{n_1-1} \dots q_{l-1}^{n_l-1}. \]
        Now let $p = m$ and $\gamma$ be a primitive $n/m$-th root of unity (observe that $n/m$ divides $m$, and so $\gamma$ is also an $m$-th root of unity). Let our automorphism $\alpha \in \Aut(R)$ be as before. Furthermore, by construction, it is clear that $u^{n/m} = 1$, and so $\alpha^n = (\alpha^m)^{n/m} = \id$. In other words, we obtain a $\Z/n\Z$-action on $R$ via this automorphism.

        We first quickly note that, other than the fact that $R$ is not norm-separable, this action satisfies the remaining required properties. Indeed, $R$ is simple as a C*-algebra - this is an easy exercise using the Dixmier averaging property, which in this context says that
        \[ \tau(x) \in \closure{\conv}^{\norm{\cdot}} \setbuilder{uxu^*}{u \in R}, \]
        where $\tau(x)$ is the (unique) faithful normal tracial state. Furthermore, we are in a position to apply Theorem~\ref{thmIntro:A}. The only $k \in \Z/n\Z$ that admit an inner automorphism are $k = ml$ for $l \in \Z$, and they are always of the form $\Ad u^l$, where $\alpha(u^l) = \gamma^l u^l$. Hence, $u^l$ can only be invariant if $l$ is a multiple of $n/m$, or equivalently, if $k = 0$ in $\Z/n\Z$. Note that the choice of unitary implementing some $\alpha^{ml}$ is not truly unique, but because $Z(R) = \C$, it is unique up to scalar multiplication by a unimodular constant. Hence, the same result of $\alpha(v) \neq v$ holds regardless of what unitary $v$ we choose instead of $u^l$. By Theorem~\ref{thmIntro:A}, the crossed product $R \rtimes_\lambda \Z/n\Z$ must have the intersection property.

        Dropping to a separable C*-subalgebra with the same properties is not too difficult. For convenience, Blackadar's notion of \emph{separably inheritable} \cite[Definition~II.8.5.1]{blackadar_operator_algebras} will keep track of things more easily. A property (P) of C*-algebras is said to be separably inheritable if
        \begin{enumerate}
            \item If $A$ has property (P) and $B \subseteq A$ is a separable subalgebra, then there is an intermediate separable subalgebra $B \subseteq C \subseteq A$ with property (P).
            \item Property (P) is closed under countable inductive limits $A_1 \subseteq A_2 \subseteq \dots$, i.e.\ taking $A = \closure{\cup_n A_n}$.
        \end{enumerate}
        For our purposes, we observe that being simple (\cite[Theorem~II.8.5.6]{blackadar_operator_algebras}) and being $\alpha$-invariant (easy exercise) are separably inheritable properties for subalgebras of $R$. By taking inductive limits in the appropriate way, it is possible to show that ``being simple and $\alpha$-invariant'' (in general, the combination of finite or even countably many separably inheritable properties) as a single property is itself separably inheritable \cite[Proposition~II.8.5.3]{blackadar_operator_algebras}.

        Thus, if we start with a unital, norm-separable, weak*-dense C*-subalgebra of $R$ that contains $u$, then it is contained in such a subalgebra with the additional two properties of being simple and $\alpha$-invariant. Call it $A$. We claim that this is the subalgebra that we are looking for.
        
        Clearly, because $u \in A$, the action of $\Z/n\Z$ on $A$ is not properly outer. We claim that, just like for $R \rtimes_\lambda \Z/n\Z$, the crossed product $A \rtimes_\lambda \Z/n\Z$ also has the intersection property. Assume otherwise, and let $k \in \Z/n\Z$ be a nonzero element such that $k = \Ad v$ for some $v \in U(I(A))$ and $\alpha(v) = v$ (obtained from Theorem~\ref{thmIntro:A} again). By Proposition~\ref{prop:simple_inner}, because $A$ is simple, the unitary $v$ in fact lies in $A \subseteq R$. However, because $\alpha^k \in \Aut(R)$ is necessarily normal, and $A \subseteq R$ is weak*-dense, this just says that $\alpha^k = \Ad v$ on all of $R$. This contradicts what we observed before, namely that such an automorphism cannot have $\alpha(v) = v$. Consequently, $A \rtimes_\lambda \Z/n\Z$ must have the intersection property.
    \end{proof}

	\section*{Acknowledgements}
	
	We are grateful to Siegfried Echterhoff for pointing out, after the first public release of our preprint, that Theorem~\ref{thmIntro:B} follows by combining Theorem~\ref{thmIntro:C} with Echterhoff's own results in \cite{echterhoff_jot}. This has now been made clear. Similarly, we had also overlooked the fact that Lemma~\ref{lem:non_free_on_A_on_inj} had already been proven in a paper of Vrej Zarikian as \cite[Lemma~2.2]{Zarikian}, and Zarikian kindly pointed this out to us shortly after. We have updated the paragraph just before Lemma~\ref{lem:non_free_on_A_on_inj} accordingly. Finally, we would like to thank Bartosz Kwa{\'s}niewski for many fruitful discussions on potential applications of our results.
	
%
%
%
	
	\bibliographystyle{amsalpha}
	\bibliography{crossed_products_fch_groups}
\end{document}